\documentclass[12pt]{amsart}
\usepackage{graphicx} 
\usepackage{amssymb,amscd,amsthm,amsmath,color}
\usepackage{fullpage}
\usepackage{epsfig}
\usepackage{graphicx}
\usepackage{xypic}
\usepackage{url}
\usepackage{mathtools}
\usepackage{hyperref}
\usepackage{mathrsfs}
\usepackage{quiver}
\usepackage{comment}
\usepackage{enumitem}

\numberwithin{equation}{section}

\newcommand{\PP}{\mathbb{P}}
\newcommand{\OO}{\mathcal{O}}

\newcommand{\ev}{\operatorname{ev}}

\newcommand{\codim}{\operatorname{codim}}

\newcommand{\rk}{\operatorname{rk}}

\newcommand{\Eff}{\overline{\operatorname{Eff}}}
\newcommand{\Nef}{\operatorname{Nef}}
\newcommand{\Supp}{\operatorname{Supp}}
\newcommand{\Mor}{\operatorname{Mor}}

\newcommand{\mult}{\mathrm{mult}}

\newcommand{\elle}{e}

\newtheorem{theorem}{Theorem}[section]
\newtheorem{lemma}[theorem]{Lemma}
\newtheorem{proposition}[theorem]{Proposition}
\newtheorem{corollary}[theorem]{Corollary}

\theoremstyle{definition}
\newtheorem{definition}[theorem]{Definition}
\newtheorem{notation}[theorem]{Notation}

\newtheorem{remark}[theorem]{Remark}

\newtheorem{example}[theorem]{Example}

\newtheorem{claim}[theorem]{Claim}

\definecolor{ao(english)}{rgb}{0.0, 0.5, 0.0}

\newcommand{\exctowerheight}{\Theta}

\input xy
\xyoption{all}

\makeatletter
\def\dasharrowfill@#1#2#3#4{%
        $\m@th
        \thickmuskip0mu
        \medmuskip\thickmuskip
        \thinmuskip\thickmuskip
        \relax
        #4#1\mkern2mu
        \xleaders\hbox{$#4\mkern2mu#2\mkern2mu$}\hfill
        \mkern2mu
        #3$%
}

\def\dashrightarrowfill@{\dasharrowfill@\relbar\relbar\rightarrow}

\providecommand*\xdashrightarrow[2][]{%
  \ext@arrow 0055{\dashrightarrowfill@}{#1}{#2}}
\makeatother

\author{Eric Jovinelly}
\address{Department of Mathematics \\
Brown University \\
Box 1917 \\
151 Thayer Street \\
Providence, RI 02912}
\email{eric\_jovinelly@brown.edu}

\author{Brian Lehmann}
\address{Department of Mathematics \\
Boston College  \\
Chestnut Hill, MA \, \, 02467}
\email{lehmannb@bc.edu}

\author{Eric Riedl}
\address{Department of Mathematics \\
University of Notre Dame  \\
255 Hurley Hall \\
Notre Dame, IN 46556}
\email{eriedl@nd.edu}

\title{Finding large families of rational curves through bend-and-break}

\begin{document}

\begin{abstract}
    We present a new construction that allows us to break off large-degree rational curves from families of higher genus curves.  Our construction and results deepen the connection between rational curves and positivity of the anticanonical divisor.  Specifically, we show that varieties with large Fujita invariant admit large families of rational curves.  We also construct free rational curves on certain singular Fano varieties. 
    As an explicit consequence of our results, we prove that for a general Fano hypersurface of index at least 3, all spaces of genus $g$ curves of sufficiently large degree have the expected dimension.
\end{abstract}

\maketitle

\section{Introduction}

Mori's celebrated Bend-and-Break Theorem is a critical tool for understanding the birational geometry of projective varieties.  It says that a family of maps from an irrational curve that fixes the image of a point must limit to a map from a reducible curve with at least one rational component.  One challenge with applying this theorem is that classical versions do not give much information about the reducible curve.  For example, when the image of the fixed point is general then Bend-and-Break produces a rational curve that deforms in a dominant family but does not provide much information about its degree.

Our main result is a new breaking technique for dominant families of irrational curves on a projective variety $Y$.  When $Y$ is a smooth Fano, our technique produces a free rational curve whose cohomology class is very close to the cohomology class of the original curve (see Theorem \ref{theo:exctowerbreaking}).  For arbitrary $Y$, the correct analogue is to produce a rational curve that deforms as much as the original curve.  The following theorem achieves this over an algebraically closed field of characteristic 0.

\begin{theorem} \label{theo:introexctower}
    Let $Y$ be a projective variety and $C$ be a smooth projective curve.  Suppose that $M \subset \Mor(C,Y)$ is a subvariety such that $\dim M > \dim Y$ and $M$ parameterizes a dominant family of maps $s: C \to Y$.

    Then there is a subvariety $N \subset \Mor(\PP^1, Y)$ parameterizing a dominant family of maps $f: \mathbb{P}^{1} \to Y$ such that 
    $$\dim N \geq \dim M$$
    and $s_{*}C - f_{*}\mathbb{P}^{1}$ is algebraically equivalent to an effective $1$-cycle.
\end{theorem}

The key idea in the proof of Theorem \ref{theo:introexctower} is to break a family of curves that are highly tangent to a general divisor at a fixed general point of $Y$.  We will record the tangency information at this point by sequentially blowing-up $Y$ to obtain a ``tower'' of exceptional divisors on a birational model $Y_{k} \to Y$.  We carry out the breaking argument on $Y_{k}$ and carefully analyze how the constructed rational curves interact with the exceptional divisors to obtain the lower bound in Theorem \ref{theo:introexctower}. 

Theorem \ref{theo:introexctower} is useful both in explicit examples and for developing general theory.  The rest of the introduction is devoted to describing several applications of the construction used for our main result.  In particular, we address:
\begin{itemize}
    \item the existence of free rational curves in the smooth loci of certain singular Fano varieties,
    \item the close relationship between the dimension of families of curves and Fujita's invariant, and
    \item new results about curves on general Fano hypersurfaces.
\end{itemize}

\subsection{Existence of free rational curves} \label{sect:introfreecurves}
By combining our breaking results with the recent techniques of \cite{JLR25b} and \cite{JLR25}, we obtain new families of rational curves on mildly singular Fano varieties.  The following theorem is a simplified version of Theorem \ref{thm: nicest curves}.

\begin{theorem}\label{theo:introratcurves}
Let $X$ be a terminal Fano variety of dimension $n \geq 3$. 
For a general point $p \in X$ and for any integer $\Theta \geq 1$, there is a rational curve $s : \PP^1 \to X$ that passes through $p$ and satisfies
\begin{itemize}
\item $\Theta \leq - K_{X} \cdot s_* \PP^1 \leq \Theta + n$; and
\item $s(\PP^1)$ meets the singular locus of $X$ at most $n -2$ times.
\end{itemize}
\end{theorem}

A famous conjecture of \cite{KM99} predicts that a terminal Fano variety will contain free rational curves in its smooth locus.  Theorem \ref{theo:introratcurves} proves something a little weaker: it yields rational curves of arbitrarily large degree which meet the singular locus a bounded number of times.  We can use it to prove new cases of the conjecture of \cite{KM99}.

\begin{theorem} \label{theo:freeonterminal}
Let $X$ be a terminal Fano variety.  Suppose $X$ is either:
\begin{enumerate}
\item smoothable (as specified in Theorem \ref{theo:smoothableterminal}), or
\item a fourfold with LCIQ singularities.
\end{enumerate}
Then there is a free rational curve $s: \mathbb{P}^{1} \to X^{sm}$ in the smooth locus $X^{sm} \subset X$.
\end{theorem}

In many cases we can prove the existence of a very free rational curve in the smooth locus; for example, Theorem \ref{theo:terminalthreefoldveryfree} verifies this property for terminal Fano threefolds. 
Combined with the results of \cite{Tian12}, this yields:

\begin{corollary}[{\cite[Theorem 1.9]{Tian12}}]
    Suppose that $X$ is a terminal Fano threefold.  Then there is a resolution of singularities $\phi: X' \to X$ such that $X'$ is symplectic rationally connected.
\end{corollary}

Even when $X$ is a smooth Fano variety, Theorem \ref{theo:introexctower} yields new insight into the properties of rational curves on $X$.  The following result resolves an important open question in the study of Geometric Manin's Conjecture (see \cite[Remark 4.14]{LRT23}).

\begin{theorem}
Let $X$ be a smooth Fano variety.  Then every rational ray in the interior of $\Nef_{1}(X)$ is generated by the class of a very free rational curve.
\end{theorem}

\subsection{Curves and Fujita invariants} \label{sect:introfujinv}
Next, we use Theorem \ref{theo:introexctower} to strengthen the connections between an invariant of the minimal model program and the existence of curves.  The key notion is the Fujita invariant.

\begin{definition}
\label{defi:introa-invariant}
Let $Y$ be a smooth projective variety and let $L$ be a big and nef $\mathbb{Q}$-Cartier divisor on $Y$.  The Fujita invariant of $(Y,L)$ is
\begin{equation*}
a(Y,L) = \min \{ t\in \mathbb{R} \mid  K_Y + tL \textrm{ is pseudo-effective }\}.
\end{equation*}
\end{definition}

The Fujita invariant plays an important role in various areas of birational geometry, including Fujita's classification theory (see e.g.~\cite{Fujita89}, \cite{Horing10}) and the effective study of Iitaka fibrations (see e.g.~\cite{Dicerbo14}, \cite{BZ16}).  It also is important in arithmetic geometry due to its close connection with the geometry of dominant families of curves on varieties (see e.g.~\cite{HTT15}, \cite{LT19}, \cite{LRT23}).  Our focus is the latter topic: we complete the story begun in earlier work by showing that varieties with large Fujita invariants carry large families of curves.

Prior work, which we summarize in Proposition \ref{prop:fujupperbound}, shows that if $s: C \to Y$ is a member of an irreducible component $M \subset \Mor(C,Y)$ parameterizing a dominant family of genus $g$ curves on $Y$, then $\dim M \leq a(Y,L) s_*C \cdot L + 2g + \dim Y$. We prove a type of converse statement: there will exist infinitely many families $M$ so that the inequality above is almost satisfied.

\begin{theorem} \label{theo:introainv}
Let $Y$ be a smooth uniruled projective variety, let $L$ be a big and nef Cartier divisor on $Y$, and let $C$ be a smooth curve of genus $g$.  There are infinitely many irreducible components $N \subset \Mor(C,Y)$ parameterizing dominant families of curves $s: C \to Y$ with
\begin{equation*}
\dim N > a(Y,L)(L \cdot s_{*}C) - 2g\dim Y.
\end{equation*}
\end{theorem}

In particular, the case $C = \PP^1$ shows one can compute the Fujita invariant of a variety from the properties of the rational curves that it carries, answering an open question of \cite[Section 6.4]{LT19}.  For stronger versions of Theorem \ref{theo:introainv}, see Proposition \ref{prop:orthogonaltoadjoint} and Theorem \ref{theo:ainvandcurves}.

Our results are particularly useful for studying curves on a Fano variety $X$.  In combination with \cite{LRT24}, Theorem \ref{theo:introainv} shows that the subvarieties $Y \subset X$ with Fujita invariant greater than $1$ are the ``accumulating subvarieties'' which admit curves that deform much more than expected.  In particular, by leveraging the dictionary between Fujita invariants and curves, we can relate the geometric properties of curves on $X$ of different genera.

\begin{theorem}\label{thm:introequivalentaconditions}
For a smooth Fano variety $X$, the following are equivalent.
\begin{enumerate}
\item For every closed subvariety $Y \subsetneq X$, $a(Y,-K_X|_Y) \leq 1$.
\item There exists a smooth projective curve $C$ such that the dimension of every irreducible component of $\Mor(C,X)$ is at most $g(C)\dim X$ more than the expected dimension.
\item For every smooth projective curve $C$ and every irreducible component $M \subset \Mor(C,X)$ the dimension of $M$ is at most $g(C)\dim X$ more than the expected dimension.
\end{enumerate}
\end{theorem}

Theorem \ref{thm:equivalentaconditions} proves a related statement for Fano varieties for which every closed subvariety $Y \subsetneq X$ has $a(Y,-K_{X}|_{Y}) < 1$.

\subsection{Hypersurfaces} \label{sect:introhypersurface}

Finally, we demonstrate how Theorem \ref{theo:introexctower} can be applied to examples by studying curves on general Fano hypersurfaces.  These moduli spaces have been studied extensively in their own right as well as for their applications to Gromov-Witten theory.  Recently there has been a renewed interest due to connections with arithmetic geometry via the geometric circle method. 

Our key result shows that there are no ``accumulating subvarieties'' for general Fano hypersurfaces in the largest possible degree range.  It strengthens a recent result of \cite{HaseLiu25} which established the same property in a more limited degree range by leveraging techniques from arithmetic geometry.

 \begin{theorem}\label{theorem:hypersurfaces_a_invariant}
 Let $X \subset \PP^n$ be a general hypersurface of degree at most $n-2$.  Then for every closed subvariety $Y \subsetneq X$ we have $a(Y,-K_{X}|_{Y}) < 1$.
 \end{theorem}

Theorem \ref{theorem:hypersurfaces_a_invariant} implies several  strong structural results about curves on general Fano hypersurfaces.  First, by combining it with Theorem \ref{thm:introequivalentaconditions} and the results of \cite{RY19} we obtain a dimension bound.

\begin{theorem} \label{theo:introhypcomparison}
Let $X \subset \mathbb{P}^{n}$ be a general Fano hypersurface of degree at most $n-2$.  Then for every curve $s: C \to X$ we have
\begin{equation*}
\dim_{[s]}\Mor(C,X) \leq -K_{X} \cdot s_{*}C + n-1.
\end{equation*}
\end{theorem}

Second, Theorem \ref{theorem:hypersurfaces_a_invariant} can be used to show that most irreducible components of $\Mor(C,X)$ parameterize free curves.  Recall that a morphism $s: C \to X$ is said to be free if the minimal slope of a non-zero quotient of $s^{*}T_{X}$ is at least $2g(C)$.  The following is a special case of Theorem \ref{theo:strongestainvcurves} applied to hypersurfaces using Theorem \ref{theorem:hypersurfaces_a_invariant}.

\begin{corollary} \label{coro:besthypersurfaces}
     Let $X \subset \PP^n$ be a general hypersurface of degree at most $n-2$.  There is a quadratic function $d(g)$ such that, for every non-negative integer $g$, if $s: C \to X$ is a genus $g$ curve of degree $\geq d(g)$ then a general deformation of $s$ is free.
\end{corollary}

\bigskip

\noindent \textbf{Acknowledgements:}
We thank Izzet Coskun, Matthew Hase-Liu, Brendan Hassett, James M\textsuperscript{c}Kernan, Joaqu\'in Moraga, and Sho Tanimoto for helpful conversations.  We would particularly like to thank Eric Larson for providing Example \ref{example: Eric L's family of conics}.

Eric Jovinelly was supported by an NSF postdoctoral research fellowship, DMS-2303335. Brian Lehmann was supported by Simons Foundation grant Award Number 851129.  Eric Riedl was supported by NSF CAREER grant DMS-1945944 and Simons Foundation grants 00011850 and 00013673.

\section{Preliminaries}

Throughout we work over a ground field $\mathbb{K}$ that is algebraically closed of characteristic $0$.  We will only work with schemes which are separated and whose connected components have finite type over the ground field.  A variety is irreducible and reduced.  A pair $(X,\Delta)$ will consist of a normal variety $X$ and an effective $\mathbb{Q}$-Weil divisor $\Delta$ such that $K_{X} + \Delta$ is $\mathbb{Q}$-Cartier.  We say the pair is reduced if every irreducible component of $\Delta$ occurs with coefficient $1$.

For a projective variety $X$, we let $N^{1}(X)_{\mathbb{R}}$ denote the space of numerical equivalence classes of $\mathbb{R}$-Cartier divisors and let $N_{1}(X)_{\mathbb{R}}$ denote the dual space of numerical equivalence classes of $\mathbb{R}$-$1$-cycles.  We denote by $\Eff^{1}(X)$ the pseudo-effective cone of divisors and by $\Nef_{1}(X)$ the dual cone of nef curves.

Given a smooth projective curve $C$ and a projective variety $X$, we denote by $\Mor(C,X)$ the moduli space of morphisms from $C$ to $X$. We let $\overline{\mathcal{M}}_{g,n}(X)$ denote the Kontsevich stack of $n$-pointed genus $g$ stable maps and $\mathcal{M}_{g,n}(X)$ denote the open substack of maps with smooth irreducible domain.  If we want to specify the numerical class $\alpha \in N_{1}(X)_{\mathbb{R}}$ of the curves, we write $\Mor(C,X,\alpha)$ or $\overline{\mathcal{M}}_{g,n}(X,\alpha)$.

\subsection{Positive curves}

\begin{definition}
    Let $X$ be a quasiprojective variety and let $C$ be a smooth projective curve.  Fix a non-negative $r \in \mathbb{R}$.  A nonconstant morphism $s: C \to X$ is said to be $r$-free if $s(C)$ lies in the smooth locus of $X$ and every non-zero quotient of $s^{*}T_{X}$ has slope at least $2g(C) + r$.
\end{definition}

When we say a curve is $r$-free, we are implicitly assuming $r \geq 0$. In some cases, it will be useful to work with curves satisfying a weaker positivity condition intended to describe curves that are $r$-free in fibers of morphisms.

\begin{definition} \label{defi:almostfree}
Let $X$ be a projective variety and let $r$ be a non-negative number.  We say that a curve $s: C \to X$ is almost $r$-free if there is a smooth open subset $U \subset X$, a smooth morphism $\pi: U \to Z$, and a fiber $F$ of $\pi$ such that $s$ has image in $F$ and is an $r$-free curve in $F$.
\end{definition}

\section{Exceptional Towers}\label{sec: exceptional towers}
In this section we present our main breaking technique, Theorem \ref{theo:exctowerbreaking}.  Theorem \ref{theo:introexctower} is an immediate consequence: apply Theorem \ref{theo:exctowerbreaking} to an irreducible component of $\Mor(C,Y)$ containing the subvariety $M$ from Theorem \ref{theo:introexctower}.  

\begin{theorem}\label{theo:exctowerbreaking}
    Let $Y$ be a projective variety and $C$ be a smooth projective curve. 
    Suppose $M \subset \Mor(C,Y,\alpha)$ is an irreducible component parameterizing a dominant family of maps.  If $\dim M > \dim Y$, then there is a component $N \subset \Mor(\mathbb{P}^1, Y,\beta)$ such that:
    \begin{enumerate}
        \item $\dim N \geq \dim M$;
        \item $\alpha - \beta$ is algebraically equivalent to an effective $1$-cycle; and
        \item $N$ parameterizes a dominant family of maps.
    \end{enumerate}
    If $Y$ is smooth, then we can also ensure that $$-K_{Y} \cdot \beta \geq -K_{Y} \cdot \alpha - g \dim Y.$$
    In particular, if $Y$ is smooth and Fano then $\alpha-\beta$ is contained in a bounded subset of $N_{1}(Y)_{\mathbb{R}}$ that does not depend on $\alpha$ or $\beta$.
\end{theorem}

In fact, the proof will show that maps in $M$ degenerate to a stable map $s: C' \to Y$  such that the general map in $N$ is a smoothing of the restriction of $s$ to some connected subcurve of $C'$.  The following definition identifies the main construction in our argument.  This construction will allow us to observe the behavior of curves very closely at a smooth point in $Y$.

\begin{definition}\label{def:exceptional tower}
Let $Y$ be a projective variety.  An \textit{exceptional tower} over $Y$ is a sequence 
$$Y_k \xrightarrow{\phi_k} Y_{k-1} \xrightarrow{\phi_{k-1}} \ldots \xrightarrow{\phi_1} Y_0 \xrightarrow{\phi_0} Y$$
of blow-ups $\phi_i$ whose centers $S_{i-1}$ and exceptional divisors $E_i$ satisfy the following conditions:
\begin{itemize}
\item The map $\phi_0 : Y_0 \rightarrow Y$ is the blow-up of a smooth point $p \in Y$ with exceptional divisor $E_{0}$; 
\item The map $\phi_1 : Y_1 \rightarrow Y_0$ is the blow-up of a hyperplane $S_0 \subset E_0 \cong \PP^{\dim Y - 1}$ with exceptional divisor $E_{1}$;
\item For $i > 1$, suppose we have constructed $S_{i-2} \cong \mathbb{P}^{\dim Y-2}$ and the $(i-1)$st exceptional divisor $E_{i-1}$ is isomorphic to $\mathbb{P}_{S_{i-2}}(\mathcal{O} \oplus \mathcal{O}(i))$.  Then the center $S_{i-1} \subset Y_{i-1}$ of the blow-up $\phi_i : Y_i \rightarrow Y_{i-1}$ is a minimal moving section of the $\PP^1$-bundle $\phi_{i-1}|_{E_{i-1}} \colon E_{i-1} \to S_{i-2}$.  In particular $S_{i-1} \cong \mathbb{P}^{\dim Y-2}$ and $E_{i} \cong \mathbb{P}_{S_{i-1}}(\mathcal{O} \oplus \mathcal{O}(i+1))$.
\end{itemize}
We call $p$ the base and $k$ the height of this exceptional tower. We use the notation $\Phi(Y,p,k)$ to denote such an exceptional tower on $Y$ (for some unspecified choice of blow-ups $\phi_{1},\ldots,\phi_{k}$).
\end{definition}

\begin{notation} \label{nota:exctower}
Let $(Y,\Delta)$ be a pair with reduced boundary; if no boundary $\Delta$ is specified we take $\Delta = \emptyset$. 
Given a point $p$ in the smooth locus of $Y \setminus \Delta$ and an exceptional tower $\Phi(Y,p,k)$ as in Definition \ref{def:exceptional tower}:
\begin{itemize}
\item for $0 \leq i < j \leq k$, we denote by $\widetilde{E}_{i}$ the strict transform on $Y_{j}$ of the exceptional divisor $E_{i}$ for $\phi_{i}: Y_{i} \to Y_{i-1}$ and we also set $\widetilde{E}_k = E_k$ for convenience; 
\item for $0 \leq i \leq k$, we let $\Delta_i \subset Y_i$ be 
$\big(\phi_1 \circ \ldots \circ \phi_i\big)^* (E_0 + \phi_0^*\Delta)$;
\item for $1 \leq i \leq k$, we denote by $\ell_{i}$ the class of the strict transform on $Y_{k}$ of a general line in a fiber of the projective bundle $\phi_{i}|_{E_{i}}: E_{i} \to S_{i-1}$ and we denote by $\ell_{0}$ the class of the strict transform of a general line in $E_{0}$; 
\item given a smooth projective curve $C$ and a nonconstant map $s : C \to Y$,  for each $0 \leq i \leq k$ we let $s_i : C \to Y_i$ denote the strict transform of $s$ mapping to $Y_i$.
\end{itemize}
\end{notation}

A straightforward computation yields:

\begin{lemma} \label{lemm:intersectionvalues}
Let $Y$ be a projective variety.  Consider an exceptional tower over $Y$ as in Definition \ref{def:exceptional tower}.  Then for $0 < i < k$ we have
\begin{equation*}
\widetilde{E}_{i} \cdot \ell_{i} = -2 \qquad \qquad \widetilde{E}_{i} \cdot \ell_{i-1} = \widetilde{E}_{i} \cdot \ell_{i+1} = 1
\end{equation*}
and $\widetilde{E}_{i} \cdot \ell_{j} = 0$ if $|i-j| \geq 2$.  If $k > 0$, we also have
\begin{align*}
E_{k} \cdot \ell_{k} = -1 \qquad E_{k} \cdot \ell_{k-1} = 1 \qquad \widetilde{E}_{0} \cdot \ell_{0} = -2 \qquad \widetilde{E}_{0} \cdot \ell_{1} = 1.
\end{align*}
\end{lemma}

The following provides geometric intuition for exceptional towers. For any exceptional tower $\Phi(Y,p,k)$, there is a very ample divisor $A \subset Y$ such that $p \in A^{sm}$ and for all $i < k$ the strict transform of $A$ in $Y_i$ intersects $E_i$ along $S_i$.  
Curves in $Y$ that are $k$-fold tangent to $A$ at $p$ are precisely those whose strict transforms in $Y_k$ meet $E_k$.  The language of exceptional towers allows us to study specializations of these curves as pointed stable maps.

\subsection{Local multiplicities}
We next define and describe a notion of multiplicity for maps that is local on the source rather than the target.  This will be helpful for describing how curves with a marked point interact with exceptional towers.

\begin{definition}\label{def:local multiplicity}
Let $Y$ be a variety, $C$ be a smooth curve, and $s \colon C \rightarrow Y$ be a nonconstant map.  Let $Z \subset Y$ be a smooth closed subscheme contained in the smooth locus of $Y$ such that $s(C) \not\subseteq Z$.  For any point $q \in C$, the \textit{local intersection multiplicity} $\mathrm{mult}_q(s,Z)$ of $Z$ with $s$ at $q$ is the maximal $r\geq 0$ such that $(s^{-1}\mathcal{I}_Z \cdot \OO_C)_{\mathfrak{m}_q} \subset (\mathfrak{m}_q^{r})_{\mathfrak{m}_{q}}$ where the subscript $(\ \ )_{\mathfrak{m}_{q}}$ denotes localization.
\end{definition}

We analyze deformations of maps that preserve certain local multiplicities.  The following sheaf records the infinitesimal structure of such deformations.

\begin{definition} \label{defi:locallogtangentsheaf}
Let $(Y,\Delta)$ be a log smooth pair with reduced boundary.   
Let $C$ be a smooth projective curve and let $\mathcal{Q} \subset C$ be a finite set of points.
Suppose that $s : C \to Y$ is a morphism such that $s(C) \not \subset \Delta$.
Consider the standard exact sequence of coherent sheaves on $Y$
\begin{equation*}
0 \to T_{Y}(- \log \Delta) \to T_{Y} \to N_{\Delta}.
\end{equation*} 
We define the \textit{local log tangent sheaf} $s^{*}T_{Y}(-\log_{(s,\mathcal{Q})} \Delta)$ to be the kernel of the natural map
\begin{equation*}
s^{*}T_{Y} \to (s^{*}N_{\Delta})_{\mathcal{Q}}
\end{equation*}
where for a torsion sheaf $\mathcal{T}$ on $C$, $\mathcal{T}_{\mathcal{Q}}$ denotes the part of $\mathcal{T}$ supported along $\mathcal{Q}$.
\end{definition}

    With notation as in Definition \ref{defi:locallogtangentsheaf},  the image of the natural map $T_Y \to N_{\Delta}$ is the direct sum $\oplus_{i=1}^{r} N_{\Delta_{i}}$ of the normal sheaves of the irreducible components $\{ \Delta_{i} \}_{i=1}^{r}$ of $\Delta$.  Therefore global sections of $s^{*}T_{Y}(-\log_{(s,\mathcal{Q})} \Delta)$ correspond to infinitesimal deformations of $s$ that preserve the local intersection multiplicities of $\Delta$ with $s$ along $\mathcal{Q}$.

We will need to understand how local log tangent sheaves interact with the blow-ups defining exceptional towers.

\begin{remark}\label{rem: cokernel}
Let $\pi : Y \to X$ be the blow-up of a smooth subvariety $Z \subset X$.  Let $E$ be the exceptional divisor of $\pi$ and consider $E$ as the relative Grassmannian $\pi|_{E} : E  \to Z$ of lines in the vector bundle $N_Z$.  The universal quotient bundle on $E$ is isomorphic to $T_{E / Z}(E)$.  The cokernel of $d\pi : T_Y \to \pi^* T_X$ may be identified with the composition of the standard map $\pi^*T_X \to \pi^* N_Z$ and the tautological quotient-bundle map $\pi^* N_Z \to T_{E / Z}(E)$.
\end{remark}

\begin{lemma} \label{lemm:loctancomputation}
Let $(Y,\Delta)$ be a log smooth pair with reduced boundary.  Let $p \in Y\setminus \Delta$ be a point.  Let $C$ be a smooth projective curve and $\mathcal{Q} \subset C$ be a finite set of points.  Fix an exceptional tower $\Phi(Y,p,k)$ with Notation \ref{nota:exctower}.  Consider a nonconstant map $s : C \to Y$ with $s(C) \not \subset \Delta$ and let $\mathcal{P}$ be the set of points in $s^{-1}(p)$ not contained in $\mathcal{Q}$.
\begin{enumerate}
\item There is an exact sequence
$$0 \to s_{0}^*T_{Y_{0}}(-\log_{(s_{0},\mathcal{Q})} \Delta_{0}) \to s^*T_Y(-\log_{(s,\mathcal{Q})} \Delta)\otimes \mathcal{O}\big (-\sum_{q \in \mathcal{Q}} \mathrm{mult}_q(s,p)\cdot q\big) \to \big (s^*T_{E_0}(E_0)\big )_{\mathcal{P}} \to 0.$$
\item For each integer $0 \leq i < k$, 
we have an exact sequence
\begin{equation*}
\begin{split}
0 \to s_{i+1}^*T_{Y_{i+1}}(-\log_{(s_{i+1},\mathcal{Q})} \Delta_{i+1}) & \to  s_{i}^*T_{Y_{i}}(-\log_{(s_{i},\mathcal{Q})} \Delta_{i}) \\
& \xrightarrow{\tau} \big (s_{i+1}^{*}T_{E_{i+1}/S_{i}}(E_{i+1} - \widetilde{E}_i) \big )_{\mathcal{Q}} \oplus \big (s_{i+1}^{*}T_{E_{i+1}/S_{i}}(E_{i+1}) \big )_{\mathcal{P}} \to 0.
\end{split}
\end{equation*}
\end{enumerate}
\end{lemma}

\begin{proof}
Since the desired exact sequences can be checked on the preimage of an open neighborhood of $p$ in $Y$, we may replace $Y$ by $Y \backslash \Delta$ to assume that $\Delta = 0$.

(1) We claim that the log tangent bundle $T_{Y_{0}}(-\log E_{0})$ may be described as the kernel of the standard map $\varphi : \phi_0^*T_Y \to \phi_0^* N_p$.  Consider the commuting diagram
\begin{equation*}
\xymatrix{
0 \ar[r] & \mathcal{F} \ar[r] \ar[d]_{f} & \phi_0^{*}T_{Y} \ar[r]^{\varphi} \ar[d]_{=} & \phi_0^{*}N_{p} \ar[r] \ar[d]_{g} & 0 \\
0 \ar[r] & T_{Y_0} \ar[r]^{d\phi_0} & \phi_0^{*}T_{Y} \ar[r] & T_{E_0/p}(E_0) \ar[r] & 0
}
\end{equation*}
where the bottom row is the standard exact sequence for a smooth blow-up, $\mathcal{F}$ is defined as the kernel of $\varphi$, and the morphism $g: \phi_0^{*}N_{p} \to T_{E_0/p}(E_0)$ is the tautological quotient-bundle map as in Remark \ref{rem: cokernel}.  

The kernel of $g$ is the tautological subbundle of $\phi_0^{*}N_{p}$, namely $\mathcal{O}_{E_0/p}(-1) = N_{E_{0}}$.  By the snake lemma, we see that the cokernel of $f$ is also isomorphic to $N_{E_{0}}$ and that the map $T_{Y_0} \to N_{E_{0}}$ is induced by $\varphi \circ d\phi_0$ and thus is the natural quotient.
We conclude that $\mathcal{F}$ is isomorphic to $T_{Y_{0}}(-\log E_{0})$.  This proves our claim.

Note that the entire diagram above remains exact upon pulling back via $s_{0}$.  Observe that the local log tangent sheaf combines the behavior of the restricted tangent sheaf $s_0^*T_{Y_{0}}$ near $\mathcal{P}$ and the log tangent sheaf $s_0^*T_{Y_{0}}(-\log E_{0})$ near $\mathcal{Q}$.  Transitioning from the pullback of the log tangent sheaf to the local log tangent sheaf, we obtain a diagram
\begin{equation*}
\xymatrix{
0 \ar[r] & s_0^*T_{Y_0}(-\log_{(s_0, \mathcal{Q})} E_0) \ar[r] \ar[d] & s^{*}T_{Y} \ar[r] \ar[d]_{=} & (s^{*}N_{p})_{\mathcal{Q}} \oplus (s^{*}T_{E_0/p}(E_0))_{\mathcal{P}} \ar[r] \ar[d] & 0 \\
0 \ar[r] & s_{0}^{*}T_{Y_0} \ar[r] & s^{*}T_{Y} \ar[r] & s^{*}T_{E_0/p}(E_0) \ar[r] & 0
}
\end{equation*}
Furthermore, since $N_p = T_Y \otimes_{\OO_Y} (\OO_Y / \mathfrak{m}_p)$ the kernel of $s^{*}T_{Y} \to (s^{*}N_{p})_{\mathcal{Q}}$ can be identified with the twist $s^*T_Y(-\sum_{q \in \mathcal{Q}} \mathrm{mult}_q(s,p)\cdot q)$.    
Finally, because $p \notin \Delta$, the intersection of $s_0^*T_{Y_0}(-\log_{(s_0, \mathcal{Q})} E_0)$ with $s^*T_Y(-\log_{(s, \mathcal{Q})} \Delta)$ is the desired sheaf $s_0^*T_{Y_0}(-\log_{(s_0, \mathcal{Q})} \Delta_0)$.  This proves (1).

(2) We have a commuting diagram with exact rows

\begin{equation*}
\xymatrix{
0 \ar[r] & T_{Y_{i+1}}(-\log \Delta_{i+1}) \ar[r] \ar[d]_{d_{\log} \phi_{i+1}} & T_{Y_{i+1}} \ar[r] \ar[d]_{d\phi_{i+1}} & N_{E_{i+1}} \oplus N_{\widetilde{E}_{i}} \oplus_{j = 0}^{i-1} N_{\widetilde{E}_j} \ar[r] \ar[d]_{\iota} & 0 \\
0 \ar[r] & \phi_{i+1}^{*}T_{Y_{i}}(-\log \Delta_{i}) \ar[r]  & \phi_{i+1}^{*}T_{Y_{i}} \ar[r] & \phi_{i+1}^{*}N_{E_i} \oplus_{j = 0}^{i-1} N_{\widetilde{E}_j} \ar[r] & 0 
}
\end{equation*}
All three vertical maps are injective, and the cokernel of $\iota$ is naturally the restriction $\phi_{i+1}^*N_{E_i}|_{E_{i+1} \cap \widetilde{E}_i}$.  Letting $K$ denote the cokernel of $d_{\log} \phi_{i+1}$, we obtain an exact sequence of cokernels 
\begin{equation*}
0 \to K \to T_{E_{i+1}/S_{i}}(E_{i+1}) \to \phi_{i+1}^*N_{E_i}|_{E_{i+1} \cap \widetilde{E}_i} \to 0.
\end{equation*}
The identification of $T_{E_{i+1}/S_{i}}(E_{i+1})$ with the tautological quotient bundle of $\phi_{i+1}^* N_{S_i}$ shows  $T_{E_{i+1}/S_{i}}(E_{i+1})|_{E_{i+1} \cap \widetilde{E}_i} = \phi_{i+1}^*N_{E_i}|_{E_{i+1} \cap \widetilde{E}_i}$.   Thus, $K \cong T_{E_{i+1}/S_{i}}(E_{i+1} - \widetilde{E_i})$.  Pulling back via $s_{i+1}$, we may glue the restriction of $s_{i+1}^*T_{Y_{i+1}}(-\log \Delta_{i+1})$ to $C \setminus \mathcal{P}$ together with the restriction of $s_{i+1}^*T_{Y_{i+1}}$ to $C \setminus \mathcal{Q}$ using the natural isomorphism over $C \setminus (\mathcal{P} \cup \mathcal{Q})$.  The resulting sheaf is the local log tangent sheaf $s_{i+1}^*T_{Y_{i+1}}(-\log_{(s_{i+1},\mathcal{Q})}\Delta_{i+1})$.  
\end{proof}

\subsection{Curves on general exceptional towers}

    We will say that an exceptional tower $\Phi(Y,p,k)$ is \textit{general} among those based at $p$ if the choice of each $S_i \subset E_i$ lies in an open subset of $|\mathcal{O}_{E_i}(S_i)|$ that depends on previously specified quantities.  For our purposes, these quantities will be
    a smooth projective curve $C$, a finite set of points $\mathcal{Q} \subset C$, and a quasi-projective subvariety $M \subset \Mor(C,Y)$ as discussed in 
    Lemmas \ref{lemm:meaningofvg} and \ref{lem: climb tower morphism scheme}.

\begin{lemma}\label{lemm:meaningofvg}
Let $(Y,\Delta)$ be a log smooth pair with reduced boundary.  Let $p \in Y\setminus \Delta$ be a closed point. 
Fix a smooth projective curve $C$ and a finite set of points $\mathcal{Q} \subset C$. 
Fix a quasi-projective subvariety $M \subset \Mor(C,Y)$.
Then a general exceptional tower $\Phi(Y,p,k)$ based at $p$ has the following property.

Suppose $s : C \to Y$ is a nonconstant map parameterized by $M$ whose image is not contained in $\Delta$.  Let $i < k$ be a nonnegative integer.  Let $T_i \subset H^0\big(C, s_i^{*}T_{Y_i}(-\log_{(s_i,\mathcal{Q})} \Delta_i)\big)$ be the preimage of 
$T_s M \subset H^0(C, s^{*}T_{Y})$ under the natural inclusion map.  Then there is a global section $\epsilon \in T_i$ such that the composition
\begin{equation*}
\mathcal{O}_{C} \xrightarrow{\epsilon} s_i^{*}T_{Y_i}(-\log_{(s_i,\mathcal{Q})} \Delta_i) \xrightarrow{\rho} \big (s_{i+1}^{*}T_{E_{i+1}/S_{i}}(E_{i+1} - \widetilde{E}_i)\big )_{\mathcal{Q}}
\end{equation*}
is surjective (where the rightmost map $\rho$ is induced by the map $\tau$ in Lemma \ref{lemm:loctancomputation}).
\end{lemma}

\begin{proof}
The strict transforms of maps in $M$ to $Y_{i}$ are parameterized by finitely many subvarieties of $\Mor(C, Y_i)$. Each subvariety is stratified by finitely many locally closed subsets wherein the intersection number $\mult_{q_j}(s_i,D)$ is fixed for each $q_j \in \mathcal{Q}$ and each irreducible component $D \subset \Delta_i$. 
Let $M_i \subset \Mor(C, Y_i)$ be the stratum containing $s_i$.  The tangent space to $M_i$ at $s_i$ corresponds to $T_i$.

Let $q_j \in \mathcal{Q}$.  We first find a section $\epsilon_j\in T_i$ such that $\rho \circ \epsilon_j : \mathcal{O}_{C} \to (s_{i+1}^{*}T_{E_{i+1}/S_{i}}(E_{i+1} - \widetilde{E}_i))|_{q_j}$ is surjective.  If $s_{i+1}(q_j) \notin E_{i+1}$, then $(s_{i+1}^{*}T_{E_{i+1}/S_{i}}(E_{i+1} - \widetilde{E}_i))|_{q_j} = 0$ and so we set $\epsilon_j = 0$.  Otherwise, consider the map $\text{ev}_{q_j} : M_i \to E_i$ given by $[h : C \to Y_i] \mapsto h(q_j)$.  We may stratify $M_i$ by locally closed subsets $N_1, \ldots , N_t \subset M_i$ such that for all $r \leq t$ the subset $N_r$ is smooth and the map $d \text{ev}_{q_j} : T_{N_r} \to \text{ev}_{q_j}^* T_{E_i}$ has constant rank.  Since there are finitely many choices for $M_i$ and $N_r$, by Bertini's Theorem $S_i \subset E_i$ will be transverse to each $\text{ev}_{q_j}(N_r) \subset E_i$ for a general choice of $\Phi(Y,p,k)$.

Let $N_r$ be the stratum of $M_i$ containing $s_i$.  The assumption that $s_{i+1}$ maps $q_j$ into $E_{i+1}$ implies $p_j : = s_{i}(q_j)$ lies in $S_i \cap \text{ev}_{q_j}(N_r)$.  However, $S_i$ and $\text{ev}_{q_j}(N_r)$ are transverse subvarieties of $E_i$ and by construction $d \text{ev}_{q_j}|_{s_i} : T_{s_i} N_r \to T_{p_j} \text{ev}_{q_j}(N_r)$ is surjective.  Thus, we may find a section $\epsilon_j \in T_i$ such that $\vec{v} : = d \text{ev}_{q_j}|_{s_i} (\epsilon_j)$ satisfies $\vec{v} \notin T_{p_j} S_i$. 

We claim $\rho \circ \epsilon_j : \mathcal{O}_{C} \to  (s_{i+1}^{*}T_{E_{i+1}/S_{i}}(E_{i+1} - \widetilde{E}_i) )|_{q_j}$ is surjective.  Indeed, the fiber $F$ of $E_{i+1}$ over $p_j$ is isomorphic to $\PP^1$.  The tautological quotient-bundle map $\phi_{i+1}^*N_{S_i} \to T_{E_{i+1}/S_{i}}(E_{i+1})$ restricts to a surjective map $\mathcal{O}_F^{\oplus 2} \to \OO_F(1)$.  The tangent vector $\vec{v}$ corresponds to a nonzero global section $v \in H^0(F, \mathcal{O}_F^{\oplus 2})$ whose image in $\OO_F(1)$ vanishes over $F \cap \widetilde{E}_i$.  Hence, $\vec{v}$ is a nonzero global section of $T_{E_{i+1}/S_{i}}(E_{i+1} - \widetilde{E}_i)|_F \cong \OO_F$.  Any nonzero global section of $\OO_F$ generates every stalk of $\OO_F$.  Pulling back by $s_{i+1}$, we obtain a nonzero section that generates the stalk of $s_{i+1}^{*}T_{E_{i+1}/S_{i}}(E_{i+1} - \widetilde{E}_i)$ over $q_j$.

Let $\epsilon = \sum_{q_j \in \mathcal{Q}} a_j \epsilon_j$ be a general linear combination of the $\epsilon_j$ obtained above.  By construction, $\epsilon \in T_i$ and $\rho \circ \epsilon : \mathcal{O}_{C} \to (s_{i+1}^{*}T_{E_{i+1}/S_{i}}(E_{i+1} - \widetilde{E}_i))|_{\mathcal{Q}}$ is nonvanishing over any point of $\mathcal{Q}$ that maps into $E_{i+1}$.  Since $T_{E_{i+1}/S_{i}}(E_{i+1} - \widetilde{E}_i)$ is locally the quotient of a line bundle on $Y_{i+1}$, $ (s_{i+1}^{*}T_{E_{i+1}/S_{i}}(E_{i+1} - \widetilde{E}_i) )_{\mathcal{Q}}$ is the quotient of a line bundle on $C$ as well.  Thus, any map $\mathcal{O}_{C} \to (s_{i+1}^{*}T_{E_{i+1}/S_{i}}(E_{i+1} - \widetilde{E}_i))|_{\mathcal{Q}}$ which is nonvanishing over each nontrivial stalk of $(s_{i+1}^{*}T_{E_{i+1}/S_{i}}(E_{i+1} - \widetilde{E}_i))|_{\mathcal{Q}}$ will be surjective.
\end{proof}

\begin{corollary}\label{cor:meaningofvg}
    Adopt the notation of Lemma \ref{lemm:meaningofvg} and distinguish a point $q \in \mathcal{Q}$.  Let 
    $$M_{j} = \{ s \in M \mid \text{the strict transform } s_j : C \to Y_j \text{ of } s \text{ maps } q \text{ into } E_j \}.$$ 
    For any $j \geq 1$ and any map $s \in M_j$, $\dim_{s} M_{j} < \dim_{s} M_{j-1}$ and thus $\dim_{s} M_0 \geq \dim_{s} M_j + j$.
\end{corollary}

\begin{proof}
    Since $M_j$ is a subvariety of $M_{j-1}$, it suffices to prove that no irreducible component of $M_{j-1}$ lies in $M_j$.  This follows from our application of Bertini's Theorem in the proof of Lemma \ref{lemm:meaningofvg}: every map $s \in M_{j}$ belongs to a stratum $N_r \subset M_{j-1}$ such that  $\mathrm{ev}_q(N_r) \subset E_{j-1}$ is transverse to $S_{j-1}$.   
\end{proof}

The following lemma describes when a family of curves contains members that meet the ``top" of a general exceptional tower.  This will be used to find subfamilies that satisfy the hypotheses of Lemma \ref{lem-numericalExceptionalTowers}.

\begin{lemma} \label{lem: climb tower morphism scheme}
    Let $Y$ be a smooth variety and $C$ be a smooth curve.  Suppose $M \subset \Mor(C,Y)$ is a quasi-projective subvariety that parameterizes generically injective maps.  Let $q \in C$ be a general point.  Let $p \in Y$ be a general point in the subvariety of $Y$ swept out by $s(q)$ as we vary $s\in M$.  Let $\Phi(Y,p,k)$ be a general exceptional tower based at $p$.  For $0 \leq j \leq k$ set
    $$M_{j} = \{ s \in M \mid \text{the strict transform } s_j : C \to Y_j \text{ of } s \text{ maps } q \text{ into } E_j \}.$$ 
    Suppose that for a general map $s\in M$, the tangent space $T_s M \subset H^0(C, s^*T_Y)$ of $M$ at $s$ and the height $k$ of $\Phi(Y,p,k)$ satisfy one of the following hypotheses:
    \begin{enumerate}
    \item the image of $ds : T_C \to s^*T_Y$ intersects the subsheaf of $s^*T_Y$ generated by $T_s M$ only along the zero section, and $k \leq \dim M_0$; 
    \item the intersection of $T_s M$ with the image of $ds_* : H^0(C, T_C) \to H^0(C, s^*T_Y)$ is a subspace of dimension $r > 0$, and $k \leq \dim M_0 + 1 - r$; or
    \item there is a subsheaf $\mathcal{F} \subset s^*T_Y$ such that $T_s M = H^0(C,\mathcal{F})$, and $k \leq \frac{\dim M_0 -2}{\rk(\mathcal{F})} - 2g(C)$.
    \end{enumerate}
    Then $M_k$ is nonempty and of dimension $\dim M_0 - k$.
\end{lemma}

We will apply Lemma \ref{lem: climb tower morphism scheme} to spaces $M$ parameterizing graphs of maps, in which case \ref{lem: climb tower morphism scheme}(1) is satisfied; to components $M \subset \Mor(\PP^1, Y)$, in which case \ref{lem: climb tower morphism scheme}(2) is satisfied; and to very specific subfamilies of free curves for which \ref{lem: climb tower morphism scheme}(3) is satisfied.  The following example of Eric Larson shows that if none of \ref{lem: climb tower morphism scheme}(1)--(3) hold for $T_s M$, then $M_k$ may be empty or have dimension less than $\dim M_0 - k$.

\begin{example}\label{example: Eric L's family of conics}
    Let $L \subset \PP^2$ be a line and $G \subset \mathrm{PGL}(3)$ be the subgroup that fixes $L$ pointwise.  Let $s : \PP^1 \to \PP^2$ be the embedding of a conic.  Consider the orbit $M \subset \Mor(\PP^1, \PP^2)$ of $s$ under $G$.  Then, with the notation of Lemma \ref{lem: climb tower morphism scheme}, we claim $M_0$ is nonempty and of dimension $1$, but $M_1$ is empty.  Indeed, the $G$-stabilizer of a point $p \in \PP^2\setminus L$ is isomorphic to $\mathbb{G}_m$ and it acts by scaling on lines through $p$.  Thus, deformations of $s$ that fix the image of a point $q \in \PP^1$ also fix the tangent line to the curve at $p = s(q)$.  

    Note that Lemma \ref{lem: climb tower morphism scheme} does not apply since the tangent space $T_s M$ and the value $k=1$ do not satisfy any of \ref{lem: climb tower morphism scheme}(1)--(3).  Indeed, the tangent space to $G$ in the automorphism group of $\PP^2$ corresponds to the space of global sections $H^0(\PP^2, T_{\PP^2}(-L))$.  The image of this space of sections in $H^0(\PP^1, s^*T_{\PP^2})$ is $T_s M$.  Moreover, the Euler sequence on $\PP^2$ shows $T_s M$ generates the subsheaf $s^*T_{\PP^2}(-L)$ inside $s^*T_{\PP^2}$.  The image of $ds : T_{\PP^1} \to T_{\PP^2}$ intersects $s^*T_{\PP^2}(-L)$ along $T_{\PP^1}(-s^*L)$ showing \ref{lem: climb tower morphism scheme}(1).  Only finitely many elements of $G$ fix the image of $s$.  Hence the intersection of $T_s M$ with the image of $ds_* : H^0(\PP^1, T_{\PP^1}) \to H^0(\PP^1, s^*T_{\PP^2})$ is zero showing \ref{lem: climb tower morphism scheme}(2) fails.  In particular, $T_s M$ does not contain the image of $H^0(\PP^1, T_{\PP^1}(-s^*L))$ under $ds_*$, and thus $T_s M \subsetneq H^0(\PP^1, s^*T_{\PP^2}(-L))$ showing \ref{lem: climb tower morphism scheme}(3) fails.  This proves our claim.
\end{example}

\begin{proof}[Proof of Lemma \ref{lem: climb tower morphism scheme}]
    We first construct a specific exceptional tower $\Phi(Y,p,k)$ such that $\dim M_k = \dim M_0 - k$.  We use Notation \ref{nota:exctower} throughout our proof to describe this tower.  We finish the proof by generalizing the choices of $p$ and $S_{i}$ that define this tower and appealing to upper semicontinuity of fiber dimension.

    Consider a general pair $(f,q) \in M \times C$. 
    By generality of $(f,q)$, the fiber of the evaluation map $\mathrm{ev} : M \times C \to Y \times C$ is smooth at $(f,q)$ and $f(C)$ is smooth at $p = f(q)$. 
    
    Let $\phi_0 : Y_0 \to Y$ be the blow-up of $Y$ at $p$.  Because $f(C)$ is smooth at $p$, the strict transform $f_0: C \to Y_0$ of $f$ meets $E_0$ transversely and only at one point, the image of $q$.  Observe that $f_0^*T_{Y_0}(-\log_{(f_0, \{q\})} \Delta_0) \cong f^*T_Y(-q)$ by Lemma \ref{lemm:loctancomputation}.
    
    We now describe some properties of the tangent space $T_f M$.  Let $\tilde{T}_C \subset f^*T_Y$ be the saturation of the image of $df : T_C \to f^*T_Y$.  Since $f \in M$ is general, by \cite[Lemma 3.41]{HM98}  
    the intersection of $T_f M$ and $H^0(C, \tilde{T}_C)$ inside $H^0(C, f^*T_Y)$ is contained in the image of $df_* : H^0(C, T_C) \to H^0(C, f^*T_Y)$.  Let $R \subset T_f M$ be this intersection and set $r = \dim R$.  

    Let $V \subset f^*T_Y$ be the subsheaf generated by $T_f M$.  Note that in the case where (3) is satisfied, $V \subset \mathcal{F}$ and so $H^0(C,V) = T_f M$.  As $\rk(V) \leq \rk(\mathcal{F})$, it suffices to consider the special case $\mathcal{F} = V$.
    
    Let $\overline{V}$ be the image of $V$ in $f^*T_Y / \tilde{T}_C$.  Since $V$ is generated by $T_f M$,  $\overline{V}$ is generated by the image of $T_f M$ in $H^0(C, \overline{V})$ as well.  
    By an incidence correspondence argument, a general element of a space of global sections that generates a vector bundle on $C$ of rank at least two is nowhere vanishing.  It follows that there is a subspace $T \subset T_f M$ of dimension $\rk(\overline{V}) - 1$ such that the natural map $\iota : T \otimes \OO_C \to \overline{V}$ is an inclusion with saturated image.  Let $Q$ be the (rank $1$) cokernel of $\iota$ and let $\pi : V \to Q$ be the induced map.  Observe that the kernel of $\pi_* : T_f M \to H^0(C,Q)$ is the linear subspace $R + T$.

    Consider the image $W \subset H^0(C, Q)$ of $T_f M$ under $\pi_*$.  By the preceding paragraphs, we conclude $\dim W = \dim M - \rk(\overline{V}) + 1 - r$.  Using the equality $\dim M_0 = \dim M - \rk(V)$, and the fact that $r$ can only be nonzero when $\rk(\overline{V}) < \rk(V)$, we see that $\dim W > k$ whenever hypothesis (1), (2), or (3) holds.  
    
    The Pl\"ucker formula for ramification sequences \cite[Theorem~7.13]{3264} shows that the weight of a linear system is finite; since $q \in C$ is general its weight with respect to the linear system $W \subset H^0(C, Q)$ is 0.  In particular, the natural map
        $$\theta_W : W \to Q|_{(k+1)q}$$
    is surjective because $\dim W > k$.  We will use the map $\theta : T_f M \to Q|_{(k+1)q}$ induced by $\theta_W$ to construct the desired exceptional tower.

    Define $\psi : T_f M \to Q|_{(k+1) q} \oplus V|_q$ to be the map coming from $\theta$ and the evaluation of sections at $q$.  Let $\psi_i : T_f M \to Q|_{(i+1)q} \oplus V|_q$ be the truncation of $\psi$ for all $0 \leq i \leq k$.  Observe that the image of $\psi_i$ is contained in the preimage $G_i$ of the diagonal under the natural map $Q|_{(i+1) q} \oplus V|_q \to Q|_q \oplus Q|_q$.  Our first goal is to show that the image of $\psi_i$ is $G_i$ for all $0 \leq i \leq k$.  
    
    If either hypothesis (1) or (2) holds, the kernel of $\theta$ surjects onto the kernel of $V|_q \to Q|_q$.  It immediately follows that the image of $\psi_i$ is $G_i$ in these cases. 
    Hence, we may suppose hypothesis (3) holds instead.  In fact, since the upper bounds on $k$ in hypotheses (1) and (2) are greater than or equal to the upper bound on $k$ in hypothesis (3), we may suppose neither hypothesis (1) nor (2) holds.  
    In other words, we may assume that the intersection of $V$ with the image of $df : T_C \to f^*T_Y$ is nontrivial and that the intersection of $T_s M = H^0(C,V)$ with the image of $df_* : H^0(C, T_C) \to H^0(C, s^*T_Y)$ is the zero section.  Since the intersection of $V$ with the image of $df : T_C \to f^*T_Y$ is saturated in $V$ by \cite[Lemma 3.41]{HM98}, it follows that $\tilde{T}_C \cap V$ is a subsheaf of $T_C$ with no sections, and in particular is a line subbundle of negative degree. In contrast, we have $\deg V \geq h^0(C, V) - \rk(V) = \dim M_0$.

    If $V$ is unstable, let $F \subset V$ be the maximal destabilizing subsheaf; otherwise, let $F = V$.  Since $\mu(F) \geq \mu(V) > 0$, $\mu^{max}(\ker(V \to Q)) \leq 0$, and $q \in C$ is a general point, the induced map $F \to Q$ is surjective in a neighborhood of $q$.  Since $V$ is globally generated, to verify that the image of $\psi_i$ is $G_i$ it suffices to show the natural map $H^0(C,F) \to F|_{(k+1)q}$ is surjective.  To prove this, observe that $\mu(F) \geq \mu(V) > k + 2g(C)$.  Hence, the vector bundle $\Omega_C \otimes F^{\vee}((k+1)q)$ is semistable with negative slope.  Thus, by Riemann--Roch $H^1(C, F(-(k+1)q)) = 0$, which proves surjectivity of $H^0(C,F) \to F|_{(k+1)q}$.  We conclude that the image of $\psi_i$ is $G_i$ for all  $0 \leq i \leq k$.

    Having shown the image of $\psi_i$ is $G_i$, we will use the following consequence.  Identify the kernel of $Q|_{(i+1)q} \to Q|_q$ with $Q(-q)|_{iq}$.  For all $0 < i \leq k$, since the image of $\psi_i$ is $G_i$ we see that $\psi_i$ maps $\ker(\psi_0)$ surjectively onto $Q(-q)|_{iq} \oplus 0$.  This shows $\rk(\psi_{i}) > \rk(\psi_{i-1})$.

    By generality of $q \in C$, there exists a line bundle $Q'$ and a quotient map $\pi' : f^*T_Y \to Q'$ such that $\pi'|_V$ factors through $\pi : V \to Q$ and the induced map $Q \to Q'$ is an isomorphism near $q$.  For each $0 \leq i \leq k$, let $\tau_i : f^*T_Y \to Q'|_{(i + 1)q} \oplus f^*T_Y|_q$ be the natural map. We note that $\tau_0$ factors through $f^*T_Y \to f^*T_Y|_q$, so that 
    $$\ker(\tau_0) \cong f^*T_Y(-q) \cong f_0^*T_{Y_0}(-\log_{(f_0, \{q\})} \Delta_0).$$
    Since the map $Q\to Q'$ is an isomorphism near $q$, $\tau_i$ therefore restricts to a surjective map $f^*T_Y(-q) \to Q'(-q)|_{iq}$ for all $1 \leq i \leq k$.

    We now construct the $Y_j$ of the exceptional tower for $j > 0$, defining $M_j$ as in the statement of this lemma and letting $f_j: C \to Y_j$ denote the strict transform of $f$. We wish to show that $Y_j$ satisfies:
\begin{enumerate}[label=(\roman*)]
\item $f_j(C)$ meets $E_j$ transversely at $f_{j}(q)$ and does not meet $\widetilde{E}_i$ for any $i < j$,
\item $\dim_{f} M_j = \dim_f M_0 - j$ and $M_j$ is smooth at $f$, and
\item  $f_j^*T_{Y_j}(-\log_{(f_j, \{q\})} \Delta_j) \cong \ker(\tau_j)$.
\end{enumerate}

\noindent By construction these properties are satisfied for $j = 0$.

Suppose that for some $j < k$ and all $i \leq j$ we've constructed $Y_i$ with properties (i)-(iii).  As $\tau_{j+1}$ restricts to a surjective map $f^*T_Y(-q) \to Q'(-q)|_{(j+1)q}$, it follows that $\tau_{j+1}$ maps $\ker(\tau_j) \cong f_j^*T_{Y_j}(-\log_{(f_j, \{q\})} \Delta_j)$ surjectively onto the kernel of the natural map $Q'(-q)|_{(j+1)q} \to Q'(-q)|_{jq}$.  We identify this latter space with $\mathbb{K}(q)$ and denote the induced map by $\tau_{j+1}' : f_j^*T_{Y_j}(-\log_{(f_j, \{q\})} \Delta_j) \to \mathbb{K}(q)$.

We next analyze the maps on tangent spaces at $q$ coming from the exceptional tower.  In the fiber over $q$, the kernel of $\tau_{j+1}'|_q$ is a hyperplane $K \subset f_j^*T_{Y_j}(-\log_{(f_j, \{q\})} \Delta_j)|_q$.  Moreover, for all $i < j$ the inclusion $\ker(\tau_j) \subset \ker(\tau_i)$ induces the natural map 
$$\gamma_{i,j} : f_j^*T_{Y_j}(-\log_{(f_j, \{q\})} \Delta_j)|_q \to f_i^*T_{Y_i}(-\log_{(f_i, \{q\})} \Delta_i)|_q$$ 
on log tangent spaces.  We claim the restriction of these maps to $K$ is injective.  To see this, tensor the following commutative diagram with $\mathbb{K}(q)$:
\[\begin{tikzcd}
	& {\ker(\tau_{j+1})} & {\ker(\tau_j)} & {\mathbb{K}(q)} & 0 \\
	0 & {f^*T_Y(-q) } & {f^*T_Y(-q) } & 0.
	\arrow["{\varsigma}"', from=1-2, to=1-3]
	\arrow["{\iota_{j+1}}", from=1-2, to=2-2]
	\arrow["{\tau_{j+1}'}"', from=1-3, to=1-4]
	\arrow["{\iota_j}", from=1-3, to=2-3]
	\arrow[from=1-4, to=1-5]
	\arrow[from=1-4, to=2-4]
	\arrow[from=2-1, to=2-2]
	\arrow["{\mathrm{id}}", from=2-2, to=2-3]
	\arrow[from=2-3, to=2-4]
\end{tikzcd}\]
The image of $\varsigma|_q$ is $K$ and $\rk(\iota_{j+1}|_q) = \rk(\iota_{j}|_q) = \dim K$.  Since $\ker(\tau_0) \cong f^*T_Y(-q)$ and $\iota_j$ is the inclusion map, we see that $\gamma_{0,j} = \iota_j|_q$ is injective on $K$.  Since $\gamma_{i,j}$ factors $\gamma_{0,j}$ for all $i > 0$, this implies our claim. 

Observe that the composition of maps $T_C \to f^*T_Y \to Q'$ is 0. Thus, $K$ contains the image of the log tangent space of $C$ at $q$ under $d_{\log}f_j|_q$.  This latter space is the kernel of the natural map from the log tangent space of $Y_j$ at $f_j(q)$ to the tangent space $T_{f_j(q)} Y_j$.  Since the image of the log tangent space in $T_{f_j(q)} Y_j$ is the tangent space $T_{f_j(q)} E_j$ of $E_j$, we conclude that the image of $K$ in $T_{f_j(q)} Y_j$ is a hyperplane section $\overline{K}$ of $T_{f_j(q)} E_j$.

If $j = 0$, then we define $S_0$ to be the hyperplane in $E_0 \cong \PP^{\dim Y-1}$ passing through $f_0(q)$ with tangent space $\overline{K}$.  
If $j > 0$, then injectivity of $\gamma_{j-1,j}$ on $K$ implies that $\overline{K}$ does not contain the kernel of the map $T_{f_j(q)} E_j \to T_{f_{j-1}(q)}S_{j-1}$.   By property (i), the point $f_{j}(q)$ does not meet $\widetilde{E}_{j-1}$ and so does not lie on the rigid section of $E_{j} \to S_{j-1}$.  Because of this, we can find a minimal moving section $S_j$ of $E_j \to S_{j-1}$ that passes through $f_j(q)$ with tangent space $\overline{K}$.    
In both cases we define $\phi_{j+1} : Y_{j+1} \to Y_{j}$ to be the blow-up of $S_{j}$.

With this construction, Remark \ref{rem: cokernel} shows $\tau_{j+1}$ is the map $\tau$ in Lemma \ref{lemm:loctancomputation}(2), which verifies (iii).  
Because $f_j(C)$ is transverse to $E_j$ and meets $S_j$, we see that $Y_{j+1}$ satisfies property (i).  To check (ii), observe that the restriction of $\tau_{j+1}$ to $T_f M \subset H^0(C, f^*T_Y)$ factors through $\psi_{j+1}$.  
As $\rk(\psi_{j+1}) > \rk(\psi_j)$, the tangent space to $M_{j+1}$ at $f$ has dimension $\dim_f M_j - 1$.  This shows $M_{j+1}$ is smooth at $f$ because $\codim(M_{j+1}, M_j) \leq 1$ and $M_j$ is smooth at $f$.  Thus $Y_{j+1}$ satisfies property (ii).

Continuing in this way, we may increment $j$ until $j = k$. This constructs the desired exceptional tower.  We now prove that $M_{k}$ satisfies the statement of this lemma for a general exceptional tower based at $p$ by a deformation argument.
One can construct a moduli space $\mathscr{E}(Y,p,k)$ of exceptional towers based at $p$ and of height $k$ on $Y$ by keeping track of the choice of the divisors $S_{i} \subset E_{i}$.  Let $\mathscr{M}_j \subset M \times \mathscr{E}(Y,p,k)$ be the relative version of $M_j$.  An incidence correspondence argument shows the dimension of each irreducible component of $\mathscr{M}_k$ is at least $\dim \mathscr{M}_0 - k$.   Our constructed exceptional tower implies the projection $\mathscr{M}_k \to \mathscr{E}(Y,p,k)$ is dominant, as fiber dimensions are upper semicontinuous.   This proves that for a general exceptional tower based at $p$, $M_k$ is non-empty and of dimension $\dim M_0 - k$.
\end{proof}

\subsection{Breaking curves}
To prove Theorem \ref{theo:exctowerbreaking} we will lift a curve on $Y$ to an exceptional tower  $Y_{k}$, break the curve on  $Y_{k}$, and then push the resulting broken curve back down to $Y$.  In this section we will study the stable limits of curves passing through the top exceptional divisor in an exceptional tower.  As we will see, certain geometric properties of the curve are controlled by the combinatorics of how the broken curve interacts with the exceptional divisors.  

\begin{lemma}\label{lem-numericalExceptionalTowers}
Let $Y$ be a projective variety and let $\Phi(Y,p,k)$ be an exceptional tower of height $k$ based at a smooth point $p$ of $Y$.  Suppose $\pi : \mathcal{C} \to B$, $\mathrm{ev} : \mathcal{C} \to Y_k$ is a one-dimensional family of stable maps with a marking $\sigma : B \to \mathcal{C}$ such that $\mathrm{ev}(\sigma(B)) \subset E_k$.  Suppose as well that a general map parameterized by $B$ is the strict transform of an irreducible curve on $Y$. 

Consider any pointed map $s : (C,q) \to Y_k$ parameterized by $B$.  Let $C_{exc} \subset C$ be the connected component of 
$s^{-1}(\Delta_k)$ that contains $q$.  Let $\{C_\ell\}_{\ell \in \mathcal{L}}$ be the set of irreducible components of $C$ that meet $C_{exc}$ but are not contained in it and set $\mathcal{Q}_\ell = C_\ell \cap C_{exc}$.  Then   
\begin{equation*}
    \sum_{\ell \in \mathcal{L}} \sum_{q_{j,\ell} \in \mathcal{Q}_\ell} \left(\sum_{i = 0}^k (i+1) \cdot \mult_{q_{j,\ell}}\big(s|_{C_\ell}, \widetilde{E}_i\big)\right) \geq k+1.
\end{equation*}
\end{lemma}

\begin{proof}
      Set $C_{adj} = \cup_{\ell \in \mathcal{L}} C_{\ell}$.  If $C_{exc}$ is a point $q$, then $C_{adj}$ must be the strict transform of an irreducible curve in $Y$ so that $\mult_{q}(s|_{C_{adj}},E_k) \geq 1$ and the inequality holds.  Thus, we may suppose $C_{exc}$ is a curve and consider $C$, $C_{adj}$, and $C_{exc}$ as divisors in $\mathcal{C}$.  As $C_{exc}$ lies in a fiber of $\pi : \mathcal{C} \to B$, each component of $C_{exc}$ pairs to 0 with the class of a $\pi$-fiber.  In particular, for any divisor $D \subset \mathcal{C}$ supported on $C_{exc}$, 
      $$D \cdot (C_{exc} + C_{adj}) = D \cdot C = 0.$$

      For each $0 \leq i \leq k$, the support of the pullback $\mathrm{ev}^*\widetilde{E}_i$ may contain components of $C_{exc}$ but not any component of $C_{adj}$.  Let $D_i$ be the part of $\mathrm{ev}^*\widetilde{E}_i$ supported on $C_{exc}$.  When $i<k$, by construction $\mathrm{ev}^{*}\widetilde{E}_i \cdot C_{exc} \geq D_i \cdot C_{exc}$ while $D_i \cdot C_{exc} = -D_i \cdot C_{adj}$, showing that $D_{i} \cdot C_{adj} \geq -\mathrm{ev}^{*}\widetilde{E}_i \cdot C_{exc}$.  As $\sigma(B)$ is contained in $\mathrm{ev}^*\widetilde{E}_k$ and $\sigma(B)$ meets $C_{exc}$, when $i=k$ the same argument yields the improved bound $ D_k \cdot C_{adj} \geq -\mathrm{ev}^{*}\widetilde{E}_k \cdot C_{exc} + 1$.

       Since $s(C_{exc})$ is contracted by the birational morphism $Y_{k} \to Y$, up to numerical equivalence we can write
\begin{equation*}
s_{*}C_{exc} \equiv \sum_{i=0}^{k} b_{i} \ell_{i}
\end{equation*}
for some integers $b_{i} \geq 0$.   Note that the multiplicity $\mult_{q_{j,\ell}}(s|_{C_\ell}, \widetilde{E}_i)$ is the local intersection number of $C_{\ell}$ with $D_i$ at $q_{j,\ell}$.  Since $D_i \cap C_{\ell} \subset \mathcal{Q}_{\ell}$ for all $\ell \in \mathcal{L}$, we can compute $D_{i} \cdot C_{adj}$ as a sum of multiplicities over all $\{ q_{j,\ell} \in \mathcal{Q}_{\ell}\}_{\ell \in \mathcal{L}}$.  On the other hand, $\ev^{*}\widetilde{E}_{i} \cdot C_{exc}$ can be computed by combining the above description of $s_{*}C_{exc}$ with Lemma \ref{lemm:intersectionvalues}.  With these modifications, our earlier inequalities show that for $0 < i < k$ we have
\begin{equation*}
\left( \sum_{\ell \in \mathcal{L}} \sum_{q_{j,\ell} \in \mathcal{Q}_{\ell}} \mult_{q_{j,\ell}}(s|_{C_\ell}, \widetilde{E}_i)  \right) \geq 2b_{i} - b_{i+1} - b_{i-1} 
\end{equation*}
and that
\begin{equation*}
\left( \sum_{\ell \in \mathcal{L}} \sum_{q_{j,\ell} \in \mathcal{Q}_{\ell}} \mult_{q_{j,\ell}}(s|_{C_\ell}, E_k)  \right)  \geq 1 + b_{k} - b_{k-1}, \qquad  \left( \sum_{\ell \in \mathcal{L}} \sum_{q_{j,\ell} \in \mathcal{Q}_{\ell}} \mult_{q_{j,\ell}}(s|_{C_\ell}, \widetilde{E}_0) \right)  \geq 2b_{0} - b_{1}.
\end{equation*}
Combining these equations, we obtain
\begin{align*}
\sum_{i=0}^{k} (i+1) & \left( \sum_{\ell \in \mathcal{L}} \sum_{q_{j,\ell} \in \mathcal{Q}_{\ell}} \mult_{q_{j,\ell}}(s|_{C_\ell}, \widetilde{E}_i) \right)  \\
& \geq (k+1)(1+b_{k} - b_{k-1}) + \left(\sum_{i=1}^{k-1} (i+1)(2b_{i} - b_{i+1} - b_{i-1})\right) + (2b_{0} - b_{1}) \\
& = k+1 + b_{k} \\
& \geq k+1. \qedhere
\end{align*}
\end{proof}

\begin{remark} \label{rmk-numericalExceptionalTowers}
    Observe from the proof of Lemma \ref{lem-numericalExceptionalTowers} that the key ingredients are the intersection numbers for the curves in the family $\mathcal{C}$ with the exceptional divisors $E_i$. We can also achieve the same conclusion if we vary $Y$ in a 1-parameter family $B$.
    
    More precisely, suppose that:
    \begin{enumerate}
        \item $\varphi : \mathcal{Y} \to B$ is a flat projective map, $p_B \in \mathcal{Y}$ is a $\mathbb{K}(B)$-point contained in the smooth locus of $\varphi$, and $\Phi(\mathcal{Y}_{\mathbb{K}(B)}, p_B, k)$ is an exceptional tower over the generic point of $B$;
        \item there is a sequence $\mathcal{Y}_{k} \to \ldots \to \mathcal{Y}_{0} \to \mathcal{Y}$ of integral models over $B$ of the blow-ups defining $\Phi(\mathcal{Y}_{\mathbb{K}(B)}, p_B, k)$ such that the restriction to each geometric fiber of $\varphi$ is an exceptional tower;
        \item $\mathcal{C} \to B$ is a generically smooth family of prestable curves equipped with a $B$-morphism $\ev: \mathcal{C} \to \mathcal{Y}_{k}$ whose image is not contained in the exceptional locus of $\mathcal{Y}_{k} \to \mathcal{Y}$.
    \end{enumerate}
    Suppose there is a marking $\sigma: B \to \mathcal{C}$ such that $\ev(\sigma(B))$ lies in the ``top'' exceptional divisor $\mathcal{E}_k$. 
    Then for every point $t \in B$, defining $\{ C_\ell \}$ and $C_{exc}$ as above, we still have 
    \begin{equation*}
    \sum_{\ell \in \mathcal{L}} \sum_{q_{j,\ell} \in \mathcal{Q}_\ell} \left(\sum_{i = 0}^k (i+1) \cdot \mult_{q_{j,\ell}}\big(s|_{C_\ell}, \widetilde{E}_i\big)\right) \geq k+1.
\end{equation*}
\end{remark}

We now connect the inequality in Lemma \ref{lem-numericalExceptionalTowers} to a lower bound on the anticanonical degree of a curve $s : C \to Y$.  The most general statement we will need is Proposition \ref{prop:lower_bound_exceptional_tower}.  Its corollary provides such a lower bound when $C = \PP^1$ and $s^*T_Y$ is globally generated.

\begin{proposition}\label{prop:lower_bound_exceptional_tower}
Let $(Y,\Delta)$ be a log smooth pair with reduced boundary and let $p \in Y\setminus \Delta$ be a closed point. 
Let $C$ be a smooth curve and fix a finite set of points $\mathcal{Q} \subset C$.  Suppose $M \subset \Mor(C,Y)$ 
is a quasi-projective subvariety.  

Consider a general exceptional tower $\Phi(Y,p,k)$ of height $k \geq 1$ based at $p$.
Suppose $s : C \to Y$ is a  map parameterized by $M$ such that $s(C) \not \subset \Delta$.  Let $V \subseteq s^*T_{Y}(-\log_{(s,\mathcal{Q})} \Delta)$ and $V_k \subseteq s_k^*T_{Y_k}(-\log_{(s_k,\mathcal{Q})} \Delta_{k})$ be the maximal generically globally generated subbundles of the respective local log tangent sheaves.  Then 
$$c_1(V) - c_1(V_k) \geq \sum_{q_{j} \in \mathcal{Q}} \left( \rk(V_k) \cdot \mathrm{mult}_{q_j}(s_0,E_0) +  \sum_{i = 1}^k \mathrm{mult}_{q_j}(s_i, E_i) \right).$$
\end{proposition}

\begin{proof}
For $0 \leq i < k$, Lemma \ref{lemm:loctancomputation}.(2) yields an exact sequence
$$0 \to s_{i+1}^*T_{Y_{i+1}}(-\log_{(s_{i+1},\mathcal{Q})} \Delta_{i+1}) \to s_{i}^*T_{Y_{i}}(-\log_{(s_{i},\mathcal{Q})} \Delta_{i}) \to P_{i+1} \to 0$$
where $P_{i+1}$ is torsion and $(P_{i+1})_{\mathcal{Q}} = \big (s_{i+1}^{*}T_{E_{i+1}/S_{i}}(E_{i+1} - \widetilde{E_i}) \big )_{\mathcal{Q}}$.
  
For $0 \leq i \leq k$, let $V_i \subset s_i^*T_{Y_i}(-\log_{(s_i,\mathcal{Q})} \Delta_{i})$ be the maximal generically globally generated subbundle and let $\delta_{i}$ denote its rank.  We let $T_{i+1}$ denote the torsion sheaf which is the image of $V_{i} \to P_{i+1}$.  The exact sequence above shows that $V_{i+1}$ is contained in the kernel of $V_{i} \to T_{i+1}$ (but might not be equal to it).  
Lemma \ref{lemm:meaningofvg} shows that $(P_{i+1})_{\mathcal{Q}} \subset T_{i+1}$.  Since $T_{E_{i+1}/S_{i}}(E_{i+1} - \widetilde{E_i})$ is a line bundle on $E_{i+1}$, Definition \ref{def:local multiplicity} shows that the length of $(P_{i+1})_{\mathcal{Q}}$ is $\sum_{q_{j} \in \mathcal{Q}} \mathrm{mult}_{q_j}(s_{i+1}, E_{i+1})$.  This provides a lower bound for the length of $T_{i+1}$.

Altogether we deduce that for each $i$ with $0 \leq i < k$, one of the following occurs: 
\begin{enumerate}
\item $\delta_{i+1} = \delta_{i}$ and $c_1(V_{i+1}) \leq c_1(V_{i}) - \sum_{q_{j} \in \mathcal{Q}} \mathrm{mult}_{q_j}(s_{i+1}, E_{i+1})$,
\item $\delta_{i+1} <  \delta_{i}$ and $c_1(V_{i+1}) \leq c_1(V_{i})$.
\end{enumerate}
Case (2) occurs at most $\delta_{0} - \delta_k$ times.  Note that $\mathrm{mult}_{q_{j}}(s_i,E_i) \geq \mathrm{mult}_{q_{j}}(s_{i+1},E_{i+1})$ whenever $i$ is between $0$ and $k-1$.  Summing up the changes (1)-(2) over all $0 \leq i < k$ we obtain
$$c_1(V_0) \geq c_1(V_k) + \sum_{q_{j} \in \mathcal{Q}} \left( (\delta_k - \delta_0) \cdot \mathrm{mult}_{q_j}(s_0,E_0) +  \sum_{i = 1}^k \mathrm{mult}_{q_j}(s_i, E_i) \right).$$

By Lemma \ref{lemm:loctancomputation}.(1) we can identify $s_{0}^*T_{Y_{0}}(-\log_{(s_{0},\mathcal{Q})} \Delta_{0})$ with a subsheaf of the twist of $s^*T_Y(-\log_{(s,\mathcal{Q})} \Delta)$ by $\mathcal{O}_C\big(-\sum_{q_j \in \mathcal{Q}} \mathrm{mult}_{q_j}(s,p)\cdot q_j \big)$.  Since $\mathrm{mult}_{q_j}(s,p) = \mathrm{mult}_{q_j}(s_0,E_0)$, we have
\begin{align*}
c_1(V) \geq & c_1\bigg (V_0\big(  \sum_{q_j \in \mathcal{Q}} \mathrm{mult}_{q_j}(s,p)\cdot q_j\big)\bigg ) \\
\geq & c_1(V_k) + \sum_{q_{j} \in \mathcal{Q}} \left( \delta_k \cdot \mathrm{mult}_{q_j}(s_0,E_0) +  \sum_{i = 1}^k \mathrm{mult}_{q_j}(s_i, E_i) \right).
\end{align*}
\end{proof}

When we break a family of curves, the irreducible components of the broken curve might intersect many different exceptional divisors $\widetilde{E}_{i}$ in our tower.  The following result allows us to leverage information about these intersection numbers to prove that the images on $Y$ have high $-K_Y$-degree. 

\begin{corollary}\label{cor: rational curves degree}
    Let $Y$ be a smooth projective variety.  Suppose $s : \PP^1 \to Y$ is a nonconstant map such that $s^*T_Y$ is globally generated.  Let $p \in s(\PP^1)$ be a point and consider a general exceptional tower $\Phi(Y,p,k)$ based at $p$, where $k \geq 0$.  For any subset $\mathcal{Q} \subset \PP^1$ of at most two points,

    $$-K_Y \cdot s_*\PP^1 \geq 2 - |\mathcal{Q}| + \sum_{q_{j} \in \mathcal{Q}}  \sum_{i=0}^{k} (i+1) \cdot \mult_{q_{j}}(s_k, \widetilde{E}_i).$$
\end{corollary}

\begin{proof}
    We use the notation from Proposition \ref{prop:lower_bound_exceptional_tower}. Here, $V = s^* T_Y$ because it is globally generated.  The line bundle $T_{\PP^1}(-\log \mathcal{Q})$ has degree $2 - |\mathcal{Q}|$.  Because $\mathcal{Q}$ consists of at most two points, this line bundle is globally generated, and so the nontrivial map $T_{\PP^1}(-\log \mathcal{Q}) \to s_k^*T_{Y_k}(-\log_{(s_k,\mathcal{Q})} \Delta_{k})$ induced by $s_k$ shows that $V_k$ has rank at least 1 and degree at least $2- |\mathcal{Q}|$.  If $k=0$, Lemma \ref{lemm:loctancomputation}(1) shows that the difference in first Chern classes between $s^{*}T_{Y}$ and $V_{0}$ is at least $\sum \mult_{q_{j}}(s_{0},E_{0})$.  If $k \geq 1$, the conclusion follows from Proposition \ref{prop:lower_bound_exceptional_tower} because $\mult_{q_j}(s_i, E_i) = \sum_{h = i}^k \mult_{q_j}(s_k, \widetilde{E}_h)$.
\end{proof}

\subsection{Proof of Theorem \ref{theo:exctowerbreaking}}  To prove Theorem \ref{theo:exctowerbreaking}, we construct the desired rational curves by smoothing a connected component of a stable limit of curves parameterized by $M$.  The next result constructs this stable limit.

\begin{theorem}\label{theo:strongest breaking result}
    Let $Y$ be a smooth projective variety and $C$ be a smooth curve.    Suppose $M \subset \Mor(C,Y)$ is a quasi-projective subvariety parameterizing a dominant family of
    maps.  Let $q \in C$ be a general point.  Let $p \in Y$ be a general point in the subvariety of $Y$ swept out by $s(q)$ as we vary $s \in M$.  Set
    $$M_0 = \{s \in M \mid s(q) = p\}.$$
    Consider the natural morphism $M_0 \to \overline{\mathcal{M}}_{g,0}(Y)$ and let $W$ be the closure of its image.  If $\dim M_0 > 0$, then the domain of some stable map $f : C' \to Y$ parameterized by $W$ contains a connected subcurve $R \subset C'$ of arithmetic genus $0$ with the following properties:  
    \begin{enumerate}
        \item the image of each irreducible component of $R$ contains $p$;
        \item $-K_Y \cdot f_* R \geq \dim M_0$.
    \end{enumerate}
\end{theorem}

We will prove this theorem by looking at graphs of maps parameterized by $M$.

\begin{proof}
If $C$ is rational, our claim is trivial.  Otherwise, set $X = C \times Y$ 
and $k = \dim M_0 - 1$.  Let $\pi_Y : X \to Y$ and $\pi_C : X \to C$ be the projections and let $\Phi(X,(q,p),k)$ be a general exceptional tower based at $(q,p)$.  Let $\psi : X_k \to X$ denote the natural blow-up map.

Let $M_X \subset \Mor(C,X)$ be the space of graphs of maps in $M$.  
Observe that each map parameterized by $M_X$ is an embedding.  Let $\mathrm{ev} : C \times M_X \to X$ be the evaluation map of the universal family over $M_X$.  Our hypotheses ensure that $(q,p)$ is a general point in the image of $\mathrm{ev}$.  Moreover, the preimage $M_{X,0} : = \mathrm{ev}^{-1}(q,p)$ is naturally isomorphic to $M_0$ and henceforth we do not distinguish between them.

Since graphs of morphisms can only deform transversely to their tangent direction, the hypotheses of Lemma \ref{lem: climb tower morphism scheme}(1) are satisfied.  It follows from Lemma \ref{lem: climb tower morphism scheme} that there exists a one-parameter subfamily $M_{X,k} \subset M_{X,0}$ consisting of maps whose strict transforms in $X_k$ meet $E_k$.  We may consider (an open subset of) $M_{X,k}$ as a subvariety of $\Mor(C, X_k)$ by replacing each map with its strict transform.  Let $M_{X,k} \to \overline{\mathcal{M}}_{g,1}(X_k)$ be the morphism that sends a map $s \in M_{X,k}$ to the pointed map $s : (C,q) \to X_k$ and let $\overline{M}_{X,k} \subset \overline{\mathcal{M}}_{g,1}(X_k)$ be the closure of its image.  By Mori's Bend-and-Break, $\overline{M}_{X,k}$ must parameterize a stable pointed map $f : (C', q') \to X_k$ whose domain is reducible with a rational component.

We will now describe some elementary properties of the map $f : (C',q') \to X_k$.  By construction the marked point $q'$ lies in a rational component of $C'$.  As well, $f(q')$ must lie in $E_k$.  Since each map parameterized by $M_X$ is a section of $\pi_C$, there is a unique component $C_0 \subset C'$ that is not contracted by the composition of the blow-up $\psi : X_k \to X$ with the projection $\pi_C : X \to C$.  Observe that $\pi_C \circ \psi \circ f|_{C_0} : C_0 \to C$ must be an isomorphism and let $q_0 \in C_0$ be the preimage of $q$ under this map.

As in Lemma \ref{lem-numericalExceptionalTowers}, let $C_{exc} \subset C'$ denote the connected component of $f^{-1}(\Delta_k)$ that contains $q'$.  Further let $\{C_\ell\}_{\ell \in \mathcal{L}}$ be the set of irreducible components of $C'$ that meet $C_{exc}$ but are not contained in it.  For each $\ell \in \mathcal{L}$, set $\mathcal{Q}_\ell = C_\ell \cap C_{exc}$.  By Lemma \ref{lem-numericalExceptionalTowers} we have 
\begin{equation}\label{eq: mult bound}
    \sum_{\ell \in \mathcal{L}} \sum_{q_{j,\ell} \in \mathcal{Q}_\ell} \left(\sum_{i = 0}^k (i+1) \cdot \mult_{q_{j,\ell}}\big(f|_{C_\ell}, \widetilde{E}_i\big)\right) \geq k+1.
\end{equation}

Let $R$ be the union of $C_{exc}$ with all rational components of $C'$ in $\{C_\ell\}_{\ell \in \mathcal{L}}$.  Every component of $C'$ aside from $C_0$ is rational.  Therefore, as the dual graph of $C'$ is a tree, $R$ is a connected curve of arithmetic genus $0$.  This also implies $| \mathcal{Q}_\ell | = 1$ for all $\ell \in \mathcal{L}$.

If $C_0 \notin \{C_\ell\}_{\ell \in \mathcal{L}}$, then by Equation \eqref{eq: mult bound} and Corollary \ref{cor: rational curves degree}.  we obtain $-K_X \cdot (\psi \circ f)_* R \geq k + 2$.   In this case, taking the stabilization of $\pi_Y \circ \psi \circ f: C' \to Y$ yields the desired map.  

Suppose instead that $C_0 \in \{C_\ell\}_{\ell \in \mathcal{L}}$. 
Since $\Delta_k$ lies in the fiber of $\pi_C \circ \psi : X_k \to C$ over $q$, the only point at which $f(C_0)$ meets $\Delta_k$ is $f(q_0)$.  Moreover, since $f|_{C_0}$ is a section of $\pi_C \circ \psi$ we have $\sum_{i=0}^k \mult_{q_0}(f|_{C_0}, \widetilde{E}_i) = 1$.  Let $0 \leq m \leq k$ be the unique integer such that $\mult_{q_0}(f|_{C_0}, \widetilde{E}_m) = 1$.  By Equation \eqref{eq: mult bound} and Corollary \ref{cor: rational curves degree} we obtain $-K_X \cdot (\psi \circ f)_* R \geq k - m + 1$.  Therefore, if $m = 0$ we may again take the stabilization of $\pi_Y \circ \psi \circ f : C' \to Y$ to obtain the desired map.  

To finish our proof, we will induct on $k = \dim M_0 - 1$.  For the base case $k = 0$, observe that $m = 0$ so we have already proven the claim.  Now suppose that $m > 0$ and we have proven our theorem for any dominant family of maps $N \subset \Mor(C,Y)$ such that the corresponding variety $N_0 \subset \Mor(C,Y)$ satisfies $\dim N_0 \leq k$.

Let $M_X \to \overline{\mathcal{M}}_{g,0}(X)$ be the natural map and let $\overline{M}_X$ be the closure of its image.  Let $\theta : C'' \to X$ be the stabilization of $\psi \circ f: C' \to X$ and observe that $\theta$ corresponds to a point in $\overline{M}_X$.  The family $N \subset \Mor(C,Y)$ we would like to consider next corresponds to deformations of $\theta$ that do not smooth any node of $C''$.

Stratifying $\overline{M}_X$ by the combinatorial type of a map's domain decomposes $\overline{M}_{X}$ into finitely many locally closed substacks.  The subvariety $N \subset \Mor(C,Y)$ parameterizes one of finitely many families obtained by restricting maps in these substacks to the unique irrational component of their domain and post-composing with $\pi_Y$.  Since these finitely many families do not depend on our choices of $(q,p)$ or the exceptional tower $\Phi(X,(q,p),k)$, we may assume these choices are still general with respect to each such family, and so in particular with respect to $N$.  Correspondingly, $\Phi(X,(q,p),k)$ is also general with respect to the space $N_X \subset \Mor(C,X)$ of graphs of maps in $N$.

By construction, for every stable degeneration of a map in $N_X$ there is a degeneration of $\theta$ whose restriction to a subcurve agrees with the map in $N_{X}$.  Moreover, let $N_0$, $N_{X,0}$, and $\overline{N}_{X,0}$ be defined analogously to $M_0$, $M_{X,0}$, and $\overline{M}_{X,0}$. 
For any degeneration of $\theta|_{C_0}$ lying in $\overline{N}_{X,0}$, we may find a degeneration of $\theta$ in $\overline{M}_{X,0}$ that fixes the image in $X$ of the subcurve of $C''$ corresponding to $R \subset C'$.   Thus, to finish our proof it suffices to show maps in $N_{X,0}$ specialize to a stable curve containing a rational subcurve through $(q,p)$ of $-\pi_{Y}^{*}K_{Y}$-degree at least $m$.

To prove this, first note that $N$ parameterizes a dominant family of curves on $Y$ because $p \in Y$ is a general point and a curve in $N$ passes through $p$.  As $\theta|_{C_0} = \psi \circ f|_{C_0}$ is a map in $N_Y$ whose strict transform passes through $\widetilde{E}_{m}$, Corollary \ref{cor:meaningofvg} implies that $\dim_{\theta} N_0 = \dim_{\theta} N_{X,0} \geq m$.   We must have $\dim_{\theta} N_{X,0} \leq k$ because $N_{X,0}$ parameterizes restrictions to $C_0$ of deformations of $\theta$ in a proper subvariety of $M_{X,0}$.  Thus, by the inductive hypothesis 
applied to $N$, $\overline{N}_{X,0}$ contains a stable map $h : C''' \to X$ such that $C'''$ contains a connected rational subcurve $R'''$, all of whose components meet $h^{-1}(q,p)$, and such that $-K_X \cdot h_*R''' \geq \dim_{\theta} N_{X,0} \geq m$.  This finishes our proof.
\end{proof}

We are now ready to prove Theorem \ref{theo:exctowerbreaking}.

\begin{proof}[Proof of Theorem \ref{theo:exctowerbreaking}:] We may suppose the genus of $C$ is greater than 0.  Let $\phi: Y' \to Y$ be a smooth birational model.  Since $M$ parameterizes a dominant family of maps, there is a non-empty open sublocus $M' \subset M$ parameterizing the strict transforms on $Y'$ of general curves parameterized by $M$.  Since rational curves passing through general points of $Y'$ must have globally generated restricted tangent bundle, Theorem \ref{theo:strongest breaking result} allows us to find a stable map $f : R \to Y'$ from a genus 0 curve such that $f^*T_{Y'}$ is globally generated and $-K_{Y'} \cdot f_*R \geq \dim M' - \dim Y$.  Since $f^*T_{Y'}$ is globally generated, we may deform $f$ to obtain a free curve $s : \PP^1 \to Y'$ of the same degree.  The component of $\Mor(\PP^1, Y')$ that contains $s$ satisfies \ref{theo:exctowerbreaking}(1)-(3) for $M'$; we may choose $N$ to be the corresponding component of $\Mor(\PP^{1},Y)$.  Lastly, if $Y$ is already smooth then the dimension of $M$ is at least $-K_Y \cdot \alpha + (1-g) \dim Y$.   This proves the last statement of Theorem \ref{theo:exctowerbreaking}.  
\end{proof}

\section{Rational curves on terminal Fanos}

In this section our goal is to prove the following theorem describing the existence of particular dominant families of rational curves on mildly singular Fano varieties. The argument combines exceptional towers with the techniques behind the proof of \cite[Theorem 1.4]{JLR25}.

\begin{theorem}\label{thm: nicest curves}
Let $X$ be a terminal Fano variety of dimension $n \geq 3$.  Suppose $\pi: X' \to X$ is a resolution of singularities. 
For a general point $p \in X'$ and for any integer $\Theta \geq 1$, there is a rational curve $s : \PP^1 \to X'$ that passes through $p$ and satisfies
\begin{itemize}
\item $-K_{X'} \cdot s_* \PP^1 \geq \Theta$;
\item $ - K_{X} \cdot (\pi \circ s)_* \PP^1 \leq \Theta + n$;
\item $(K_{X'} - \pi^* K_{X}) \cdot s_* \PP^1 \leq n - 2$; and
\item $\pi \circ s(\PP^1)$ meets the singular locus of $X$ at most $n -2$ times. 
\end{itemize}
\end{theorem}

\begin{example}
    For certain $X$ and $\Theta$ there may be no rational curves in the smooth locus $X^{sm}$ of degree between $\Theta$ and $\Theta + n$.  For example, let $X' = \mathbb{P}_{\mathbb{P}^{n-1}}(\mathcal{O} \oplus \mathcal{O}(n-1))$ be a projective bundle and $X' \to X$ be the contraction of its rigid section.  The monoid of nef curve classes on $X'$ is generated by a fiber of the projection $g : X' \to \PP^{n-1}$ and by a minimal moving section of the map $g|_{X_L} : X_L \to L$, where $L \subset \PP^{n-1}$ is a line.  In particular, the minimal $(-K_X)$-degree of a rational curve contained in $X^{sm}$ is $2n - 1$.
    
    This same example shows that there may not be any generically injective maps $s : \PP^1 \to X$ of $(-K_X)$-degree between $\Theta$ and $\Theta + n$.  For $X$ as above, there will not be any such maps when $3 \leq \Theta < n - 1$.  We believe this issue could be rectified for large enough $\Theta$ by weakening the bounds in Theorem \ref{thm: nicest curves}, but do not pursue this here.
\end{example}

\subsection{Log deformation theory}
Suppose that $(Y,\Delta)$ is a pair consisting of a smooth projective variety $Y$ and a simple normal crossings divisor $\Delta$.  Let $f: C \to Y$ be a morphism from a smooth projective curve such that $f(C) \not \subset \Supp(\Delta)$.  We would like to study the deformations of $f$ while keeping track of the set $f^{-1}(\Delta)$.  

This is best accomplished using log deformation theory.  We can equip $Y$ with the log structure associated to $\Delta$ and equip $C$ with the log structure associated with $f^{-1}(\Delta)$.  Then deformations of $f$ are described by the proper log Deligne-Mumford stack $\mathscr{M}_{g,\varsigma}(Y,\beta)$ parameterizing stable log maps of curves to $Y$ which have genus $g$, curve class $\beta$, and contact orders at the marked points described by $\varsigma$ as constructed by \cite{Chen14, AC14, GS13}.

The log deformation theory of $f$ depends on whether or not we allow the set $f^{-1}(\Delta)$ to vary:
\begin{itemize}
    \item The first order deformations and obstructions of $[f]$ which preserve $f^{-1}(\Delta)$ are respectively $H^{0}$ and $H^{1}$ of $f^{*}T_{Y}(-\log \Delta)$.
    \item If $f$ is non-constant, then the first order deformations and obstructions of $[f]$ which allow the log structure on $C$ to vary are respectively $H^{0}$ and $H^{1}$ of the log normal sheaf $N^{log}_{f}$ defined via the exact sequence
    \begin{equation*}
        0 \to T_{C}(-\log f^{-1}(\Delta)) \to f^{*}T_{Y}(-\log \Delta) \to N^{log}_{f} \to 0.
    \end{equation*}
\end{itemize}
Recall that a stable log map over a geometric log point is said to be non-degenerate if the log structure on the point is trivial.  In this case the domain of the map is a smooth irreducible curve and its image is not contained in $\Delta$.  We denote the open substack of non-degenerate stable log maps by $\mathscr{M}^{\circ}_{g,\varsigma}(Y,\beta)$.

We will need a criterion determining when a rational stable log map (i.e.~a stable log map whose domain is a smooth rational curve) has globally generated log normal sheaf.

\begin{lemma} \label{lem: normal bundles}
    Let $(Y,\Delta)$ be a log smooth pair.  Fix a big Cartier divisor $H$ on $Y$ and a positive integer $d$.  There is a non-empty open set $U_{H,d,\Delta} \subset Y\setminus \Delta$ such that for any non-degenerate rational stable log map $f : C \to Y$, if $f(C)$ intersects $U_{H,d,\Delta}$ and $H \cdot f_*C \leq d$ then $N^{log}_{f}$ is globally generated.
\end{lemma}

\begin{proof}
The divisor $H$ is $\mathbb{Q}$-linearly equivalent to a sum $A + E$ of an ample $\mathbb{Q}$-Cartier divisor $A$ and an effective $\mathbb{Q}$-Cartier divisor $E$.  Any curve in $Y$ that has non-positive intersection with $H$ must pair negatively with $E$, and must therefore be contained in $E$.  It follows that irreducible rational curves that meet $Y\setminus E$ and have degree at most $d$ against $H$ also have degree at most $d$ against $A$ and therefore form a finite-type family.  Thus, for such curves there are only finitely many possible tuples $\varsigma$ recording the intersections with the boundary.  

Let $M \subset \bigcup_{\varsigma, \beta} \mathscr{M}^{\circ}_{0,\varsigma}(X',\beta)$ denote the finite-type substack parameterizing non-degenerate rational stable log maps whose image meets $Y \setminus E$ and whose degree against $H$ is at most $d$.
The proof is by Noetherian induction on $M$.  If $M$ does not parameterize a dominant set of maps, let $U_{H,d,\Delta}$ be the complement in $Y\setminus E$ of the closure of the sublocus swept out by the curves parameterized by $M$.  If $M$ does parameterize a dominant family, then for a general $f$ the log normal sheaf will be globally generated.  Since global generation of a log normal sheaf is an open condition on $M$, we can replace $M$ by the complement of this open set and conclude by induction.
\end{proof}

\begin{lemma} \label{lemm:loggluing}
    Let $(Y,\Delta)$ be a log smooth pair.  For $i=1,2$ let $f_{i}: C_{i} \to Y$ be non-degenerate non-contracted genus $0$ stable log maps such that $N_{f_{i}}^{log}$ is globally generated.  Then suitably chosen deformations of $f_{1},f_{2}$ can be glued to give a genus $0$ stable log map $f: C \to Y$ such that a general deformation of $f$ will be a non-degenerate non-contracted rational stable log map whose log normal sheaf is globally generated. 
\end{lemma}

\begin{proof}
    Note that a non-degenerate non-contracted rational stable log map $g: C \to Y$ such that $N_{g}^{log}$ is globally generated will deform in a dominant family on $Y$.  Furthermore the global generation of the log normal sheaf is an open property in $\mathscr{M}^{\circ}_{0,\varsigma}(Y,\beta)$. Thus after replacing $f_{1},f_{2}$ by suitably chosen deformations we may ensure that $f_{1}(C_{1})$ and $f_{2}(C_{2})$ intersect at a point $p$ in $Y \backslash \Delta$ while preserving global generation of the log normal sheaf.  We may also ensure that both $f_{1}$ and $f_{2}$ are local immersions at points $q_{1},q_{2}$ mapping to $p$.  Finally, we may ensure that either $f_{1}(C_{1})$ and $f_{2}(C_{2})$ meet transversally at $p$ or
    that all the curves of our family have the same tangent direction at each general point of $Y$. In this latter case, there must be a rank 1 foliation on $Y$ with the curves of the family as algebraic leaves. This shows that $f_1(C_1)$ and $f_2(C_2)$ must be equal to each other. By gluing the domains at $q_{1},q_{2}$ we obtain a stable log map $f$.
    
    First suppose that $f_1(C_1) = f_2(C_2)$.  Then $f$ factors through a covering map $\phi : C \to \PP^1$ followed by a generically injective map $h : \PP^1 \to Y$. 
    We may view $h$ as a log map by equipping $\PP^1$ with the log structure coming from $h^{-1}(\Delta)$.  Because $N_{f_i}^{log}$ is globally generated, the snake lemma applied to the commutative diagram
\[\begin{tikzcd}
	0 & {T_{C_i}(-\log f_i^{-1}(\Delta))} & {f_i^*T_Y(-\log \Delta)} & {N_{f_i}^{log}} & 0 \\
	0 & {\phi^*T_{\mathbb{P}^1}(-\log h^{-1}(\Delta))|_{C_i}} & {(h \circ \phi)^*T_Y(-\log \Delta)|_{C_i}} & {\phi^*N_h^{log}|_{C_i}} & 0
	\arrow[from=1-1, to=1-2]
	\arrow[from=1-2, to=1-3]
	\arrow[from=1-2, to=2-2]
	\arrow[from=1-3, to=1-4]
	\arrow["{\mathrm{id}}", from=1-3, to=2-3]
	\arrow[from=1-4, to=1-5]
	\arrow[from=1-4, to=2-4]
	\arrow[from=2-1, to=2-2]
	\arrow[from=2-2, to=2-3]
	\arrow[from=2-3, to=2-4]
	\arrow[from=2-4, to=2-5]
\end{tikzcd}\]
    shows that $N^{log}_{h}$ is globally generated. Hence $h$ deforms with the expected dimension.  
    Since deformations of $h$ parameterize leaves of a rank one foliation, this expected dimension must be $\dim Y - 1$.  We conclude that $N_{h}^{log} \cong \OO_{\PP^1}^{\oplus \dim Y - 1}$.

    We claim there is a log deformation of $\phi$ that smooths $C$.  Indeed, by the preceding paragraph $h^*(K_Y + \Delta) = K_{\PP^1} + h^{-1}(\Delta)$.  This shows the expected dimension of log deformations of $f_i$ is larger than the expected dimension of log deformations of $\phi|_{C_i}$ by exactly $\dim Y - 1$.  Since $h$, $f_1$, and $f_2$ deform with the expected dimension, the space of log deformations of $\phi$ that do not smooth the node also has the expected dimension.  As this is one less than the expected dimension of the total space of log maps, this proves the existence of a smoothing.  
    
    The deformation of $\phi$ smoothing $C$ induces a deformation $f'$ of $f$ as a log map.  In fact, a dimension count shows that a general deformation of $f$ is the composition of deformations of $\phi$ and $h$.  Thus, a general deformation $g$ of $f$ will be a non-degenerate stable log map.  As before, we see from the snake lemma that $N_{g}^{log}$ will be globally generated.

    Next suppose that $f_1(C_1) \neq f_2(C_2)$.  By construction $f$ is locally an immersion near the node $q$ of $C$ and thus $N^{log}_{f}$ is a well defined sheaf.  Let $N \subset N_f^{log}$ be the subsheaf of the log normal bundle that does not smooth the node.  We have an exact sequence
    $$0 \to N_{f_1}^{log}(-q) \to N \to N_{f_2}^{log} \to 0.$$
    We conclude that $H^{1}(C,N^{log}_{f}(-t))=0$ where $t$ is any point on $C_{2}\setminus\{q\}$.  Since $f$ is unobstructed, a general deformation $g$ of $f$ will be a non-degenerate stable log map. By upper semicontinuity of sheaf cohomology, we see that the log normal sheaf of $g$ will be globally generated.
\end{proof}

\subsection{Proof of Theorem \ref{thm: nicest curves}}

The core idea in the proof of Theorem \ref{thm: nicest curves} is to carefully degenerate a large-degree free curve of higher genus.  By studying the degrees of the various components of the degenerate curve, we will show that $X$ has to contain a rational curve with the desired properties.

\begin{proof}[Proof of Theorem \ref{thm: nicest curves}]
The proof is by strong induction on $\Theta$.

\textbf{Step 1:} It suffices to prove the statement after blowing $X'$ up further as this operation can only decrease $-K_{X'} \cdot s_*\mathbb{P}^{1}$.  Thus we may assume that the exceptional locus of $X' \to X$ is a simple normal crossings divisor.  Let $\{ a_i \}_{i=1}^{t}$ be the discrepancies of the $\pi$-exceptional divisors $\{ E_{i}\}_{i=1}^{t}$ on $X'$.  Since each discrepancy is a rational number, there is a smallest positive number that may be written as a $\mathbb{Z}$-linear combination of them.  Let $a_{min} > 0$ be the minimum of this number and 1.

Next we set some constants.  Let $\delta = \max (0, \exctowerheight - 2)$.  By \cite[Theorem~1.3]{JLR25}, there is a smooth curve $C$ such that for any $r \geq 0$ the variety $X$ admits $r$-free curves $s: C \to X$.  By possibly replacing $C$ by a finite cover we may suppose that the genus $g$ of $C$ is at least $2$.  Choose positive integers $k$ and $\elle'$ so that
\begin{equation}\label{eq: def e'}
    k > \max\{(g+2)n + 2\delta, n(\delta+2g)+2 \} \hspace{.5cm} \text{ and } \hspace{.5cm} a_{\min}(k+\elle' - 2g) > 3gn+2g\delta + \delta + k.
\end{equation}

\textbf{Step 2:} Now we pick a family of large-degree free curves.  Select an irreducible component $M$ of $\Mor(C,X)$ that generically parameterizes $r$-free curves, where $r$ is chosen large enough to ensure that $\dim M \geq (n+\delta)(k+\elle') + k$.  By construction general curves parameterized by $M$ do not meet the image of the exceptional locus of $\pi : X' \to X$.  Let $M'$ be the irreducible component of $\Mor(C,X')$ that parameterizes the lifts of general maps in $M$ to $X'$. 

We let $d$ denote the $-K_{X}$-degree of the curves parameterized by $M$.  Note that $d$ is also the $-K_{X'}$-degree of the curves parameterized by $M'$.  Since $M$ has the expected dimension we have
\begin{equation*}
\dim M' = \dim M = d + (1-g)n.
\end{equation*}
We will degenerate maps in $M'$ to a stable map and analyze its free components to prove Theorem \ref{thm: nicest curves}.

\textbf{Step 3:}  Before selecting a degeneration, we must describe a closed set $Z \subset X'$ that contains all components of degenerations of maps $M'$ such that the restricted tangent bundle has undesirable behavior.  To do so,
let $W \subset \Mor(\mathbb{P}^1, X')$ and $W' \subset \Mor(C, X')$ be the union of all components parameterizing maps whose degree against some fixed ample divisor is at most the degree of a curve parametrized by $M'$.  Let $\mathcal{Q} = \{q_1, \ldots , q_{\dim M}\} \subset C$ be a general set of points.  We define $Z_{1} \subsetneq X'$ to be the closure of the locus swept out by $\sigma(\mathcal{Q})$ as $\sigma$ varies over all maps in $W'$ for which $\sigma^*T_{X'}$ is not generically globally generated.  We define $Z_{2} \subsetneq X'$ to be the closure of the locus swept out by all maps $\sigma$ in $W$ for which $\sigma^{*}T_{X'}$ is not globally generated.  Next, equip $X'$ with the log structure associated to the divisor $\cup_{i} E_{i}$ and set $H = -\pi^*K_X$.  We define $Z_3 \subset X'$ to be the complement of the non-empty open subset $U_{H, d, \cup_i E_i} \subset X'$ from Lemma \ref{lem: normal bundles}.  
Finally, we let $Z \subsetneq X'$ be the proper closed subset $Z_{1} \cup Z_{2} \cup Z_{3}$.

\textbf{Step 4:} We next construct a subfamily of $M'$ of dimension at least $k$ that parameterizes maps sending lots of the general points $\{q_1, \ldots , q_{\dim M} \}$ to fixed points of $X'$ and whose jets at these points satisfy certain conditions.  
We achieve this by inductively constructing exceptional towers over $X'$ and considering strict transforms of curves that meet the top of each tower. For the first step, note that as we vary over all maps in $M'$ the images of $q_1 \in \mathcal{Q}$ sweep out all of $X'$, which has dimension $n_1 = n$.  Fix a general point $p_{1} \in X'$.  Let $Y_1 = X'$ and let 
$$Y_{1,\delta} \xrightarrow{\phi_{1,\delta}} Y_{1,\delta - 1} \xrightarrow{\phi_{1, \delta - 1}} \ldots \xrightarrow{\phi_{1,1}} Y_{1,0} \xrightarrow{\phi_{1,0}} Y_1$$
be a general exceptional tower based at $p_1$. Let $\Phi(Y_1, p_1, \delta)$ denote this exceptional tower and let $E_{1,j} \subset Y_{1,j}$ denote the exceptional divisor of $\phi_{1,j}$.  By Lemma \ref{lem: climb tower morphism scheme}(3) with the sheaf $\mathcal{F}$ given by the restricted tangent bundle, there is a subvariety $M_1 \subset M'$ of dimension $\dim M' - n_1 - \delta$ that parameterizes maps whose strict transforms send $q_1$ into $E_{1,\delta}$.

Now suppose we have constructed an exceptional tower $\Phi(Y_i, p_i, \delta)$ and a family of maps $M_i$ of dimension at least $n + k + \delta$.  By replacing $M_i$ with an open subset and taking strict transforms, we may consider $M_i$ as a family of maps to the final blow-up $Y_{i, \delta}$.  Let $Y_{i+1} = Y_{i,\delta}$ and consider the point $q_{i+1} \in \mathcal{Q}$.  As we vary among maps in $M_i$, the point $q_{i+1}$ will sweep out a subvariety $G_{i+1} \subset Y_{i+1}$ of dimension $n_{i+1} \leq n$.  

We claim that $G_{i+1}$ is not contained in the preimage of $Z$ under the natural map $Y_{i+1} \to X'$.  To show this, we observe that the family $M_i \subset M'$ depends only on the previously constructed exceptional towers, and as we vary these towers the subvariety $M_i$ sweeps out a dense subset of $M'$.  If $G_{i+1}$ lay inside the preimage of $Z$, then, by generality of the choices of exceptional towers in our construction, the image of $q_{i+1}$ under every map in $M'$ would lie in $Z$, contradicting the fact that $M$ parameterizes free curves.
Similarly, one can show $G_{i+1}$ does not lie in the exceptional locus of $Y_{i+1} \to X'$.  

Choose a general point $p_{i+1} \in G_{i+1}$ and let 
$$Y_{i+1,\delta} \xrightarrow{\phi_{i+1,\delta}} Y_{i+1,\delta - 1} \xrightarrow{\phi_{i+1, \delta - 1}} \ldots \xrightarrow{\phi_{i+1,1}} Y_{i+1,0} \xrightarrow{\phi_{i+1,0}} Y_{i+1}$$
be a general exceptional tower based at $p_{i+1}$.  As before, let $\Phi(Y_{i+1}, p_{i+1},\delta)$ denote this exceptional tower and let $E_{i+1,j} \subset Y_{i+1,j}$ denote the exceptional divisor of $\phi_{i+1,j}$.  Consider the log structure on $Y_{i+1,j}$ given by the total exceptional divisor for the map to $X'$.  The hypotheses of Lemma \ref{lem: climb tower morphism scheme}(3) hold for $\mathcal{F}$ given by the restriction of the log tangent sheaf provided that $\dim M_i \geq n(\delta+2g)+2$.  Thus, when this inequality holds, by Lemma \ref{lem: climb tower morphism scheme}(3) there is a subvariety $M_{i+1} \subset M_i$ of dimension $\dim M_i - n_{i+1} - \delta$ that parameterizes maps whose strict transforms send $q_{i+1}$ into $E_{i+1,\delta}$. 

Since $\dim M' \geq (n+\delta)(k+\elle') + k$, this process will continue for $k+\elle$ steps for some integer $\elle$ satisfying $\elle' \leq \elle \leq \dim M' - k$.  In the end we obtain a parameter space $M_{k+\elle}$ of curves in $X'$ whose dimension is at least $k$ but less than $k+n + \delta$. Observe that from our construction, we have points $p_1, \dots, p_{k+\elle} \in X' \setminus Z$ such that every $s \in M_{k+\elle}$ satisfies $s(q_i) = p_i$.
The construction guarantees that $\dim M'$ is between $k + \sum_{j=1}^{k+\elle} (n_j + \delta)$ and $ k+n + \delta +\sum_{j=1}^{k+\elle} (n_j + \delta)$.  Because $\dim M' = d + (1-g)n$, we obtain
\begin{equation}\label{eq: bound on d}
    k+(g-1)n + \sum_{j=1}^{k+\elle} (n_j + \delta) \leq d \leq k+gn + \delta + \sum_{j=1}^{k+\elle} (n_j + \delta).
\end{equation}

\textbf{Step 5:} Now we degenerate our chosen family of large-degree curves. Consider $M_{k+\elle}$ as a family of curves on $Y_{k + \elle, \delta}$ by taking strict transforms of general elements of the family. Let $\overline{M}_{k+\elle}$ be the closure in $\overline{\mathcal{M}}_{g,k + \elle}(Y_{k + \elle, \delta})$ of the locus of pointed maps $s : (C,q_1, \ldots , q_{k+ \elle}) \to Y_{k+\elle, \delta}$ parameterized by $M_{k+\elle}$.

Let $\psi : Y_{k+ \elle, \delta} \to X'$ be the natural birational map.  By applying \cite[Lemma 2.1]{JLR25b} to the family of maps to $X'$ obtained by composing maps in $\overline{M}_{k+\elle}$ with $\psi$, we see that $\overline{M}_{k+\elle}$ parameterizes a stable map $s': (C',q'_{1},\ldots,q'_{k+\elle}) \to Y_{k + \elle, \delta}$ where $C'$ consists of:
\begin{itemize}
\item an irreducible component $C_{0}$ isomorphic to $C$, 
\item $k$ connected components of $C' \setminus C_{0}$ containing the first $k$ marked points, that are trees of rational curves which are not contracted by the map to $X'$, and
\item other connected components of $C' \setminus C_{0}$, that are trees of rational curves which may or may not contain marked points or be contracted by the map to $X'$.
\end{itemize}
Let $\Delta \subset Y_{k + \elle, \delta}$ be the exceptional locus of $\psi$.  Let $J \subset \{1,\ldots,k+\elle\}$ be the set of indices such that there is a tree of rational curves $T_j \subset C'\setminus C_{0}$ that contains $q'_{j}$ and is not contracted to a point by $\psi \circ s'$.   We select each $T_j$ so that it is a connected component of $C' \setminus C_0$.  Observe that $|J| \geq k$ because $\{1,\ldots , k\} \subset J$.

\textbf{Step 6:} We wish to carefully study the possibilities for the $T_j$. We adopt notation from Lemma \ref{lem-numericalExceptionalTowers}:
\begin{itemize}
    \item for each $i \leq k + \elle$, let $C_{i,exc} \subset C'$ be the connected component of $s'^{-1}(\Delta)$ that contains $q_i'$, which is either a curve or the point $q_i'$ itself;
    \item let $\{C_{\ell_i}\}_{\ell_i \in \mathcal{L}_i}$ be the set of irreducible components of $C'$ that meet $C_{i,exc}$ but are not contained in it;
    \item for $\ell_i \in \mathcal{L}_i$ such that $C_{\ell_i}$ is rational, set $\mathcal{Q}_{\ell_i} = C_{\ell_i} \cap C_{i,exc}$.
\end{itemize}
Observe that when $C_{\ell_i}$ is rational, $\mathcal{Q}_{\ell_i}$ contains exactly one point and $C_{\ell_i}$ cannot appear in any set  $\{C_{\ell_{i'}}\}_{\ell_{i'} \in \mathcal{L}_{i'}}$ with $i' \neq i$.  On the other hand, the component $C_0 \subset C'$ can appear in multiple such sets.  Let $H  \subset \{1,\ldots,k+\elle\}$ be the set of indices $h$ such that $C_0 \in \{C_{\ell_h}\}_{\ell_h \in \mathcal{L}_h}$.  For each $h \in H$, we regard $q_h \in C$ as the point $C_0 \cap C_{h,exc}$ and set $\mathcal{Q}_0 = \{q_h\}_{h \in H}$. Note that $J \cup H = \{1,\ldots,k+\elle\}$ because if $C_{0}$ does not meet $C_{i,exc}$ then there is a tree of rational curves containing $q_{i}'$ that is not entirely contracted by $\psi \circ s'$ (so $i \in J$).  Furthermore note that $\psi \circ s'(q_{h}) = \psi \circ s'(q'_{h}) = p_{h}$ for every $h \in H$.

We will construct the desired map $s : \PP^1 \to X'$ from the rational curves $s'|_{C_{\ell_i}}$.   To begin with, we outline some useful consequences from Section \ref{sec: exceptional towers}.  For any distinct $i,h \in H$, $\psi \circ s'(q_h) \neq p_i$ and so $\mult_{q_h}(s'|_{C_0}, \widetilde{E}_{i,m}) = 0$ for all $0 \leq m \leq \delta$.  Thus, for each $i \leq k + \elle$ Lemma \ref{lem-numericalExceptionalTowers} guarantees
\begin{equation}\label{eq: multiplicities}
    \sum_{\ell \in \mathcal{L}_i} \sum_{q_{j,\ell} \in \mathcal{Q}_{\ell}} \left(\sum_{m = 0}^\delta (m+1) \cdot \mult_{q_{j,\ell}}\big(s'|_{C_{\ell}}, \widetilde{E}_{i,m}\big)\right) \geq \delta +1.
\end{equation}
  
For each $i \leq k + \elle$, we define $\beta_i$ to be
\begin{itemize}
    \item $0$ if $i \notin H$, and
    \item the minimum of $\delta + 1$ and $\sum_{m = 0}^\delta (m+1) \cdot \mult_{q_i}(s'|_{C_0}, \widetilde{E}_{i,m})$ if $i \in H$.
\end{itemize} 
Note that if $i \in H \backslash J$ then $\beta_{i} = \delta+1$ by Inequality \eqref{eq: multiplicities}.

For each $j \in J$, let $R_j$ be the union of all rational curves in $\{C_{\ell_j}\}_{\ell_j \in \mathcal{L}_j}$.  Since the point $p_{j}$ is not contained in $Z$, the restriction of $(\psi \circ s')^{*}T_{X'}$ to $R_{j}$ is globally generated.  Thus by Corollary \ref{cor: rational curves degree} and the above equation,
\begin{equation}\label{eq: Rj deg bound}
    -K_{X'} \cdot (\psi \circ s')_*R_j \geq 2 + \delta - \beta_j \geq \exctowerheight - \beta_j.
\end{equation}

\textbf{Step 7:} The following claim about the properties of some $R_j$ is the crux of this proof.

\begin{claim}\label{claim: finding Rj}
    There exists a $j \in J$ such that one of the following statements about $R_j$ holds:
    \begin{enumerate}
        \item $\beta_j = 0$ (i.e.\ $j \notin H$), 
        $-K_{X'} \cdot (\psi \circ s')_* R_j \geq 2 + \delta$, $-K_{X} \cdot (\pi \circ \psi \circ s')_* R_j \leq \exctowerheight + n$, $(K_{X'} - \pi^*K_X) \cdot (\psi \circ s')_* R_j \leq n-2$, and $\pi \circ \psi \circ s'(R_j)$ meets the singular locus of $X$ at most $n - 2$ times; or
        \item $\beta_j > 0$ (i.e.\ $j \in H$), $-K_{X'} \cdot (\psi \circ s')_* R_j = 2 + \delta - \beta_j$, and $\pi \circ \psi \circ s'(R_j)$ lies in the smooth locus of $X$. 
    \end{enumerate}
\end{claim}

Before proving Claim \ref{claim: finding Rj}, we explain how to finish the proof after having found such a $j \in J$.  First assume that $R_j$ satisfies \ref{claim: finding Rj}(1).  We may repeatedly apply Lemma \ref{lemm:loggluing} to glue and smooth its components while keeping the total number of intersection points with the SNC divisor $\cup_{i} E_{i}$ fixed.  Since the resulting rational curve $s: \mathbb{P}^{1} \to X'$ has globally generated log normal sheaf, we may ensure that the image of $s$ contains a general point $p$ of $X$.  Finally, the desired intersection properties for $s$ are implied by \ref{claim: finding Rj}(1).

Suppose instead that $R_j$ satisfies \ref{claim: finding Rj}(2). Because $R_j$ lies in the smooth locus of $X$, its $-K_X$-degree is the same as its $-K_{X'}$-degree.  Observe that if $\exctowerheight \leq 2$, then $\delta = 0$. As every component of $R_j$ has $-K_{X'}$-degree at least 2, we conclude that if $R_j$ satisfies \ref{claim: finding Rj}(2), then $\exctowerheight > 2$.  Thus $0 < \beta_j \leq \delta + 1 < \exctowerheight$, and by strong induction there exists a rational curve $s'' : R_j' \to X'$ satisfying the constraints of Theorem \ref{thm: nicest curves} associated to $\Theta' = \beta_j$.  By gluing and smoothing the maps $s'|_{R_j}$ and $s''$ we obtain the desired rational curve.
\end{proof}

\begin{proof}[Proof of Claim \ref{claim: finding Rj}:]
To derive a contradiction, assume that for each $j \in J$ neither \ref{claim: finding Rj}(1) nor \ref{claim: finding Rj}(2) hold.  We break our argument into two cases, depending on the size of $|J \setminus H|$.  In both cases we will translate the failure of \ref{claim: finding Rj}(1-2) into a lower bound on the degree of each $T_j$, and of other subcurves of $C'$ where necessary.  Comparing to the degree of our original curve, we obtain a contradiction.

As it is used in both cases, first we prove the following lower bound. Suppose $j$ is an index such that $\beta_j = 0$.  
 If such a component exists, let $S_j \subset \overline{T_j \setminus R_j}$ be 
a free component whose image is not contained in the preimage of $Z$; otherwise, let $S_{j}$ be the empty set.    
We claim that the failure of \ref{claim: finding Rj}(1) for $j$ implies 
    \begin{equation}\label{eq: lower bound on -KX-degree}
        -K_X \cdot (\pi \circ \psi \circ s')_*T_j \geq a_{\min} + \min\big\{ n + \delta + 1, -K_{X'} \cdot (\psi \circ s')_*(R_j \cup S_j) + n - 2\big\} \geq \delta + n + a_{\min}.      
    \end{equation}
    
    We prove this inequality by conditioning on which hypothesis of \ref{claim: finding Rj}(1) fails.  We have already shown that $-K_{X'} \cdot (\psi \circ s')_*R_j \geq 2 + \delta$ in Inequality \eqref{eq: Rj deg bound}.  Thus, one of the following hypotheses of \ref{claim: finding Rj}(1) must be false: 
    \begin{enumerate}
        \item $-K_{X} \cdot (\pi \circ \psi \circ s')_* R_j \leq \exctowerheight + n$;
        \item $(K_{X'} - \pi^* K_{X}) \cdot (\psi \circ s')_* R_j \leq n - 2$; or
        \item $\psi \circ s'(R_j)$ meets the $\pi$-exceptional locus in at most $n-2$ points.
    \end{enumerate}
    We remark that the bound $-K_{X'} \cdot (\psi \circ s')_*R_j \geq 2 + \delta$ also verifies the rightmost inequality in \eqref{eq: lower bound on -KX-degree}.  Furthermore, if $-K_X \cdot (\pi \circ \psi \circ s')_*T_j$ is greater than an integer $t \in \mathbb{Z}$, then by the definition of $a_{\min}$ it must be at least $t + a_{\min}$.  Thus, as $-K_{X'} \cdot (\psi \circ s')_*(R_j \cup S_j)$ is an integer it suffices to show $-K_X \cdot (\pi \circ \psi \circ s')_*T_j > \min\big\{ n + \delta + 1, -K_{X'} \cdot (\psi \circ s')_*(R_j \cup S_j) + n - 2\big\}$.

    First suppose $-K_{X} \cdot (\pi \circ \psi \circ s')_* R_j > \exctowerheight + n$.  Since $\exctowerheight \geq \delta + 1$ and $-K_{X}$ is nef, we obtain $-K_{X} \cdot (\pi \circ \psi \circ s')_* T_j > n + \delta + 1$, as desired.  Similarly, if $(K_{X'} - \pi^* K_{X}) \cdot (\psi \circ s')_* R_j > n - 2$, then because $X$ is terminal and $(\psi \circ s')_* S_j$ is a nef curve, $(K_{X'} - \pi^* K_{X}) \cdot (\psi \circ s')_* S_j \geq 0$. Thus we obtain $-K_{X} \cdot (\pi \circ \psi \circ s')_* T_j > -K_{X'} \cdot (\psi \circ s')_*(R_j \cup S_j) + n - 2$, as desired.

    Lastly, suppose $\psi \circ s'(R_j)$ meets the $\pi$-exceptional locus in at least $n-1$ points.  Consider the connected subcurve $T_j' \subset T_j$ containing $R_j \cup S_j$ and all connected components of $T_j \setminus (R_j\cup S_j)$ that are not incident to $C_0$.  Observe that the $-K_{X}$-degree of $(\pi \circ \psi \circ s')_* T_j$ is at least the $-K_{X}$-degree of $(\pi \circ \psi \circ s')_* T_j'$ because $-K_X$ is nef.  Thus, it suffices to prove our claim for $T_j'$.  To do so we consider a one-dimensional family of curves $\varphi : \mathcal{C} \to B$ that admits a morphism $s : \mathcal{C} \to X'$ with the following properties:
    \begin{itemize}
        \item $\mathcal{C}$ is a surface with $\mathbb{Q}$-factorial singularities, 
        \item every fiber of $\varphi$ is nodal,
        \item a general fiber $\mathcal{C}_b$ of $\varphi$ is irreducible and the strict transform to $Y_{k + \elle, \delta}$ of the map $s_b : \mathcal{C}_b \to X'$ is parameterized by $M_{k + \elle}$, and 
        \item the central fiber of $\varphi$ is isomorphic to the map $\psi \circ s' : C' \to X'$.
    \end{itemize}
    Consider the union $E$ of all components of $s^{-1}(\cup_i E_i)$ supported on the central fiber of $\varphi$, where $\cup_i E_i$ is the exceptional locus of $\pi : X' \to X$.  Every irreducible component of $E$ that meets $T_{j}'$ is either a component of $T_{j}'$ or meets $T_{j}'$ at the node $q \in C'$ joining $T_j'$ to $C'\setminus T_j'$.  Thus $-K_{X'} \cdot (\psi \circ s')_{*}T_j' \leq -K_X \cdot (\pi \circ \psi \circ s')_*T_j'$ and strict inequality holds whenever this node $q$ maps into the exceptional locus of $\pi$. 

    By \cite[Lemma~4.3]{JLR25}, every connected component of $T_j' \setminus (R_j\cup S_j)$ pairs nonnegatively with $-K_{X'}$.  From our hypothesis that $\psi \circ s'(R_j)$ meets the $\pi$-exceptional locus in at least $n-1$ points, \cite[Lemma~4.2]{JLR25} shows there are at least $n-2$ connected components of $T_j' \setminus (R_j\cup S_j)$ whose image in $X'$ meets the $\pi$-exceptional locus.  By \cite[Lemma~4.3]{JLR25}, the $-K_{X'}$-degree of each of these connected components must be at least $1$.  Moreover, either there are at least $n-1$ such connected components (contributing at least $n-1$ to the $-K_{X'}$-degree), or the node $q \in C'$ lies in $R_j$ and maps into the $\pi$-exceptional locus.  In the former case we obtain $- K_{X} \cdot (\pi \circ \psi \circ s')_* T_j' \geq -K_{X'} \cdot (\psi \circ s')_*(R_j \cup S_j) + n - 1$.  In the latter case, we have 
    $$-K_X \cdot (\pi \circ \psi \circ s')_*T_j' > -K_{X'} \cdot (\psi \circ s')_*T_j' \geq -K_{X'} \cdot (\psi \circ s')_*(R_j \cup S_j) + n - 2.$$  
    In both cases this proves that Inequality \eqref{eq: lower bound on -KX-degree} holds whenever $\beta_j = 0$.

    We now return to the proof of Claim \ref{claim: finding Rj} by studying two cases depending on the size of $|J \setminus H|$.

    \textbf{Case 1:} Suppose that $|J \setminus H| \geq k+\elle-2g$.  Recall that $\beta_j = 0$ for each $j \in J \setminus H$.  We next sum over the degrees of $(\pi \circ \psi \circ s')_* T_j$ for $j \in J \setminus H$.  Since $-K_X$ is nef and $e' \leq e$, the bounds provided by  Inequalities \eqref{eq: def e'}, \eqref{eq: bound on d}, and \eqref{eq: lower bound on -KX-degree} yield
\begin{align*}
d & \geq (n+\delta + a_{\min})(k+\elle-2g) \\ 
& \geq \left( \sum_{i=1}^{\elle+k} (n_i + \delta) \right) -2g(n + \delta) + a_{\min}(k+\elle-2g) \\
& > \left( \sum_{i=1}^{\elle+k} (n_i + \delta) \right) + k + gn + \delta \\
& \geq d.
\end{align*}
This is a contradiction. 

\textbf{Case 2:}  Suppose instead that $|J \setminus H| < k+\elle-2g$. In this case, we do a careful degree count of the various components of $C'$ to arrive at a contradiction. We start by observing that $\psi \circ s'$ cannot contract $C_0$.  Using the definition of $Z$, we see that $(\psi \circ s')^*T_{X'}|_{C_0}$ is generically globally generated since $(\psi \circ s')(C_0)$ contains the points $s'(q_h) = p_h \notin Z$ for each $h \in H$.  In particular, $-K_{X'} \cdot (\psi \circ s')_*C_0 \geq 0$.  

\textbf{Step 1, $\mathbf{j \in J \setminus H}$:}
Because $-K_{X'} \cdot (\psi \circ s')_*C_0 \geq 0$, \cite[Lemma~4.3]{JLR25} proves that the connected components $T_j$ of $C' \setminus C_0$ satisfy $(-K_{X'} + \pi^*K_X) \cdot (\psi \circ s')_*T_j \geq 0$.  Thus, Inequality (\ref{eq: lower bound on -KX-degree}) implies $-K_{X'} \cdot (\psi \circ s')_* T_j \geq n + \delta + 1$ for each $j \in J \setminus  H$ because $-K_{X'}$ is Cartier and $a_{\min} > 0$.

\textbf{Step 2, $\mathbf{j \in J \cap H}$:} We next show that $-K_{X'} \cdot (\psi \circ s')_* T_j \geq 3 + \delta - \beta_j$ for each $j \in J \cap H$.  By Inequality \eqref{eq: Rj deg bound} the ($-K_{X'}$)-degree of $R_j$ is positive for each $j \in J$.  From \cite[Lemma~4.3]{JLR25} it follows that each connected component of $C'\setminus (C_0 \cup_{j \in J} R_j)$ has nonnegative $(-K_{X'})$-degree.  Thus, for each $j \in J$ we have $-K_{X'} \cdot (\psi \circ s')_*T_j \geq -K_{X'} \cdot (\psi \circ s')_*R_j$.  If $\psi \circ s' (R_j)$ is disjoint from the $\pi$-exceptional locus, then we obtain
$$-K_{X'} \cdot (\psi \circ s')_* T_j \geq -K_{X'} \cdot (\psi \circ s')_{*}R_{j} \geq 3 + \delta - \beta_j$$
from Inequality \eqref{eq: Rj deg bound} and the assumption that $R_j$ does not satisfy \ref{claim: finding Rj}(2).  If instead $\psi \circ s' (R_j)$ meets the $\pi$-exceptional locus, then by \cite[Lemma~4.2]{JLR25} there is a connected component of $T_j \setminus R_j$ whose image in $X'$ meets the $\pi$-exceptional locus.  By \cite[Lemma~4.3]{JLR25}, the $-K_{X'}$-degree of such a connected component must be at least $1$.  Since each connected component of $C'\setminus (C_0 \cup_{j \in J} R_j)$ has nonnegative $(-K_{X'})$-degree this proves $-K_{X'} \cdot (\psi \circ s')_* T_j \geq -K_{X'} \cdot (\psi \circ s')_{*}R_{j} + 1 \geq 3 + \delta - \beta_j$ by Inequality (\ref{eq: Rj deg bound}) once more.

\textbf{Step 3, $\mathbf{C_0}$:}
We next bound the $-K_{X'}$-degree of $(\psi \circ s')_*C_0$.  
Let $s'_i : C_0 \to Y_i$ denote the composition of $s'|_{C_0}$ with the natural map $Y_{k + \elle, \delta} \to Y_i  = Y_{i-1, \delta}$.  Furthermore, let $\Delta_i \subset Y_i$ be the total exceptional locus of $Y_i \to X'$.  Set $V_0 = (\psi \circ s')^*T_{X'}|_{C_0}$ and recall that $V_{0}$ is generically globally generated because $s'(q_h) \notin Z$ for each $h \in H$.   Let $V_i \subset s_i'^*T_{Y_i}(-\log_{(s'_i, \mathcal{Q}_0)} \Delta_i)$ be the maximal generically globally generated subbundle.  Let $i_H$ be the largest integer in $H$.  For each $i \in H$, the rank of $V_i$ must be at least $n_i$ because the image of $q_i \in C_0 \cong C$ is $p_{i}$, which is general within the $n_i$-dimensional variety $G_i \subset Y_i$.  In particular, for all $i \leq i_H$, $\rk(V_i) \geq n_{i_H} > 0$.  Using this fact, and that the pullback of $E_{i,m}$ to $Y_{k + \elle, \delta}$ is $\widetilde{E}_{i,m} + \ldots + \widetilde{E}_{i,\delta}$, we may rewrite the inequality in Proposition \ref{prop:lower_bound_exceptional_tower} to compare it to the definition of $\beta_i$.  More specifically, write $s_{i,m}' : C_0 \to Y_{i,m}$ for the strict transform of $s_i'$.  Then since $\mult_{q_i}(s'_{i,m}, E_{i,m}) = \sum_{h = m}^\delta \mult_{q_i}(s'|_{C_0}, \widetilde{E}_{i,h})$ and  $\rk(V_{i+1}) \cdot \mult_{q_i}(s_{i,0}, E_{i,0}) \geq \mult_{q_i}(s_{i,0}, E_{i,0}) + n_{i+1} - 1$ for all $i \in H\setminus\{i_H\}$, from Proposition \ref{prop:lower_bound_exceptional_tower} we obtain 
$$c_1(V_{i}) - c_1(V_{i+1}) \geq n_{i + 1} + \beta_i - 1 \qquad \text{ for all } \qquad i \in H \setminus \{i_H\}.$$
Recalling that $\beta_i = \delta + 1$ if $i \notin J$ and summing over all $i \in H\setminus\{i_H\}$ yields 
$$c_1(V_0) - c_1(V_{k + \elle}) \geq \left(\sum_{i \notin J} (n_i + \delta) + \sum_{j \in J \cap H} (n_j + \beta_j - 1)\right) - n - \delta$$
where we have used the fact that the contribution to the sum coming from the index $i_{H}$ is at most $n+\delta$.
To obtain a final bound on $c_1(V_0)$, we will use the behavior of the stable map as we vary certain $q_i \in C$ to bound $c_1(V_{k+ \elle})$ below.

Let $J' \subset J \setminus H$ be the subset of indices $j$ for which $n_j = n$.  For $j \in J'$, as we vary $q_j \in C$ there are two possibilities: either the curve $s'(C_0)$ changes, or $s'(C_0)$ remains fixed. Let $I$ be the set of indices $i \in J'$ corresponding to $q_i$ such that $s'(C_0)$ moves when we vary $q_{i}$. Thus, $s'(C_0)$ moves in a family of dimension at least $|I|$; in other words, 
$$h^0(C_0, V_{k+ \elle}) = h^0(C_0, s_{k+ \elle}'^*T_{Y_{k+ \elle}}(-\log_{(s'_{k+ \elle}, \mathcal{Q}_0)} \Delta_{k+ \elle})) \geq |I|.$$
Because $V_{k + \elle}$ is generically globally generated, this implies $c_1(V_{k + \elle}) \geq |I| - n$.  Using the earlier bound for $c_1(V_0) - c_1(V_{k + \elle})$, we see the $-K_{X'}$-degree of $(\psi \circ s')_*C_0$ is at least 
$$|I| - 2n - \delta + \sum_{i \notin J} (n_i + \delta) + \sum_{j \in J \cap H} (n_j + \beta_j - 1).$$

\textbf{Step 4, $\mathbf{j \in J \setminus (I\cup H)}$:} We now show that for all $j \in J \setminus (I \cup H)$, the $-K_{X'}$-degree of the image of $T_j$ is at least $n_j+\delta + 2$.  For $j \notin J' \setminus I$, this follows from Step 1 because $n_j < n$.  Suppose instead that $j \in J' \setminus I$.  Varying $q_j \in C$ induces a nontrivial deformation of $\psi \circ s'(T_j)$.  Observe that a general point of $C_0$ is mapped to a point in $X'$ not contained in $Z$.  Since the map $\psi \circ s'|_{C_0}$ does not change as we vary $q_j$, for a general choice of $q_j$ the irreducible component of $\psi\circ s'(T_j)$ meeting $\psi \circ s'(C_0)$ along the image of $q_j$ must be free.  

Either there is a free component of $S_j \subset \overline{T_j\setminus R_j}$ that meets $C_0$ along $q_j$, or $R_j$ itself has a nontrivial deformation that fixes the intersection multiplicity of each component with $\Delta$.  
In the former case, the total $-K_{X'}$-degree of $R_j \cup S_j$ must be at least $4 + \delta$: $2$ for the varying free component $S_j \subset \overline{T_j \setminus R_j}$ and $2+\delta$ from $R_{j}$ by Inequality \eqref{eq: Rj deg bound} (since $\beta_{j}=0$).  In the latter case, there must be a component $C_{\ell_i} \subset R_j$ such that $s'^*T_{Y_{k + \elle, \delta}}(-\log_{(s', \mathcal{Q}_{\ell_i})} \Delta)|_{C_{\ell_i}}$ has a global section that does not correspond to an infinitesimal automorphism.  Arguing as in the proof of Corollary \ref{cor: rational curves degree}, we see from Equation \eqref{eq: multiplicities} that $-K_{X'} \cdot (\psi \circ s')_*R_j \geq 3 + \delta$ because the maximal globally generated subsheaf of $s'^*T_{Y_{k + \elle, \delta}}(-\log_{(s', \mathcal{Q}_{\ell_i})} \Delta)|_{C_{\ell_i}}$ has rank at least 1 and degree at least $2$.  In both cases, reasoning as in the proof of Step 1 we see that for $j \in J' \setminus I$ the $-K_{X'}$-degree of the image of $T_j$ is at least $n+\delta + 2$.

\textbf{Step 5, $\mathbf{j \in I}$:}
Finally, recall that $I \subset J \backslash H$ and thus for $i \in I$ the $-K_{X'}$-degree of the image of $T_{i}$ is at least $n+\delta+1$ as established by the first paragraph of Case 2.  

\textbf{Step 6, summing the degrees:} Thus, the total $-K_{X'}$-degree $d$ must be at least 

\begin{align*}
d \geq & \deg \psi \circ s'(C_0) + \sum_{j \in J \cap H} \deg \psi \circ s'(T_j) + \sum_{i \in I} \deg \psi \circ s'(T_i) + \sum_{j \in J \setminus (I \cup H)} \deg \psi \circ s'(T_j)\\
  \geq & \left(|I| - 2n - \delta + \sum_{j \notin J} (n_j + \delta) + \sum_{j \in J \cap H} (n_j + \beta_j - 1) \right) + \sum_{j \in J \cap H} (3 + \delta - \beta_j) \\
 &  +  \sum_{i \in I} (n + \delta + 1) + \sum_{j \in J \setminus (I \cup H)} (n_j + \delta + 2)\\
\geq & \sum_{j \notin J} (n_j + \delta) + \sum_{j \in J}(n_j + \delta + 2) - 2n - \delta.
\end{align*}
Because $k > (g+2)n + 2\delta$ and $|J| \geq k$, the last line simplifies to 
\begin{equation*}
 \left( \sum_{j=1}^{k+\elle} (n_j + \delta) \right) +2|J| - 2n - \delta 
 > \left( \sum_{j=1}^{k+\elle} (n_j + \delta) \right) + k + gn + \delta 
 \geq d
\end{equation*}
by Inequality (\ref{eq: bound on d}). This is impossible, finishing our proof of Claim \ref{claim: finding Rj} and Theorem \ref{thm: nicest curves}. 
\end{proof}

\section{Free rational curves} \label{sect:freecurves}

We next leverage Theorem \ref{theo:exctowerbreaking} and Theorem \ref{thm: nicest curves} to prove the existence of free rational curves in the smooth locus of certain types of terminal Fano varieties.  As explained in the introduction, this proves new cases of a conjecture of \cite{KM99} predicting that a terminal (or even klt) Fano variety has free rational curves in its smooth locus.  Previously this conjecture was known for klt Fano surfaces by \cite{KM99}, for LCI terminal varieties by work of Koll\'ar (see \cite[Remark 4.1]{JLR25}), and for terminal threefolds by \cite[Theorem 1.4]{JLR25}.

This section has three subsections, each proving the existence of free curves in a different situation. Before addressing the individual cases, we will need a couple lemmas.  The first constructs families of free curves such that the evaluation map on the universal family has connected fibers.  In particular, we can apply the Grauert-M\"ulich Theorem of \cite{LRT24} to such families.

\begin{lemma} \label{lem: producing irreducible fibers}
Let $X$ be a projective variety.  Suppose $s : C \to X$ is an $r$-free curve of genus $g$.  Then for some positive integer $m$, there exists an $r$-free curve $s' : C' \to X$ such that 
\begin{enumerate}
\item $s'_*C' =  m s_*C$;
\item $C'$ is a smooth curve of genus $m g$; and
\item if $U'$ denotes the universal family over the component of $\overline{\mathcal{M}}_{mg,0}(X,m\beta)$ containing $s'$, then general fibers of the evaluation map $\mathrm{ev}_{U'} : U' \to X$ are irreducible.
\end{enumerate}
\end{lemma}

\begin{proof}
    Let $U_{orig}$ be the universal family over the component of Kontsevich space containing $s$. Let $p \in X$ be a general point and let $U_1, \dots, U_m$ be the irreducible components of the fiber of $U_{orig} \to X$ over $p$. Consider the locus of maps $U^p_{star} \subset \overline{\mathcal{M}}_{mg,1}(X, m\beta)$ such that the domain consists of
    \begin{itemize}
        \item a rational component $R$ that carries the marked point and is contracted to $p$, and
        \item $m$ genus $g$ components $C_i$ attached to $R$, with $C_i$ parameterized by $U_i$ for $i = 1,\ldots , m$.
    \end{itemize}
    Observe that $U^p_{star}$ is irreducible, since it is the image of a map $M_{0,m+1} \times U_1 \times \ldots \times U_m \to \overline{\mathcal{M}}_{mg,1}(X, m\beta)$.
    Since $U^{p}_{star}$ meets the smooth locus of $\overline{\mathcal{M}}_{mg,1}(X, m\beta)$, there is a unique irreducible component $U_{new}$ containing it.
    Let $U'_{star}$ be the locus of curves in $U_{new}$ of the same topological type as curves in $U^p_{star}$, that is, the locus of curves whose domain is a rational curve containing the marked point union $m$ other genus $g$ components. Let $U_{star}$ be the irreducible component of $U'_{star}$ containing $U^p_{star}$. The fibers of the evaluation map $U_{star} \to X$ are irreducible: as we go around a loop in $X$, we get a monodromy action on $U_1, \dots, U_m$ via permutation, so that any stable map parameterized by the fiber of $U_{star}$ over $p$ has one genus $g$ component from each $U_i$. This means that the fiber of $U_{star}$ over $p$ must be $U^p_{star}$, and is hence irreducible. Since $p$ is a general point, we conclude that a general fiber of the map $U_{star} \to X$ is irreducible.

    We claim that $U_{new} \to X$ has irreducible general fiber. To see this, let $U_{new}^p$ be the fiber of $U_{new} \to X$ over $p$.  Let $U_1' \subset U_{new}^p$ be the unique irreducible component containing $U_{star}^p$ and suppose $U_2' \subset U_{new}^p$ is any irreducible component.  Since $U_{new}$ is irreducible, the monodromy action coming from $U_{new} \to X$ acts transitively on components of $U_{new}^p$. Thus some element of the monodromy group of $U_{new} \to X$ sends $U'_1$ to $U'_2$, meaning that there is a loop in $X$ under which curves of $U_{star}^p$ go to curves in $U'_2$. This implies that $U'_{2}$ also contains $U_{star}^{p}$.  However, this contradicts the irreducibility of the fibers of $U_{star}$, because the generic point of $U_{star}^p$ lies in the smooth locus of the fiber. Thus, $U_1' = U_2'$ and $U_{new} \to X$ has irreducible general fiber, because the construction of $U_{star}$ does not depend on the choice of the general point $p$.

    By standard deformation theory, a general curve in $U_{new}$ will have genus $mg$ and will be $r$-free, and hence will have smooth domain.  The result follows.
\end{proof}

The next lemma allows us to construct very free rational curves from free rational curves.

\begin{lemma} \label{lemm:interiorfreeimpliesvf}
Let $X$ be a $\mathbb{Q}$-factorial normal projective variety.  Suppose there is a free rational curve $s: \mathbb{P}^{1} \to X^{sm}$ whose numerical class is in the interior of $\Nef_{1}(X)$. Then for every $r \geq 0$ there is an $r$-free rational curve $\widehat{s}: \mathbb{P}^{1} \to X^{sm}$ such that $\widehat{s}_{*}\mathbb{P}^{1}$ is proportional to $s_{*}\mathbb{P}^{1}$.
\end{lemma}

\begin{proof}
By Lemma \ref{lem: producing irreducible fibers}, for some $m>0$ we may construct a free curve $s': \mathbb{P}^{1} \to X^{sm}$ whose class is equal to $ms_{*}\mathbb{P}^{1}$ such that the evaluation map for the family of deformations of $s'$ has connected fibers.  After possibly increasing the number of times we glue-and-smooth, we may ensure that $m > (\dim X)^{2}/2$.

\cite[Proposition 2.13]{JLR25} shows that $\mu^{min}_{s_{*}\mathbb{P}^{1}}(T_{X}) > 0$.  In particular this slope must be at least $\frac{1}{\dim X}$.  By choice of $m$, we see that $\mu^{min}_{s'_{*}\mathbb{P}^{1}}(T_{X}) > (\dim X)/2$.  Since the evaluation map for the family of deformations of $s'$ has connected fibers, the Grauert-M\"ulich theorem of \cite[Proposition 3.1]{PRT24} (applied on a resolution of singularities) shows that for a general deformation $\widetilde{s}$ of $s'$ we have
\begin{equation*}
|\mu^{min}(\widetilde{s}^{*}T_{X}) - \mu^{min}_{\widetilde{s}_{*}\mathbb{P}^{1}}(T_{X})| \leq \rk(T_{X})/2 = (\dim X)/2
\end{equation*}
and thus $\mu^{min}(\widetilde{s}^{*}T_{X}) > 0$.  We obtain the desired map $\widehat{s}$ by precomposing $\widetilde{s}$ with a finite morphism $\mathbb{P}^{1} \to \mathbb{P}^{1}$ of sufficiently large degree.
\end{proof}

\begin{remark}
    Lemma \ref{lemm:interiorfreeimpliesvf} shows that the existence of free rational curves on all klt Fano varieties is equivalent to the existence of very free rational curves.  In brief, by passing to $\mathbb{Q}$-factorializations, it suffices to consider only $\mathbb{Q}$-factorial klt Fano varieties.  Assuming free rational curves always exist, one can use the MMP to find free rational curves on each extremal ray of $\Nef_{1}(X)$ and hence (after gluing and smoothing) a free rational curve with class in the interior of $\Nef_{1}(X)$.  Lemma \ref{lemm:interiorfreeimpliesvf} then yields a very free rational curve.
\end{remark}

\subsection{Smoothable terminal Fano varieties}

We first address the case of smoothable terminal Fano varieties.

\begin{theorem} \label{theo:smoothableterminal}
Let $X$ be a terminal Fano variety. Suppose that $X$ is a fiber of a flat projective family $\pi : \mathcal{X} \to T$ whose general fiber is smooth. Then there is a free rational curve $\widehat{s}: \mathbb{P}^{1} \to X^{sm}$. If we assume further that $X$ is $\mathbb{Q}$-factorial, then we can find an $r$-free curve $\widehat{s}$ for every $r \geq 0$.
\end{theorem}

\begin{proof} 
We may assume that $T$ is smooth of dimension $1$.  Since $X$ is a normal Cartier divisor in $\mathcal{X}$, \cite[Corollaire 5.12.7]{EGAIVII} shows that $\mathcal{X}$ is a normal variety on an open neighborhood of $X$. Thus we may apply \cite[Proposition 3.8]{dFH11} to see that the general fiber of $\pi$ is a Fano variety and that the canonical divisor of $\mathcal{X}$ restricts to the canonical divisor in each fiber.  Let $0 \in T$ be a point such that $X \cong \mathcal{X}_0$.  We will construct the desired curve on $\mathcal{X}_{0}$ by specializing free curves on general fibers using exceptional towers.

Let $\mathcal{X}_{t}$ denote a general fiber of $\pi$.  Choose a finite set of divisors $\{ D_{i} \}$ which generate $\Eff^{1}(\mathcal{X}_{t})$ and choose a finite set of points in $\mathcal{X}_{t}$ that includes one point on each $D_{i}$ as well as a general point of $\mathcal{X}_{t}$.  By \cite[2.1 Theorem]{KMM92b} we know that $\mathcal{X}_{t}$ contains an irreducible rational curve $C_{t}$ through this finite set of points.  In particular $C_{t}$ is a free rational curve of anticanonical degree $d > 2$ whose class lies in the interior of $\Nef_{1}(\mathcal{X}_{t})$.  By Lemma \ref{lem: climb tower morphism scheme}(2), if we take a general exceptional tower $\Phi(\mathcal{X}_{t},p,k)$ of height $k = d-2$ then there will be a deformation of $C_{t}$ whose strict transform meets the top divisor $E_{k}$ of the exceptional tower.

Let $\Phi(X,p,k)$ be a general exceptional tower of height $k = d-2$ on the central fiber.  There exists an \'etale neighborhood $U$ of $0$ such that the base change of $\pi$ to $U$ admits a section $\sigma : U \to \mathcal{X}_U$ with $\sigma(0) = p$ and an extension of $\Phi(X,p,k)$ to a relative exceptional tower $\Phi(\mathcal{X}_U,\sigma,k)$, i.e.~a family of exceptional towers in the fibers of the map $\mathcal{X}_{U} \to U$.  By applying a Hilbert scheme argument to the rational curves constructed above, we obtain a family $f : \PP^1 \times U\setminus\{0\} \to \mathcal{X}_{U}$ of rational curves contained in the fibers of $\pi_U : \mathcal{X}_{U} \to U$ whose strict transforms meet the top divisor $\mathcal{E}_k$ of the exceptional tower.  Let $f_0: C \to X$ be the limit of these curves in $X = \mathcal{X}_0$. We note that $C$ must have arithmetic genus $0$, must have $-K_X$ degree $d$, and it must meet $E_k$ in $\Phi(X,p,k)$.

Let $\phi: X' \to X$ be a resolution of singularities of $X$ and let $C_{exc}$, $\mathcal{L}$, $C_{adj}$, and $\mathcal{Q}_\ell$ be as given in Lemma \ref{lem-numericalExceptionalTowers} and Remark \ref{rmk-numericalExceptionalTowers}. Observe that for each component $C_{\ell}$ of $C_{adj}$, $|\mathcal{Q}_\ell| = 1$, so let $q_{\ell}$ be the unique point of $\mathcal{Q}_\ell$. Let $C_{small} = C_{exc} \cup C_{adj}$. Taking strict transforms, we get a map $s: C_{small} \to X'$. By Corollary \ref{cor: rational curves degree}, we see that for each component $C_{\ell}$ of $C_{adj}$ we have 
$$-K_{X'} \cdot s_*C_{\ell} \geq 1 + \sum_{i=0}^{k} (i+1) \cdot \mult_{q_{\ell}}(s_k, \widetilde{E}_i).$$
Thus, 
\begin{align*}
    -K_{X} \cdot f_{0*}C & \geq -\phi^{*}K_X \cdot s_* C_{adj} \\
    & \geq -K_{X'} \cdot s_*C_{adj} \\
    & \geq |\mathcal{L}| + \sum_{\ell \in \mathcal{L}} \left(\sum_{i = 0}^k (i+1) \cdot \mult_{q_{\ell}}\big(s|_{C_\ell}, \widetilde{E}_i\big)\right) \\
    & \geq |\mathcal{L}| + k+1 = |\mathcal{L}| + d - 1 \geq d.
\end{align*}
Hence all these inequalities must be equalities.  In particular $|\mathcal{L}| = 1$ and $-K_{X} \cdot f_{0*}C = -K_{X} \cdot \phi_{*}s_{*}C_{adj}$, and hence, $C$ is irreducible. If $s(C)$ met the singular locus of $X$, then we would have $-\phi^{*}K_X \cdot s_*C > -K_{X'} \cdot s_* C$, contradicting the equality above. Hence, $s(C)$ is contained in the smooth locus, and since it meets a general point, it must be a free curve.

Now suppose further that $X$ is $\mathbb{Q}$-factorial. By \cite[Proposition 6.5]{dFH11} there is a sheaf $\mathcal{GN}^{1}(\mathcal{X}/T)$ on $T$ in the analytic topology whose fiber at a point $t \in T$ is isomorphic to $N^{1}(\mathcal{X}_{t})$.  Furthermore, by \cite[Discussion before Theorem 6.8]{dFH11} this identification preserves the pseudo-effective cone of divisors.  Dually, we obtain an identification of the nef cone of curves for fibers of $\pi: \mathcal{X} \to T$.  This implies that $f_{0}: C \to X$ lies in the interior of the nef cone of curves on $X$.  The existence of the desired $r$-free curves follows from Lemma \ref{lemm:interiorfreeimpliesvf}.
\end{proof}

\subsection{Terminal Fano varieties of low dimension} 
We next discuss LCIQ terminal fourfolds.  By combining Theorem \ref{theo:introratcurves} with \cite[Corollary 5.9]{KM99}, we obtain a quick proof of the existence of a free rational curve.

\begin{theorem} \label{theo:lciqfourfold}
    Let $X$ be an LCIQ terminal Fano fourfold.  Then there is a free rational curve in the smooth locus of $X$.
\end{theorem}

\begin{proof}
Theorem \ref{theo:introratcurves} shows that there is a dominant family of rational curves on $X$ which meet the singular locus at most $2$ times.  The existence of a free rational curve follows from \cite[5.9 Corollary]{KM99} which shows that any terminal Fano variety that has LCIQ singularities and admits a dominant family of rational curves meeting the singular locus at most twice will also have a free rational curve in its smooth locus.
\end{proof}

\begin{proof}[Proof of Theorem \ref{theo:freeonterminal}:]
Combine Theorem \ref{theo:smoothableterminal} and Theorem \ref{theo:lciqfourfold}.
\end{proof}

We can also improve \cite[Theorem 1.4]{JLR25}:

\begin{theorem} \label{theo:terminalthreefoldveryfree}
    Suppose that $X$ is a terminal Fano threefold.  Then $X$ admits an $r$-free rational curve for every $r \geq 0$.
\end{theorem}

\begin{proof}
    Since a general $r$-free curve will avoid any fixed codimension $\geq 2$ subset, it suffices to find an $r$-free curve on a $\mathbb{Q}$-factorialization of $X$.  So we will assume $X$ is $\mathbb{Q}$-factorial.
    
    We first show that each extremal ray $R$ of $\Nef_{1}(X)$ is spanned by a free rational curve. Indeed, suppose that $D$ is an effective divisor on $X$ such that $\Nef_{1}(X)$ intersects $D^{\perp}$ only along the ray $R$.  By running the $D$-MMP, we obtain a birational model $\phi: X \dashrightarrow X'$ and a Mori fibration $\pi: X' \to Z$.  \cite[Theorem 1.4]{JLR25} guarantees that there is a free rational curve in a general fiber of $\pi$.  By deforming this curve, we may assume it is contained in the locus where $\phi^{-1}$ is an isomorphism.  Then the strict transform on $X$ is a free rational curve whose numerical class is in $R$.
    
    Thus, for each extremal ray $R$ of $\Nef_{1}(X)$ we may find a free rational curve whose numerical class lies in $R$.  By gluing and smoothing suitable deformations of these curves, we find a free rational curve whose numerical class lies in the interior of $\Nef_{1}(X)$.  Then the existence of an $r$-free rational curve follows from Lemma \ref{lemm:interiorfreeimpliesvf}.
\end{proof}

\subsection{Very free rational curves on smooth Fano varieties}

Our final result in this section addresses very free curves on smooth Fano varieties.

\begin{theorem}
  Let $X$ be a smooth Fano variety.  Then every rational ray in the interior of $\Nef_{1}(X)$ is generated by the class of a very free rational curve.  
\end{theorem}

\begin{proof}
Let $\mathcal{R} \subset \Nef_{1}(X)^{\circ}$ denote the set of rays generated by classes of very free rational curves.  We first show that $\mathcal{R}$ is dense in $\Nef_{1}(X)^{\circ}$.  By \cite[Theorem 3.6]{JLR25} there is a dense set of rays in $\Nef_{1}(X)$ which are represented by (not necessarily rational) curves $s: C \to X$ such that $\mu^{min}(s^{*}T_{X}) > 0$.  Thus it suffices to show that for every open neighborhood $U \subset \Nef_{1}(X)$ of a class $s_{*}C$ as above there is an element of $\mathcal{R}$ which intersects $U$.

Choose a positive integer $m$ sufficiently large so that the finitely many $\mathbb{Z}$-classes $\beta \in \Eff_{1}(X)$ with $-K_{X} \cdot \beta \leq g(C) \dim X$ each satisfy $s_{*}C - \frac{1}{m} \beta \in U$.  By \cite[Theorem 3.2]{JLR25}  we can find a (not necessarily rational) free curve $s': C \to X$ such that $s'_{*}C = m's_{*}C$ for some $m' \geq m$.  If $g(C) = 0$, let $T$ be $s'_{*}C$.  If $g(C)>0$, apply Theorem \ref{theo:exctowerbreaking} to $s': C \to X$ to find a free rational curve $T$ such that $(s'_{*}C - T) \in \Eff_{1}(X)$ and $-K_{X} \cdot (s'_{*}C - T) \leq g(C) \dim X$.  Then $\frac{1}{m'}T \in U$.  Lemma \ref{lemm:interiorfreeimpliesvf} shows that there is a very free rational curve whose numerical class is proportional to $T$.  This establishes density of $\mathcal{R}$.

Finally, we must show that every rational ray $R$ in the interior of $\Nef_{1}(X)$ is generated by a very free curve.  We can choose a simplicial cone in the interior of $\Nef_{1}(X)$ containing $R$.  By choosing sufficiently small open neighborhoods of the generators of this simplicial cone and appealing to density of $\mathcal{R}$, we can find a finite set of very free rational curves $\{ C_{i} \}$ such that the cone generated by their numerical classes contains $R$ in its interior.  In particular there is a positive $\mathbb{Z}$-linear combination of the numerical classes of the $C_{i}$ that lies on $R$.  Thus we can glue and smooth a compatibly chosen collection of deformations of the various $C_{i}$ to find a very free rational curve along the ray $R$.
\end{proof}

\section{Fujita invariants} \label{sect:fujitainvariants}

The Fujita invariant compares the negativity of the canonical divisor against the positivity of a chosen big and nef $\mathbb{Q}$-Cartier divisor $L$.  The following definition slightly extends Definition \ref{defi:introa-invariant}.

\begin{definition}
Suppose $(Y,L)$ is a pair consisting of a projective variety $Y$ and a big and nef $\mathbb{Q}$-Cartier divisor $L$.  When $Y$ is smooth, the Fujita invariant is
\begin{equation*}
a(Y,L) = \min \{ t\in \mathbb{R} \mid  K_Y + tL \textrm{ is pseudo-effective }\}.
\end{equation*}
If $Y$ is singular we choose a resolution of singularities $\phi: Y' \to Y$ and define $a(Y,L)$ to be $a(Y',\phi^{*}L)$; this value is independent of the choice of $\phi$ by \cite[Proposition 2.7]{HTT15}.
\end{definition}

The papers \cite{LT19} and \cite{LRT24} demonstrate that the existence of a family of curves of large dimension implies that the Fujita invariant is also large.  The following easy result demonstrates a typical application of Fujita invariants to curves on projective varieties.

\begin{proposition} \label{prop:fujupperbound}
    Let $Y$ be a smooth projective variety equipped with a big and nef $\mathbb{Q}$-Cartier divisor $L$. Suppose that $M$ is an irreducible component of $\Mor(C,Y)$ parameterizing a dominant family of curves on $Y$.  Then for $s \in M$ we have
    \[ \dim M \leq a(Y,L) (L \cdot s_{*}C) + \dim Y + 2g(C). \]
\end{proposition}

\begin{proof}
    Since $K_{Y} + a(Y,L)L$ is pseudo-effective, it has non-negative intersection against any curve that moves in a dominant family, and in particular against $s_{*}C$.  Thus the statement follows from the upper bound $\dim M \leq -K_{Y} \cdot s_{*}C + \dim Y + 2g(C)$ that holds for a dominant family of curves on $Y$ by \cite[Proposition 2.5]{LRT23}.
\end{proof}

The main results of this section prove a type of converse to Proposition \ref{prop:fujupperbound}: there exist irreducible components of $\Mor(C,Y)$ whose dimension is bounded below by an expression involving the Fujita invariant.  We will deduce these dimension bounds from bounds on intersection numbers of moving curves against the adjoint divisor $K_{Y}+a(Y,L)L$.  Proposition \ref{prop:orthogonaltoadjoint} proves optimal bounds for certain curves $C$, while Theorem \ref{theo:approximatebound} gives slightly weaker bounds for all curves.

We also give numerous applications of these results to the Fujita invariant and its relationship with families of curves.  In particular, we show that as we vary the curve $C$ the moduli space $\Mor(C,Y)$ behaves ``uniformly'': we can study the behavior for a single curve and then deduce geometric properties of this space for curves of arbitrary genus.

\subsection{Fujita invariants and existence of curves}

In this section we prove two main results: 
we construct irreducible components of $\Mor(C,Y)$ whose dimension captures the Fujita invariant of $Y$.  A crucial ingredient is the recent work of \cite{JLR25} which constructs free curves in specified parts of the nef cone of curves.  Our first result finds curves which pair to $0$ with the adjoint divisor $K_Y+a(Y,L)L$.

\begin{proposition} \label{prop:orthogonaltoadjoint}
Let $Y$ be a smooth projective uniruled variety and let $L$ be a big and nef $\mathbb{Q}$-Cartier divisor on $Y$.  Then there is a curve $C$ with the following property: for any $r \geq 0$ and any $\Theta \geq 0$ there is a morphism $s: C \to Y$ satisfying:
\begin{enumerate}
    \item $s$ is an $r$-free curve in a fiber of the Iitaka fibration for $(K_{Y} + a(Y,L)L)$ (and in particular, $s$ is almost $r$-free),
    \item $L \cdot s_{*}C \geq \Theta$, and
    \item $(K_{Y} + a(Y,L)L) \cdot s_{*}C = 0$.
\end{enumerate}
\end{proposition}

\begin{proof}
By \cite[0.3 Corollary]{BDPP13} the uniruledness of $Y$ is equivalent to asserting that $a(Y,L) > 0$.
Since $K_{Y} + a(Y,L)L$ lies on the boundary of the pseudo-effective cone, the subset of $\Nef_{1}(Y)$ with vanishing intersection against this divisor is an extremal face $\mathcal{F}$.  By Wilson's theorem there is an effective $\mathbb{Q}$-Cartier divisor $\Delta$ and an ample $\mathbb{Q}$-Cartier divisor $H$ such that $\Delta + H \sim_{\mathbb{Q}} a(Y,L)L$ and $(Y,\Delta+H)$ is a klt pair.  Since $H$ is ample, $\mathcal{F}$ is a $(K_{Y} + \Delta)$-negative extremal face of the enlarged cone $\Nef_{1}(Y) + \Eff_{1}(Y)_{K_{Y} + \Delta \geq 0}$.  Thus the argument of \cite[Proposition 3.5]{JLR25} shows that there is a curve satisfying (1) that lies on the face $\mathcal{F}$.  By construction this curve satisfies (3). To see (2), we may increase $r$ so that $r \geq \frac{a(Y,L) \Theta}{\dim F } - 2g(C)$ where $F$ is a fiber of the Iitaka fibration for $K_{Y} + a(Y,L)L$.  Then
\begin{equation*}
L \cdot s_{*}C = \frac{-K_{Y} \cdot s_{*}C}{a(Y,L)} \geq \frac{(\dim F) (2g(C)+r)}{a(Y,L)} \geq \Theta.
\end{equation*}
\end{proof}

Proposition \ref{prop:orthogonaltoadjoint} leads to a precise dimension computation.  Since a family of almost $r$-free curves will always deform in a $H^{0}(C,s^{*}T_{X})$-dimensional family, any component $M$ as in Proposition \ref{prop:orthogonaltoadjoint} will satisfy
\begin{equation*}
    \dim M = a(Y,L)(L \cdot s_{*}C) + (\dim Y)(1-g(C)) + g(C) \kappa(Y,K_{Y}+a(Y,L)L).
\end{equation*}
We emphasize that this equation is an equality: the Fujita invariant and the Iitaka dimension of $K_Y + a(Y,L)L$ precisely determine the dimension.

Our second result shows that for every curve $C$ there is a morphism $s: C \to X$ such that $s_{*}C$ has ``small'' intersection against $K_Y+a(Y,L)L$.  In contrast to Proposition \ref{prop:orthogonaltoadjoint} which only guarantees the existence of a single curve, Theorem \ref{theo:ainvandcurves} applies to arbitrary curves at the cost of proving a slightly weaker statement.  The proof combines Theorem \ref{thm: nicest curves} with an MMP argument.

\begin{theorem} \label{theo:approximatebound}
Let $Y$ be a smooth uniruled projective variety and let $L$ be a big and nef $\mathbb{Q}$-Cartier divisor on $Y$.  Denote by $m$ the dimension of a general fiber of the Iitaka fibration for $K_{Y} + a(Y,L)L$.

Suppose $C$ is a smooth projective curve of gonality $\rho$.  For any positive integer $\Theta$, there is an irreducible component $N_{C,\Theta} \subset \Mor(C,Y)$ parameterizing a dominant family of curves with $\Theta \leq \frac{a(Y, L)}{\rho}L \cdot s_* C \leq \Theta + m$ such that
\begin{equation*}
(K_{Y} + a(Y,L)L) \cdot s_{*}C \leq \rho(m - 2).
\end{equation*}
\end{theorem}

\begin{proof}
We first prove the statement for $\mathbb{P}^{1}$, then extend to arbitrary curves $C$ by taking covers $C \to \mathbb{P}^{1}$.

Consider the Iitaka fibration $\pi: Y \dashrightarrow Z$ for the divisor $K_{Y} + a(Y,L)L$.  Let $\phi: Y' \to Y$ be a smooth projective birational model admitting a morphism $\pi': Y' \to Z$ resolving $\pi$ and let $F'$ be a general fiber of $\pi'$.  By running a relative $(K_{Y'} + a(Y,L)\phi^{*}L)$-MMP over $Z$, we obtain a birational contraction $\psi: Y' \dashrightarrow \widehat{Y}$ such that a general fiber $\widehat{F}$ of $\widehat{Y} \to Z$ satisfies $-K_{\widehat{F}} \equiv a(Y,L)\psi_{*}\phi^{*}L|_{\widehat{F}}$.  For some sufficiently small $\epsilon > 0$, this is also a run of the $(K_{Y'} + (1-\epsilon)a(Y,L)\phi^{*}L)$-MMP. 
We continue running the latter MMP to obtain a birational model $\widehat{Y} \dashrightarrow \widetilde{Y}$ such that $\widetilde{Y}$ admits a relative Mori fibration over $Z$.  Letting $\gamma: Y' \dashrightarrow \widetilde{Y}$ denote the induced birational map, we see that a general fiber $T$ of this Mori fibration will be a terminal Fano variety satisfying $-K_{T} \equiv a(Y,L)\gamma_{*}\phi^{*}L$.

Let $W$ be a smooth birational model of $Y$ admitting birational morphisms $\phi_{W}: W \to Y$ and $\gamma_{W}: W \to \widetilde{Y}$.  Let $T'$ denote the strict transform of $T$ on $W$.  Then $T'$ is smooth and since $\phi_{W}^{*}L|_{T'}$ is big and nef we have $\gamma_{W}^{*}\gamma_{W,*}\phi^{*}L|_{T'} \geq \phi_{W}^{*}L|_{T'}$ by the negativity of contraction lemma (\cite[Lemma 3.39]{KM98}).  If $\dim T' \geq 3$, Theorem \ref{thm: nicest curves} shows that there is a free rational curve $s': \mathbb{P}^{1} \to T'$ satisfying
\begin{align*}
    \dim T - 2 & \geq (K_{T'} - \gamma_{W}^{*}K_{T}) \cdot s'_{*}\mathbb{P}^{1} \\
    & =  (K_{T'} + a(Y,L)\gamma_{W}^{*}\gamma_{W,*}\phi_{W}^{*}L) \cdot s'_{*}\mathbb{P}^{1} \\
    & \geq (K_{T'} + a(Y,L)\phi_{W}^{*}L) \cdot s'_{*}\mathbb{P}^{1}.
\end{align*}
(If $\dim T' \leq 2$, then $T$ is smooth and we can run a similar argument with a free rational curve on $T$.)
Since $T'$ is a fiber of a morphism from $W$, we have $K_{T'} \cdot s'_{*}\mathbb{P}^{1} = K_{W} \cdot s'_{*}\mathbb{P}^{1}$.  Finally, since $s'$ is a free curve, defining $s = \phi_{W} \circ s': \mathbb{P}^{1} \to Y$ we have $K_{W} \cdot s'_{*}\mathbb{P}^{1} \geq K_{Y} \cdot \phi_{W,*}s'_{*}\mathbb{P}^{1}$.  Thus $s_{*}\mathbb{P}^{1}$ satisfies the desired intersection bound against $K_{Y} + a(Y,L)L$.

It only remains to bound the $L$-degree of $s_{*}\mathbb{P}^{1}$.  We may use our choice of $\Theta$ when applying Theorem \ref{thm: nicest curves} to construct $s'$.  First, note that since $s'(\mathbb{P}^{1})$ moves in a dominant family on $T'$, we have
\begin{equation*}
    \Theta + \dim T \geq -K_{T} \cdot \gamma_{W,*}s'_{*}\mathbb{P}^{1} = a(Y,L)\gamma_{W,*}\phi_{W}^{*}L \cdot \gamma_{W,*}s'_{*}\mathbb{P}^{1} \geq a(Y,L)\phi_{W}^{*}L \cdot s'_{*}\mathbb{P}^{1}.
\end{equation*}
Second, note that since $K_{T'} + a(Y,L)\phi_{W}^{*}L$ is pseudo-effective
\begin{equation*}
    \Theta \leq -K_{T'} \cdot s'_{*}\mathbb{P}^{1} \leq a(Y,L)\phi_{W}^{*}L \cdot s'_{*}\mathbb{P}^{1}.
\end{equation*}
Together these prove the desired bounds for $\mathbb{P}^{1}$.

Now suppose that $C$ is an arbitrary curve equipped with a morphism $h: C \to \mathbb{P}^{1}$ of degree $\rho$.  Fix an $s \in N_{\mathbb{P}^{1}, \Theta}$ and let $N_{C,\Theta}$ denote the irreducible component of $\Mor(C,Y)$ containing $s \circ h$.  Then
\begin{equation*}
    (K_{Y} + a(Y,L)L) \cdot s_{*}h_{*}C \leq \rho (m-2)
\end{equation*}
and intersections against $L$ are also scaled by $\rho$.
\end{proof}

As a consequence, we prove that for every curve $C$ there is an irreducible component $M \subset \Mor(C,X)$ whose dimension is a bounded amount away from the expected value given by a Fujita invariant computation.

\begin{theorem}\label{theo:ainvandcurves}
Let $Y$ be a smooth uniruled projective variety and let $L$ be a big and nef $\mathbb{Q}$-Cartier divisor on $Y$.  Let $C$ be a smooth curve of genus $g \geq 0$ and gonality $\rho \geq 1$.  For any positive integer $\Theta$, there is an irreducible component $N_{C,\Theta} \subset \Mor(C,Y)$ parameterizing a dominant family of maps $s : C \to Y$ satisfying $ \Theta \leq \frac{ a(Y,L)}{\rho} \cdot (L \cdot s_{*}C) \leq \Theta + \dim Y$ and
\begin{equation*}
 \dim N_{C,\Theta} \geq a(Y,L) \cdot (L \cdot s_{*}C) +(1-g - \rho)\dim Y + 2\rho.
 \end{equation*}
\end{theorem}

\begin{proof}
This follows from Theorem \ref{theo:approximatebound} and the lower bound on $\dim N_{C,\Theta}$ coming from the expected dimension.
\end{proof}

\begin{remark} \label{rmk:manyRationalCurves}
    In the case $C = \PP^1$, the statement says that $Y$ has infinitely many components $N$ of rational curves $s: \PP^1 \to Y$ satisfying $\dim N \geq a(Y,L) \cdot (L \cdot s_*\PP^1) + 2$. If $L=-K_{Y}$, this is $\dim Y - 2$ less than the ``expected'' value $a(Y,L) \cdot (L \cdot s_*\PP^1) + \dim Y$, so we see that there will be families of rational curves on $Y$ with ``almost'' the dimension that we would expect coming from the $a$-invariant.
\end{remark}

\begin{proof}[Proof of Theorem \ref{theo:introainv}:]
Follows immediately from Theorem \ref{theo:ainvandcurves} plus the gonality bound $\rho \leq g+1$.
\end{proof}

Suppose $X$ is a smooth Fano variety and that $f: Y \to X$ is a generically finite morphism from a projective variety $Y$.  Then the Fujita invariant $a(Y,-f^{*}K_{X})$ allows us to compare the behavior of curves on $Y$ and curves on $X$.  Since the behavior of the Fujita invariant is highly constrained by the Minimal Model Program, we obtain good control of the properties of curves on $X$ that factor through $Y$.

Let $N$ denote an irreducible component of $\Mor(C,Y)$ and let $M$ denote an irreducible component of $\Mor(C,X)$ containing $f_{*}N$.  Loosely speaking, we expect
\begin{equation*}
a(Y,-f^{*}K_{X}) = \limsup_{-K_{X}\textrm{-deg} \to \infty} \frac{\dim N}{\mathrm{expdim} \, M}
\end{equation*}
where $\mathrm{expdim}(M)$ denotes the expected dimension $-f^{*}K_{X} \cdot s_{*}C + (\dim X)(1-g(C))$.  One inequality was proved by \cite[Theorem 1.6]{LRT23}: we can bound the value of $a(Y,-f^{*}K_{X})$ from below using the dimension of families of curves on $Y$.  The following result establishes the reverse inequality.

\begin{corollary}
Let $X$ be a smooth projective Fano variety and let $C$ be a smooth curve.  Suppose that $f: Y \to X$ is a generically finite morphism from a smooth uniruled projective variety $Y$.  For any $\epsilon>0$, there is an irreducible component $N \subset \Mor(C,Y)$ such that
\begin{equation*}
\frac{\dim N}{\mathrm{expdim} \, M} > a(Y,-f^{*}K_{X}) - \epsilon
\end{equation*}
where $M$ denotes an irreducible component of $\Mor(C,X)$ containing the pushforward of $N$.
\end{corollary}

\begin{proof}
Apply Theorem \ref{theo:ainvandcurves} to $Y$ with $L = -f^{*}K_{X}$.  When $\Theta$ is sufficiently large the desired inequality follows.
\end{proof}

In particular, the subvarieties where curves ``accumulate'' -- e.g.~where there are more curves than expected -- will be exactly the ones with large Fujita invariant.

\subsection{Uniform behavior of curves on Fano varieties}

Let $X$ be a Fano variety.  The following theorems prove an equivalence between properties of Fujita invariants of subvarieties, the moduli spaces of specific curves on $X$, and the moduli spaces of arbitrary curves on $X$. The two main results in this section are distinguished by the bound they place on the Fujita invariants of subvarieties of $X$.  We first discuss when the Fujita invariant of every subvariety is $\leq 1$.

\begin{theorem}\label{thm:otherequivalentaconditions}
For a smooth Fano variety $X$, the following are equivalent.
\begin{enumerate}
\item For every closed subvariety $Y \subsetneq X$, $a(Y,-K_X|_Y) \leq 1$.
\item There exists a smooth projective curve $C$ and a constant $R$ such that the dimension of every irreducible component of $\Mor(C,X)$ is at most $R$ greater than the  expected dimension.
\item There exists a smooth projective curve $C$ such that the dimension of every irreducible component of $\Mor(C,X)$ is at most $g(C)\dim X$ more than the expected dimension.
\item For every smooth projective curve $C$ and every irreducible component $M \subset \Mor(C,X)$ the dimension of $M$ is at most $g(C)\dim X$ more than the expected dimension.
\end{enumerate}
\end{theorem}

\begin{proof}
(1) $\implies$ (4): Fix a curve $C$ and suppose for a contradiction that $M \subset \Mor(C,X)$ is an irreducible component parameterizing maps $s: C \to X$ such that $\dim M > \mathrm{expdim}{M} + g(C)\dim X$.  Let $f: Y \to X$ denote a resolution of the closed subvariety of $X$ swept out by this family of curves.  Theorem \ref{theo:exctowerbreaking} yields a dominant family $N$ of rational curves on $Y$ satisfying $\dim N \geq \dim M$.  But this is larger than the expected dimension of the corresponding irreducible component of $\Mor(\mathbb{P}^{1},X)$.  By \cite[Theorem 1.1]{LT19} we conclude that $X$ admits a subvariety of Fujita invariant $>1$.

(4) $\implies$ (3), (3) $\implies$ (2): immediate.

(2) $\implies$ (1): we prove the contrapositive.  Suppose there is a closed subvariety $Y \subsetneq X$ such that $a(Y,-K_{X}|_{Y}) > 1$.  Let $f: Y' \to Y$ denote a resolution of singularities and apply Theorem \ref{theo:ainvandcurves} to $Y'$ with $L = -f^{*}K_{X}|_{Y}$.  Note that $Y'$ is uniruled by \cite[Theorem 0.3]{BDPP13}. For every sufficiently large $\Theta$ we obtain an irreducible component $N_{C,\Theta} \subset \Mor(C,Y')$ such that any irreducible component $M \subset \Mor(C,X)$ containing the images of the curves in $N_{C,\Theta}$ will violate the dimension bound in (2).
\end{proof}

\begin{proof}[Proof of Theorem \ref{thm:introequivalentaconditions}:]
Follows immediately from Theorem \ref{thm:otherequivalentaconditions}.
\end{proof}

The Fano varieties such that every subvariety has Fujita invariant $<1$ are even more special.  The following theorem shows that this condition is equivalent to having few families of curves of larger than expected dimension.

\begin{theorem}\label{thm:equivalentaconditions}
For a smooth Fano variety $X$, the following are equivalent.  
\begin{enumerate}
\item For every closed subvariety $Y \subsetneq X$, $a(Y,-K_X|_Y) < 1$.
\item For every closed subvariety $Y \subsetneq X$, there exists a curve $C$ whose genus $g$ and gonality $\rho$ satisfy $g > 1 + \rho \frac{\dim Y-2}{\dim X-\dim Y}$ such that there are only finitely many irreducible components of $\Mor(C,X)$ which sweep out $Y$ and have larger than expected dimension.
\item There exists a curve $C$ whose genus $g$ and gonality $\rho$ satisfy $g > 1 + \rho (\dim X - 3)$ such that there are only finitely many irreducible components of $\Mor(C,X)$ with larger than expected dimension.
\item For every curve $C$ there are only finitely many irreducible components of $\Mor(C,X)$ with larger than expected dimension.
\end{enumerate}
\end{theorem}

\begin{proof}
(1) $\implies$ (4): We prove the contrapositive statement. Suppose there exists a curve $C$ such that $\Mor(C,X)$ has infinitely many components with larger than expected dimension.   Such a component cannot parameterize any free curves.  According to \cite[Theorem 1.2]{LRT23}, when the degree of the curves is sufficiently large then any such component falls into one of three categories.  In the first category, the curves sweep out a subvariety $Y \subsetneq X$ with $a(Y,-K_{X}) \geq 1$.  In the second category, there is a dominant generically finite map $f: Y \to X$ and a subvariety $F \subsetneq Y$ with $a(F,-f^{*}K_{X}|_{F}) = 1$.  Then the image of $F$ in $X$ will be a proper subvariety of Fujita invariant at least $1$.  In the third category, $X$ must carry a rational curve $C$ with $-K_{X} \cdot C = 2$ and such a curve has Fujita invariant $1$.  Thus in every case we find a subvariety of Fujita invariant $\geq 1$, proving the contrapositive.

(4) $\implies$ (3), (3) $\implies$ (2): immediate

(2) $\implies$ (1): We prove the contrapositive statement.  Suppose there is a proper subvariety $Y \subsetneq X$ with Fujita invariant $\geq 1$.  Let $C$ be any curve whose genus and gonality satisfy the inequality in (2).  We apply Theorem \ref{theo:ainvandcurves} to a resolution of singularities of $Y$ and with $L$ equal to the pullback of $-K_{X}$ to obtain infinitely many irreducible components $M \subset \Mor(C,Y)$ satisfying
\begin{align*}
\dim M & \geq \mathrm{expdim}_{X} M + (\dim X-\dim Y)(g(C)-1) - \rho(\dim Y-2) \\
& > \mathrm{expdim}_{X} M.
\end{align*}
This finishes the proof.
\end{proof}

\begin{remark}
We expect the statement of Proposition \ref{prop:orthogonaltoadjoint} to hold for every smooth projective curve $C$.  Then the argument above would imply that the equivalent conditions of Theorem \ref{thm:equivalentaconditions} are also equivalent to:
\begin{enumerate}
\setcounter{enumi}{4}
\item There exists a smooth projective curve $C$ of genus $\geq 2$ such that there are only finitely many irreducible components of $\Mor(C,X)$ with larger than the expected dimension.
\end{enumerate}
\end{remark}

In fact, we can put even stronger restrictions on families of curves on Fano varieties with no subvarieties of Fujita invariant $\geq 1$.

\begin{theorem} \label{theo:strongestainvcurves}
Let $X$ be a smooth Fano variety such that there are no proper subvarieties $Y$ with $a(Y,-K_{X}) \geq 1$.  There is a quadratic function $L(g)$ such that for a smooth curve $C$ of genus $g$ and number $r \geq 0$, every irreducible component $M \subset \Mor(C,X)$ parameterizing maps of $(-K_X)$-degree $\geq L(g)+r \dim X$ will generically parameterize $r$-free maps.  Furthermore, the evaluation morphism for the universal family over $M$ will have geometrically irreducible generic fiber.
\end{theorem}

\begin{proof}
By \cite[Theorem 1.3]{HL20} there is a constant $a<1$ such that no varieties of dimension $\leq \dim X$ have Fujita invariant in the interval $(a,1)$ with respect to any big and nef Cartier divisor.  
Let $M \subset \Mor(C,X)$ be an irreducible component   
and let $ev : U \to X$ denote the evaluation map on the normalization $U$ of the universal family over $M$.  Let $f' : Z' \to X$ denote the Stein factorization of (a projective closure of) $ev$ and let $f: Z \to X$ denote a birational model of $f'$ with $Z$ smooth.  Note that there is an open subset of $M$ parameterizing curves on $Z$ which are the strict transforms of general curves parameterized by $M$ on $Z'$.  

Let $s : C\to Z$ be a general map parameterized by $M$.  Since $M$ parameterizes a dominant family of curves on $Z$, it has dimension at most $-K_Z \cdot s_*C + \dim Z + 2g(C)$ by \cite[Proposition 2.5]{LRT23}.  On the other hand, $M$ has at least the expected dimension $-K_X \cdot (f \circ s)_*C + (1-g(C))\dim X$.  As $s_*C$ moves in a dominant family it has non-negative intersection against $K_{Z} - a(Z,-f^{*}K_{X})f^{*}K_{X}$, yielding 
\begin{equation*}
a(Z,-f^*K_X) \geq \frac{-K_{Z} \cdot s_{*}C}{-K_X \cdot (f \circ s)_*C} \geq \frac{-K_X \cdot (f \circ s)_*C + (\dim X - \dim Z) -g(C)(\dim X + 2)}{-K_X \cdot (f \circ s)_*C}.
\end{equation*}
If $-K_X \cdot (f \circ s)_*C > \frac{g(C)(\dim X + 2)}{1-a}$ then $a(Z,-f^*K_X) > a$.  By choice of $a$ we conclude $a(Z,-f^{*}K_{X}) \geq 1$ and so Lemma \ref{lemm:ainvsubtodom}.(2) shows that $f$ is birational.  Thus we have a dominant family of curves and the evaluation morphism will have geometrically irreducible generic fiber.

Since $K_{X} + a(X,-K_{X})(-K_{X}) = 0$, Lemma \ref{lemm:ainvsubtodom} shows that $X$ has Picard rank $1$.  Letting $\alpha$ be a generator of $\Nef_{1}(X)$, we have $\mu^{min}_{\alpha}(T_{X}) > 0$ by \cite[Proposition 2.13]{JLR25}.  \cite[Theorem 1.2]{LRT23} shows that there is a quadratic function $Q_{1}(g)$ such that when the degree of the curves is larger than $Q_{1}(g)$ then the general map $s$ parameterized by $M$ is free and maps birationally onto its image.   
The Grauert-M\"ulich theorem of \cite[Theorem 3.8.(1)]{LRT23} yields a quadratic function $Q_{2}(g)$ such that a general curve $s: C \to X$ parameterized by $M$ satisfies
\begin{equation*}
\mu^{min}(s^{*}T_{X}) \geq \mu^{min}_{s_{*}C}(T_{X}) - Q_{2}(g) = \frac{\mu^{min}_{\alpha}(T_{X})}{\mu_{\alpha}(T_{X})} \cdot \frac{-K_{X} \cdot s_{*}C}{\dim X} - Q_{2}(g).
\end{equation*}
Rearranging, we see that when $-K_{X} \cdot s_{*}C \geq \sup\{Q_{1}(g), \frac{\mu_{\alpha}(T_{X})}{\mu^{min}_{\alpha}(T_{X})}(\dim X)(2g + r + Q_{2}(g))\}$ we have the desired freeness.
\end{proof}

The previous result relied on the following lemma describing the geometry of subvarieties with large Fujita invariant.  It is well-known to experts  -- for example, some version of it appears in \cite[Section 2]{HaseLiu25}.

\begin{lemma} \label{lemm:ainvsubtodom}
Let $X$ be a smooth projective uniruled variety and let $L$ be an ample $\mathbb{Q}$-Cartier divisor on $X$.  Assume that $\kappa(K_{X} + a(X,L)L) = 0$.
\begin{enumerate}
\item Suppose that there are no proper subvarieties $Y \subsetneq X$ such that $a(Y,L|_{Y}) > a(X,L)$.  Then the only dominant generically finite morphisms $f: Y \to X$ with the properties $a(Y,f^{*}L) = a(X,L)$ and $\kappa(K_{Y}+a(Y,f^{*}L)f^{*}L) = 0$ are birational maps.  

\item Suppose that there are no proper subvarieties $Y \subsetneq X$ such that $a(Y,L|_{Y}) \geq a(X,L)$.  Then the only generically finite morphisms $f: Y \to X$ such that $a(Y,f^{*}L) \geq a(X,L)$ are birational maps.
\end{enumerate}
If furthermore $X$ is Fano and $L = -K_{X}$, then in case (2) the Picard rank of $X$ is $1$.
\end{lemma}

\begin{proof}
(1)  Suppose that $f: Y \to X$ is a generically finite morphism from a smooth projective variety $Y$ such that $a(Y,f^{*}L) = a(X,L)$ and $\kappa(K_{Y} + a(Y,f^{*}L)f^{*}L) = 0$.   Note that  $K_{Y} + a(Y,f^{*}L)f^{*}L = K_{Y/X} + f^{*}(K_{X} + a(X,L)L)$ and so every divisor $E$ contained in the ramification locus of $f$ is contracted by the $(K_{Y} + a(Y,f^{*}L)f^{*}L)$-MMP.  \cite[Corollary 2.8]{Sengupta21} shows that the image of each such divisor $E$ in $X$ either has codimension $\geq 2$ or is an irreducible divisor with $L$-Fujita invariant $>a(X,L)$.  By assumption the latter case cannot happen, thus the Stein factorization of $f: Y \to X$ cannot have any branch divisors.  We conclude $f$ must be \'etale by purity of the branch locus as in \cite[Tag 0EA4]{stacks-project}.  Since $X$ is rationally connected by \cite[Proposition 2.5]{LTT18}, we conclude that the Stein factorization is trivial, i.e.~$f$ must be birational.

(2) Suppose for a contradiction that there was a non-birational generically finite morphism $f: Y \to X$ such that $a(Y,f^{*}L) \geq a(X,L)$.  Let $\pi: Y \dashrightarrow Z$ be the Iitaka fibration for $(K_{Y} + a(Y,f^{*}L)f^{*}L)$ and let $F$ denote the closure of a general fiber of $\pi$.  By \cite[Lemma 4.7 and Lemma 4.9]{LST22} we see that
\begin{equation*}
a(Y,f^{*}L) = a(F,f^{*}L|_{F}) \leq a(f(F),L|_{f(F)}).
\end{equation*}
By our assumption, the only possibility is that $a(Y,f^{*}L) = a(X,L)$ and that $f(F) = X$.  In other words, $\kappa(K_{Y} + a(Y,f^{*}L)f^{*}L) = 0$.  Then (1) implies that $f$ is birational.

It only remains to prove the final statement.  Assume $X$ is Fano and $L=-K_X$.  If $X$ has Picard rank $>1$, let $\pi: X \to Z$ be the contraction of an extremal ray of the Mori cone.  If $\pi$ is birational, then an irreducible component $Y$ of the exceptional locus is swept out by a non-dominant family of rational curves and thus satisfies $a(Y,-K_{X}) > 1$ by \cite[Theorem 1.1]{LT19}.  If $\pi$ is a Mori fibration, then a general fiber $F$ of $\pi$ will satisfy $a(F,-K_{X}) = 1$.  In either case $X$ carries a subvariety with Fujita invariant at least $a(X,-K_{X})$.
\end{proof}

\subsection{Dominant families of rational curves} 
In this section, we show that varieties with large Fujita invariant must contain dominant families of low-degree rational curves.  Before giving the statement, we need a preliminary result that uses  \cite{JLR25b} to bound the anticanonical degrees of dominant families of rational curves.

\begin{lemma} \label{lemm: uniruled minimal curve}
Let (X,B) be a projective log pair such that $X$ is uniruled.  Then there is a dominant family of rational curves on $X$ whose $-(K_X + B)$ degree is at most $\dim X + 1$.
\end{lemma}

\begin{proof}
We first apply \cite[Corollary 1.4.4]{BCHM} to find a birational morphism $\phi: X' \to X$ from a $\mathbb{Q}$-factorial normal projective variety $X'$ and effective $\mathbb{Q}$-divisors $\Gamma_{1},\Gamma_{2}$ on $X'$ such that
\begin{equation*}
    K_{X'} + \Gamma_{1} + \Gamma_{2} = \phi^{*}(K_{X} + B)
\end{equation*}
and the pair $(X',\Gamma_{1})$ is klt.  Since the difference $\phi^{*}(K_{X} + B) - (K_{X'} + \Gamma_{1})$ is effective the statement for $(X',\Gamma_{1})$ implies the statement for $(X,B)$.  So we may assume that $(X,B)$ is $\mathbb{Q}$-factorial klt.

If $K_{X}+B$ is pseudo-effective, then $-(K_{X}+B)$ has non-positive intersection against every moving rational curve on $X$ and the statement follows.  Otherwise, we can run a $(K_{X}+B)$-MMP with scaling $\psi: X \dashrightarrow \widetilde{X}$ to obtain a Mori fibration $\pi: \widetilde{X} \to Z$.  For a general fiber $F$ of $\pi$ the pair $(F,\psi_{*}B|_{F})$ is a klt Fano pair.  By taking general complete intersections of very ample divisors, we can find an irreducible curve $C$ through a general point $x \in F$ that is entirely contained in the smooth locus.  Since $(K_{F} + \psi_{*}B|_{F}) \cdot C < 0$, we also have $K_{F} \cdot C < 0$.  We apply \cite[Theorem 1.1]{JLR25b} with $H = -(K_{F}+\psi_{*}B|_{F})$ to see that there is a rational curve through $x$ satisfying $-(K_{F} + \psi_{*}B|_{F}) \cdot R \leq \dim F+1$.  Thus
\begin{equation*}
-(K_{\widetilde{X}} + \psi_{*}B) \cdot R = -(K_{F} + \psi_{*}B|_{F}) \cdot R \leq (\dim F + 1) \leq (\dim X + 1).
\end{equation*}
Since $x$ is a general point and $F$ is a general fiber, the deformations of $R$ cover all of $\widetilde{X}$.  

By \cite[Lemma 3.38]{KM98} if we pass to a smooth birational model $W$ with morphisms $g_{1}: W \to X$ and $g_{2}: W \to \widetilde{X}$ then the difference $g_{1}^{*}(K_{X} + B) - g_{2}^{*}(K_{\widetilde{X}} + \psi_{*}B)$ is an effective divisor.  Thus the strict transform of $R$ on $X$ satisfies the desired inequality.
\end{proof}

If $X$ is a smooth Fano variety with $-K_{X} \sim rH$ then by Mori's Bend-and-Break $X$ is covered by rational curves of $H$-degree $\leq \frac{\dim X+1}{r}$.  The following result gives us a birational version of this observation.

\begin{theorem} \label{theo:ainvandcoveringrationals}
Let $X$ be a uniruled projective variety and let $L$ be a big and nef $\mathbb{Q}$-Cartier divisor on $X$.  Then $X$ is dominated by rational curves $C$ satisfying $L \cdot C \leq (d+1)/a(X,L)$ where $d$ is the dimension of a general fiber of the rational map $\pi: X \dashrightarrow Z$ to the canonical model of $K_{X} + a(X,L)L$.
\end{theorem}

\begin{proof}
By pulling back $L$ to a resolution, it suffices to consider the case when $X$ is smooth.  Let $\pi: X \dashrightarrow Z$ be the Iitaka fibration for $K_{X} + a(X,L)L$.  By again pulling back $L$ to a birational model, we may suppose that $\pi$ is a morphism.  It suffices to prove that a general fiber of $\pi$ is covered by rational curves satisfying the desired bound.

A general fiber $F$ of $\pi$ has the property that $\kappa(K_{F} + a(X,L)L|_{F}) = 0$.  Thus there is a birational map $\psi: F \dashrightarrow F'$ to a terminal variety $F'$ such that $K_{F'} + a(X,L)\psi_{*}L|_{F} \equiv 0$.  Lemma \ref{lemm: uniruled minimal curve} implies that there is a dominant family of rational curves $C'$ on $F'$ satisfying $\psi_{*}L|_{F} \cdot C' \leq (\dim F+1)/a(X,L)$.  Letting $C$ denote the strict transform on $F$ of the general curve $C'$,  the negativity of contraction lemma (\cite[Lemma 3.39]{KM98}) shows that
\begin{equation*}
L|_{F} \cdot C \leq \psi^{*}\psi_{*}L|_{F} \cdot C = \psi_{*}L|_{F} \cdot C' \leq (\dim F+1)/a(X,L).
\end{equation*}
\end{proof}

\subsection{Dominant maps}
The final result in this section clarifies the possible geometric behavior of dominant morphisms $f: Y \to X$ which preserve the Fujita invariant.  It partially answers a question asked to us by Koll\'ar.

\begin{theorem}
Let $X$ be a smooth Fano variety.  Suppose that $f: Y \to X$ is a dominant generically finite morphism from a smooth projective variety $Y$ such that $a(Y,-f^{*}K_{X}) = 1$ and $\kappa(K_{Y}-f^{*}K_{X}) = 0$.  Then either:
\begin{enumerate}
\item there is an irreducible component $M \subset \mathcal{M}_{g,0}(X)$ such that the finite part $\psi: Z \to X$ of the Stein factorization of the evaluation morphism for the normalization of a compactification of the universal family over $M$ has degree $\geq 2$, and $f$ factors rationally through $\psi$, or

\item there is a morphism $f': Y' \to X'$ that is birationally equivalent to $f$ such that $X'$ has log Fano type and $f'$ is \'etale away from a codimension $2$ subset of $X'$.
\end{enumerate}
\end{theorem}

\begin{proof}
Since $\kappa(K_{Y}-f^{*}K_{X}) = 0$, Proposition \ref{prop:orthogonaltoadjoint} yields a $1$-free curve $s: C \to Y$ such that $(K_{Y} - f^{*}K_{X}) \cdot C = 0$. By Lemma \ref{lem: producing irreducible fibers}, we may assume $s$ belongs to a component $N \subset \mathcal{M}_{g,0}(Y,\alpha)$ such that the evaluation map of the normalization of the universal family over $N$, $\mathrm{ev}_N : U_N \to Y$, has irreducible fibers over general points in $Y$.  

First suppose that for $s$ general $f(s(C))$ avoids the branch divisor $B$.  Since deformations of $f(s(C))$ connect two general points of $X$, we see that $\kappa(B) = 0$.  Since $X$ is Fano, there is a birational map $\phi: X \dashrightarrow X'$ to a normal projective variety $X'$ of log Fano type that contracts $B$.  Taking the Stein factorization of a resolution of $Y \dashrightarrow X'$, we obtain a finite morphism of normal varieties $f': Y' \to X'$ that is birationally equivalent to $f$.  Since the branch locus for $f'$ contains no divisors, by purity of the branch locus as in \cite[Tag 0EA4]{stacks-project} we see that $f'$ must be \'etale over the smooth locus of $X'$.

Next suppose that for $s$ general $f(s(C))$ intersects the branch divisor $B$.  Let $M \subset \mathcal{M}_{g,0}(X, f_*\alpha)$ be the component dominated by $N$ under the natural map $f_* : \mathcal{M}_{g,0}(Y, \alpha) \to \mathcal{M}_{g,0}(X,f_*\alpha)$.  Consider the normalization $U_{M}$ of the universal family over $M$.  We claim the evaluation map $\mathrm{ev}_M : U_M \to X$ has reducible fiber over a general point $p \in X$. If instead the fiber over $p$ were irreducible, then for any point $q \in f^{-1}(p)$ composition with $f$ would induce a birational map $\mathrm{ev}_N^{-1}(q) \dashrightarrow \mathrm{ev}_M^{-1}(p)$.  Letting $\pi: U_{M} \to M$ be the natural map, it follows that for a general map $h \in \pi(\mathrm{ev}_M^{-1}(p))$ and for any point $q \in f^{-1}(p)$ there is a lifting $h_q : C \to Y$ of $h$ such that $q \in h_q(C)$.  Note that $\alpha$ pairs to $0$ with $K_Y - f^*K_X$, and so a general map in $N$ is disjoint from the ramification locus of $f : Y \to X$.  Thus the generality of $h$ ensures $f : Y \to X$ is finite over $h(C)$.  Hence, for any point $y \in f^{-1}(h(C))$, there exists some $q \in f^{-1}(p)$ such that $y \in h_q(C)$.  
This cannot be, as by hypothesis $f^{-1}(h(C))$ meets the ramification divisor $R \subset Y$ of $f$, but for all $q \in f^{-1}(p)$ we have $R \cdot h_{q*}C = (K_Y - f^*K_X) \cdot h_{q*}C = 0$.  Since $U_{M}$ is normal, we conclude that the finite part $g: Z \to X$ of the Stein factorization of the extension of the evaluation morphism to a normal compactification of $U_{M}$ has degree $\geq 2$.  By construction $f$ factors rationally through $g$.
\end{proof}

\section{Curves on general hypersurfaces}

The goal of this section is to study general hypersurfaces in $\PP^n$ of degree $m \leq n-2$ and show that they have no subvarieties of $a$-invariant at least 1. We start by considering the case where $m = n-2$. Let $H$ be the hyperplane class. We begin with a lemma about the curves on a subvariety $Y$ with $a$-invariant at least 1. 

\begin{lemma} \label{lem: Eric J section 7 Lemma 1}
    Let $X$ be a general hypersurface in $\PP^n$ of degree $n-2$. If $Y \subset X$ is any subvariety with $a(Y, -K_X) \geq 1$, then $a(Y, -K_X) = 1$ and there is a degree $d$ such that $\dim \overline{M}_{0,0}(Y,e) = 3d + \dim Y - 3$.
\end{lemma}
\begin{proof}
    By Theorem \ref{theo:introainv} applied to a resolution of $Y$, there is a positive integer $G$ such that for arbitrarily large curve classes $\beta$, we can find a dominant family of maps $M_{mor} \subset \Mor(\PP^1, Y, \beta)$ of dimension $-K_X \cdot \beta + \dim Y - G$. Select integers $r, k, e$ with $r,k > G$ and $e \geq (\dim Y + 1)r + (\dim Y )k$ such that there is a family in $\Mor(\PP^1, Y, e)$ of dimension $3e + \dim Y - G$. Let $Z \subset Y$ be the union of the loci over which the evaluation maps fail to be flat as we take all of the components of rational curves in $Y$ of degree at most $e$.  To derive a contradiction, suppose that through every point of $Y \setminus Z$, there is at most a $3d-3$ family of degree $d$ rational curves.
 
    Let $M_{konts} \subset \overline{\mathcal{M}}_{0,r+k}(Y, e)$ be the closure of the image of the map $M_{mor} \times \mathcal{M}_{0,r+k} \to \overline{\mathcal{M}}_{0,r+k}(Y, e)$.  We study the locus $M_{r,k} \subset M_{konts}$ of tuples $(f,C,q_1, \dots, q_{r+k})$ satisfying the open condition that $f(q_i) \notin Z$ for all $i$ and the closed condition that $q_1, \dots, q_r$ are all on distinct components of $C$ that are contracted by the stabilizing the domain curve.  
    
    We claim there is an irreducible component $M_{r,k}' \subset M_{r,k}$ that dominates $\overline{\mathcal{M}}_{0,r+k}$ under the forgetful map and generically parameterizes maps that send the $q_i$ to distinct points in $Y\setminus Z$.  Indeed, let $s : \PP^1 \to Y$ be a general map parameterized by $M_{mor}$ and select general points $q_1', \dots, q_{r+k}' \in \PP^1$.  Since $M_{mor}$ parameterizes a dominant family, $s(q_i') \notin Z$ for all $i$.  The space of maps fixing the images of $q_1', \ldots , q_{r+k}'$ has dimension at least $r$.  Hence, by \cite[Lemma~2.1]{JLR25b} $M_{r+k}$ is non-empty and contains at least one stable map sending the $q_{i}$ to distinct points.  Since we can find such a map for any general points $q_{1}',\ldots,q_{r+k}'$, this proves the existence of $M_{r,k}'$.

    Because $M_{r,k}$ is defined as the locus where we break off $r$ rational components, each irreducible component of $M_{r,k}$ has codimension at most $r$ in $M_{konts}$, i.e., dimension at least $3e+\dim Y + k - 3 - G$.  In fact, since $M_{r,k}'$ dominates $\overline{\mathcal{M}}_{0,r+k}$, the family $M_{r,k}'' \subset \overline{\mathcal{M}}_{0,0}(Y,e)$ underlying $M_{r,k}'$ has dimension at least $3e+\dim Y - r - G$.  Below, we will perform a careful dimension count of $M_{r,k}''$ to obtain a contradiction. 
    
    Let $N$ be the number of irreducible components of the domain of a general map parameterized by $M''_{r,k}$. Since every rational curve on $Y$ is also contained in $X$, by applying \cite[Proposition 6.5]{RY19} to $X$ we see the dimension of $M''_{r,k}$ is at most $3e+\dim Y - 2 - N$.  Thus $N$ is at most $G + r - 2 < k + r$. 
    
    Let $(f,C, q_1, \dots, q_{r+k})$ be a general element of $M'_{r,k}$. Then for some $i$ there is a $q_i$ on a component $C_0$ of $C$ that is not contracted by the stabilization map.  In particular, since $q_i$ does not map into $Z$, a general point of $C_0$ will not map into $Z$. 

    Let $e_0$ be the degree of $C_0$ and $e_1, \dots, e_{N-1}$ be the degrees of the other irreducible components of $C$. We do a dimension-counting argument analogous to \cite[Proposition 6.5]{RY19}. Namely, $C_0$ is chosen from a family of curves of dimension at most $3e_0 + \dim Y -3$. We then build up our curve component by component.  Since $(f,C, q_1, \dots, q_{r+k})$ was chosen generally and $M_{r,k}'$ dominates $\overline{\mathcal{M}}_{0,r+k}$, there are $r$ irreducible components of $C$ attached to $C_0$ at $q_1, \ldots, q_r$.  To count the moduli of these components we break our argument into two cases.  First, if the image of the point where $C_0$ meets $C_i$ is in $Z$, there is by \cite[Theorem 3.3]{RY19} a $3e_i-2$ dimensional family of choices for $C_i$. Otherwise, there is a 1-dimensional family of choices of attachment point on $C_0$ plus a $3e_i - 3$ dimensional family of choices for $C_i$ passing through that point. In either case, we have $3e_i - 2$ choices for adding these curves. 
    
    For adding other irreducible components of $C$, we have a $3e_i-1$-dimensional family of choices as in \cite[Theorem 3.3]{RY19}. Thus, the total dimension of $M''_{r,k}$ is at most $3 \sum_i e_i + \dim Y - 3 - (N -1) - r \leq 3e + \dim Y - 2 - 2r$. Since $r > G$, we obtain a contradiction.
\end{proof}

\begin{lemma} \label{lem: Eric J section 7 Lemma 2}
    Let $X$ be a general hypersurface of degree $n-2$ in $\PP^n$.  Suppose $Y \subset X$ is a subvariety such that $a(Y, -K_X) = 1$.  Let $c = \codim (Y,X)$.  Then there exists a positive integer $e \neq 3$ with $\frac{n-c+2}{3} < e \leq 2 \frac{n-c+2}{3}$ and a component $M_e \subset \Mor(\PP^1,Y,e)$ such that
\begin{enumerate}
    \item $\dim M_e = 3e + \dim Y$;
    \item $M_e$ parameterizes a dominant family of curves on $Y$;
    \item a general map
    parameterized by $M_e$ is generically injective;
\end{enumerate}
\end{lemma}

\begin{proof}
        Since $X$ has degree $n-2$, it follows that $a(X,H) = 3$, so we have $a(Y,H) = 3$ as well.  Observe that $X$ contains no 2-planes by \cite[Theorem 6.28]{3264}. Since the only subvarieties of $\PP^n$ of dimension at most 2 with $a$-invariant 3 are 2-planes, $Y$ cannot have dimension 1 or 2.  Hence $\dim Y = n-1-c \geq 3$.

        By Lemma \ref{lem: Eric J section 7 Lemma 1} and \cite[Proposition 6.5]{RY19}, there is an integer $d$ such that through a general point of $Y$ there is a $3d-2$ dimensional family of irreducible rational curves of degree $d$. We claim that we can find such a family with $d \leq \frac{n-c+2}{3}$.  Our argument uses descending induction on $d$ via a threshold degree argument originally in \cite{HRS04}. If we have such a family with $ d > \frac{n-c+2}{3}$, we see that we have a family of curves through a general point of $Y$ of dimension at least $n-c+1$. Hence, we have a 1-parameter family of rational curves through two general points of $Y$, which means we can break into several components, at least one of which passes through a general point of $Y$. Let $d' < d$ be the degree of the curve passing through the general point of $Y$. By a dimension count similar to the proof of Lemma \ref{lem: Eric J section 7 Lemma 1} (modeled after \cite[Proposition 6.5]{RY19}), it follows that there must be a $3d'-2$ dimensional family of such curves. The result follows by descending induction on $d$.

    Observe that for the minimal such $d$ with a $3d-2$-dimensional family of curves through a general point, the general curve in the family will be a generically injective map.

    By gluing together and smoothing curves of degree at most $\frac{n-c+2}{3}$, we see that we can find a $3e-2$-dimensional family of curves of degree $e$ through a general point of $Y$, where $e$ is strictly larger than $\frac{n-c+2}{3}$ and at most $2 \frac{n-c+2}{3}$. In fact, we can even ensure that $e \neq 3$, since if $\frac{n-c+2}{3} < 3$ then we must have a curve of degree $1$ or $2$ and by gluing and smoothing such curves we can obtain a suitable curve of degree 4. Since $n-1-c \geq 3$, it follows that $4 \leq 2 \frac{n-c+2}{3}$ as required.
\end{proof}

\begin{proposition}\label{prop:behaviorOfCurvesOnGeneralHyps}
Let $X$ be a general hypersurface of degree $n-2$ in $\PP^n$ with $n \geq 4$. Let $e \leq n$ with $e \neq 3$. Let $M'$ be the open locus in $\Mor(\PP^1,X,e)$ consisting of generically injective maps. Then outside of a set of codimension at least $n-e+1$, $M'$ parameterizes rational normal curves. 
\end{proposition}

\begin{proof}
If $e = 1, 2$, the statement is trivial, so we may assume $e > 3$.  We consider the incidence correspondence $I \subset \Mor(\PP^1, \PP^n) \times \PP^{N-1}$ of pairs $(f,X)$ where $f$ is a generically injective map sending $\PP^1$ to a degree $e$ rational curve in $\PP^n$ and $X$ is a degree $n-2$ hypersurface containing $f(\PP^1)$. By \cite[Theorem in introduction]{GrusonLazarPeskine}, nonplanar curves of degree $e$ have the property that for every $q \geq e-1$ the degree $q$ hypersurfaces in $\PP^n$ trace out a non-special complete linear system on the curve. Since $n \geq e$, the fibers of the map $I \to \Mor(\PP^1, \PP^n)$ have dimension $\binom{2n-2}{n} - e(n-2)-2+g$,  where $g$ is the arithmetic genus of $C$.

We work out the dimension of the algebraic subset $I_{\delta}$ of $I$ corresponding to curves spanning a $\PP^{e-\delta}$. Observe that every irreducible component of $I$ will have dimension at least $E = 3e+n-1 + \binom{2n-2}{n} - 1$, since $I$ is locally cut out by $e(n-2)+1$ equations in the incidence correspondence.

First we study $I_{e-2}$. There is a $3(n-2)$ dimensional choice of 2-planes in $\PP^n$. For each 2-plane, there is a $3e+2$ dimensional family of degree $e$ morphisms from rational curves to it. For each curve in that $\PP^2$, there is a $\binom{n-e}{2}-1$ dimensional family of degree $n-2$ hypersurfaces in $\PP^2$ containing the curve. For each such hypersurface in $\PP^2$, there is a $\binom{2n-2}{n} - \binom{n}{2}$ dimensional family of hypersurfaces in $\PP^n$ with the given restriction to $\PP^2$. Thus, the dimension of $I_{e-2}$ is 
\[ 3(n-2)+3e+2+\binom{n-e}{2}-1+\binom{2n-2}{n} - \binom{n}{2} \]
\[ = E - \left( e \left(n-\frac{e+1}{2} \right) -2n+3 \right) \]
To complete our study of $I_{e-2}$, it remains to show that the codimension $ e( n- \frac{e+1}{2} )-2n+3 $ is at least $n-e+1$ when $n \geq e > 3$. The difference between these functions is $\frac{-1}{2}e^{2} + (n+\frac{1}{2})e - 3n + 2$.  For $e=4$, this simplifies to $n-4$ which is non-negative since $n\geq 4$.  Since the difference is increasing in $e$ on the interval $4 \leq e \leq n$, we conclude that it is non-negative in this entire range.

Now we study $I_{\delta}$ for $\delta < e-2$. Observe that the general rational curve spanning a $\PP^{e-\delta}$ has arithmetic genus 0. The space of degree $e$ maps from $\PP^1$ to $\PP^n$ has dimension $(e+1)(n+1) - 1$, while the space of degree $e$ maps from $\PP^1$ to $\PP^n$ whose image lies in an $(e-\delta)$-plane has dimension
\begin{align*}
    (e-\delta+1)(n-(e-\delta))+(e+1)(e-\delta+1)-1 & = (e-\delta+1)(n+\delta+1) - 1 \\
    & = ((n+1)(e+1)-1)-\delta (n-e+\delta) .
\end{align*} 
Thus, the codimension of the space of curves spanning a $\PP^{e-\delta}$ is $\delta(n-e+\delta)$. The fiber over a curve $f$ of the map from $I_{\delta}$ to $\Mor(\PP^1, \PP^n)$ will jump by the arithmetic genus $p_a(f(\PP^1))$ as compared to the fiber over a general curve. Since the generic rational curve spanning $\mathbb{P}^{e-\delta}$ will have arithmetic genus $0$, $I_\delta$ will have dimension at most $E - \delta(n-e+\delta) + g-1$, where $g \geq 1$ is the maximum genus of a curve parameterized by $I_{\delta}$.

We next claim that $g \leq \delta^2$. 
By the Castelnuovo bound (as in \cite[Page 40]{Harris81}), we see that $g$ is at most ${q \choose 2}(e-\delta-1) + qm$ where $q$ and $m$ are the quotient and the remainder of dividing $e-1$ by $e-\delta-1$. Observe that if $2\delta < e-1$, we have that $q \leq 1$, and the relationship is trivially satisfied. Similarly, if $2\delta = e-1$, we have $q=2$, $m=0$ and the relationship is again satisfied because $e - \delta - 1 = \delta$ is less than $\delta^2$.  Thus, we may assume $2\delta > e-1$, i.e., $2\delta \geq e$. We have
\[ {q \choose 2}(e-\delta-1) + qm \leq \frac{q}{2} \left( q (e-\delta-1) + 2m \right) = \frac{q}{2} \left( e-1+m \right) .\]
Using the fact that $m < e-\delta-1$ and $q \leq \frac{e-1}{e-\delta-1}$, this last quantity is strictly less than
\[ \frac{e-1}{2(e-\delta-1)} (e-1+e-\delta-1) = \frac{(e-1)^2}{2(e-\delta-1)} + \frac{e-1}{2}  .\]
Using the fact that $2\delta \geq e$ and $e-\delta \geq 3$, this is at most
\[ \frac{(2\delta-1)^2}{4} + \frac{2\delta-1}{2} = \frac{4\delta^2-4\delta+1 + 4\delta - 2}{4} = \frac{4\delta^2-1}{4} < \delta^2 . \]
This shows that $g \leq \delta^2$.

Using the fact that $I_\delta$ has dimension at most $E-\delta(n-e) - \delta^2 + g - 1$, we see that the dimension of $I_\delta$ is at most $E - \delta(n-e)-1 \leq E-(n-e+1)$.

Because the fibers of the natural map $I \to \Mor(\mathbb{P}^{1}, \mathbb{P}^{n})$ are projective spaces, $I$ can only be reducible if there is some locus of curves in one of the $I_{\delta}$ where the fibers have dimension so large that it forms a component of dimension at least $E$. The above analysis shows that this is impossible. Thus, $I$ is irreducible, and the codimension of every $I_\delta$ is at least $n-e+1$, as required.
\end{proof}

\begin{lemma} \label{lem: Eric J section 7 Lemma 3}
    Let $X$ be a general hypersurface of degree $n-2$ in $\mathbb{P}^{n}$.  Suppose $Y \subset X$ is a subvariety such that $a(Y,-K_{X}) = 1$ and let $c$ denote the codimension of $Y$.  Let $C$ be a general curve parameterized by the moduli space $M_{e}$ of Lemma \ref{lem: Eric J section 7 Lemma 2}.  Then the normal bundle $N_{C/X}$ has a quotient $Q$ of rank $c$ and degree at most 0.
\end{lemma}
\begin{proof}
    Let $f:\PP^1 \to \tilde{Y}$ be the strict transform of $C$ on a resolution $\tilde{Y}$ of $Y$. Then the natural map $N_{f/\tilde{Y}} \to N_{f/X}$ shows that $N_{f/X}$ has a subbundle of degree at least $\deg N_{f/ \tilde{Y}}$. However, $\deg N_{f/ \tilde{Y}} = h^0(N_{f/ \tilde{Y}}(-1)) = 3e-2$. Thus, the quotient of $N_{C/X}$ by the saturation of the image of $N_{f/ \tilde{Y}}$ has non-positive degree.
\end{proof}

\begin{lemma} \label{lemma:hypersurfaceNormalBundleNoQuotients}
Let $n \geq 5$ and $X \subset \PP^n$ be a general hypersurface of degree $n-2$.  For any integers $c,e$ with $1 \leq c \leq n-3$ and $\frac{n-c+2}{3} \leq e \leq 2 \frac{n-c+2}{3}$, let $M_{e,c} \subset \Mor(\PP^1, X, e)$ be the locus of maps parameterizing rational normal curves $C \subset X$ for which $N_{C/X}$ admits a rank $c$ quotient $Q$ with $c_1(Q) \leq 0$.  Then the codimension of $M_{e,c}$ in $\Mor(\PP^1, X, e)$ is at least $c+1$.
\end{lemma}

\begin{proof}
We start by considering curves having $N_{C/X}$ of some fixed splitting type. Fix a vector bundle $V$ on $\PP^1$ that has a rank $c$ quotient of non-positive degree and let $z$ be the number of $\OO(e+2)$ factors in $V$.  By \cite[Corollary~1.4]{Mioranci}, the codimension in $\Mor(\PP^1, X, e)$ of the space of rational normal curves $C \subset X$ with $N_{C/X} \cong V$ is
$$h^1(C, \mathcal{E}nd(V)) - (n-e)z.$$
We will estimate $h^1(C, \mathcal{E}nd(V))$ to prove our claim.  

Write $V = \oplus_{i = 1}^{n-2}\OO_{\PP^1}(a_i)$ with $a_i \geq a_{i+1}$ for all $i$. Because there exists some rank $c$ quotient $Q$ with non-positive degree, we know that $Q_0 = \OO_{\PP^1}(a_{n-c-1})\oplus \ldots \oplus \OO_{\PP^1}(a_{n-2})$ has non-positive degree.

First, assume that $z \geq 1$. Set $K' = \OO_{\PP^1}(e+2)^{\oplus z}$, set $Q' = \oplus_{i = z+1}^{n-2}\OO_{\PP^1}(a_i)$, and observe that $V \cong K' \oplus Q'$. Then
$$h^1(C, \mathcal{E}nd(V)) \geq h^1(C,\mathcal{H}om(K', Q')) = \sum_{i = z + 1}^{n-2} z(e+2 - a_i - 1).$$
Moreover $\sum_{i = z+1}^{n-2} a_i = 3e - 2 - z(e+2)$ since $c_1(V) = 3e - 2$.  Thus, 
\begin{align*}
    h^1(C, \mathcal{E}nd(V)) -(n-e)z &\geq z(n-2-z)(e+1) - z(3e - 2 -z(e + 2)) -(n-e)z\\
    &=z(z + e(n-4)) \geq 1 + e(n-4).
\end{align*} 
Since $e \geq 2$ and $n \geq 5$ this last expression is at least $2n-7$.  As $n \geq 5$ and $c \leq n - 3$, we have $2n-7 \geq c+1$ showing the desired lower bound on the codimension.

Suppose instead that $z= 0$. Let $K$ be $\OO(a_1) \oplus \dots \oplus \OO(a_{n-c-2}) $.  Observe that by Riemann Roch
$$h^1(C, \mathcal{E}nd(V)) \geq h^1(C, \mathcal{H}om(K, Q_0)) \geq c(3e - 2) - c(n-2-c).$$
Since $(n-e)z = 0$, this is a lower bound for the codimension of $M_{e,c}$ in $\Mor(\PP^1, X, e)$.  As the coefficient of $e$ is positive, we may substitute $e = \frac{n - c + 2}{3}$ into the preceding expression to get a lower bound of
$c(n-c) - c(n-2-c) = 2c \geq c +1$
which proves our claim.
\end{proof}

Now we are ready to prove the main theorem of this section.

\begin{proof}[Proof of Theorem \ref{theorem:hypersurfaces_a_invariant}]
   When $n \leq 4$, $X$ is a hyperplane or a quadric threefold and the statement is well-known.  Thus we may assume $n \geq 5$.  We must show that $X$ has no subvarieties $Y$ with $a(Y,-K_{X}|_{Y}) \geq 1$.  Equivalently, letting $H$ denote the hyperplane class and $m$ denote the degree of $X$, we must show that there is no subvariety $Y$ such that $a(Y,H|_{Y}) \geq n+1-m$.  By \cite[Theorem 3.15]{LT19b}, it suffices to prove the result after cutting down $X$ and $Y$ by hyperplanes so that either $Y$ is a curve or $m=n-2$.  When $Y$ is a curve, then $a(Y,H|_{Y}) \leq 2$ so the desired inequality holds.  
   
   Now suppose $m=n-2$. In this case, observe that a general such hypersurface contains no 2-planes by \cite[Theorem 6.28]{3264}. To get a contradiction, suppose $Y$ is a subvariety with $a(Y,-K_{X}|_{Y}) \geq 1$. Let $M_e$ be the sublocus of $\Mor(\mathbb{P}^{1},Y)$ constructed by Lemma \ref{lem: Eric J section 7 Lemma 2} and let $c$ be the codimension of $Y$ in $X$. According to Lemma \ref{lem: Eric J section 7 Lemma 2}(1), $c$ is also the codimension of $M_{e}$ in the space of maps of degree $e$ rational curves on $X$.  We see by Proposition \ref{prop:behaviorOfCurvesOnGeneralHyps} that the codimension of the locus of degree $e$ curves in $X$ that are not rational normal curves is at least $n-e+1$. We claim $n-e+1$ is at least $c+1$, which shows that a general curve of $M_e$ is a rational normal curve. Using the upper bound on $e$, we see that $n-e+1$ is at least
    \[ n- 2 \frac{n-c+2}{3} + 1 = c-2 + (n-c+2) - 2 \frac{n-c+2}{3} + 1 = c-2+ \frac{n-c+2}{3} + 1 . \]
    This is greater than $c$ because $n -c \geq 2$, i.e.\ $Y$ has positive dimension.
    
    By Lemma \ref{lem: Eric J section 7 Lemma 3}, we see that a general curve in $M_e$ has $N_{C/X} = E \oplus Q$, with $Q$ having rank $c$ and $c_1(Q) \leq 0$. However, by Lemma \ref{lemma:hypersurfaceNormalBundleNoQuotients}, this is impossible.
\end{proof}

\begin{proof}[Proof of Theorem \ref{theo:introhypcomparison}:]
Combine Theorem \ref{theorem:hypersurfaces_a_invariant} and Theorem \ref{thm:otherequivalentaconditions}.
\end{proof}

\begin{proof}[Proof of Corollary \ref{coro:besthypersurfaces}:]
Combine Theorem \ref{theorem:hypersurfaces_a_invariant} and Theorem \ref{theo:strongestainvcurves}.
\end{proof}

\bibliographystyle{alpha}
\bibliography{exctowers}

\def\cprime{$'$}
\begin{thebibliography}{BCHM10}

\bibitem[AC14]{AC14}
Dan Abramovich and Qile Chen.
\newblock Stable logarithmic maps to {D}eligne-{F}altings pairs {II}.
\newblock {\em Asian J. Math.}, 18(3):465--488, 2014.

\bibitem[BCHM10]{BCHM}
Caucher Birkar, Paolo Cascini, Christopher. Hacon, and James McKernan.
\newblock Existence of minimal models for varieties of log general type.
\newblock {\em J. Amer. Math. Soc.}, 23(2):405--468, 2010.

\bibitem[BDPP13]{BDPP13}
S{\'e}bastian Boucksom, Jean-Pierre Demailly, Mihai Pa{\v u}n, and Thomas Peternell.
\newblock The pseudo-effective cone of a compact {K}\"ahler manifold and varieties of negative {K}odaira dimension.
\newblock {\em J. Algebraic Geom.}, 22(2):201--248, 2013.

\bibitem[BZ16]{BZ16}
Caucher Birkar and De-Qi Zhang.
\newblock Effectivity of {I}itaka fibrations and pluricanonical systems of polarized pairs.
\newblock {\em Publ. Math. Inst. Hautes \'Etudes Sci.}, 123:283--331, 2016.

\bibitem[Che14]{Chen14}
Qile Chen.
\newblock Stable logarithmic maps to {D}eligne-{F}altings pairs {I}.
\newblock {\em Ann. of Math. (2)}, 180(2):455--521, 2014.

\bibitem[DC14]{Dicerbo14}
Gabriele Di~Cerbo.
\newblock Uniform bounds for the {I}itaka fibration.
\newblock {\em Ann. Sc. Norm. Super. Pisa Cl. Sci. (5)}, 13(4):1133--1143, 2014.

\bibitem[dFH11]{dFH11}
Tommaso de~Fernex and Christopher~D. Hacon.
\newblock Deformations of canonical pairs and {F}ano varieties.
\newblock {\em J. Reine Angew. Math.}, 651:97--126, 2011.

\bibitem[EH16]{3264}
David Eisenbud and Joe Harris.
\newblock {\em 3264 and all that---a second course in algebraic geometry}.
\newblock Cambridge University Press, Cambridge, 2016.

\bibitem[Fuj89]{Fujita89}
Takao Fujita.
\newblock Remarks on quasi-polarized varieties.
\newblock {\em Nagoya Math. J.}, 115:105--123, 1989.

\bibitem[GLP83]{GrusonLazarPeskine}
Laurent Gruson, Robert Lazarsfeld, and Christian Peskine.
\newblock On a theorem of {C}astelnuovo, and the equations defining space curves.
\newblock {\em Invent. Math.}, 72(3):491--506, 1983.

\bibitem[Gro65]{EGAIVII}
Alexander Grothendieck.
\newblock {\'E}l{\'e}ments de g{\'e}om{\'e}trie alg{\'e}brique. {IV}. { \'E}tude locale des sch{\'e}mas et des morphismes de sch{\'e}mas. {II}.
\newblock {\em Inst. Hautes {\'E}tudes Sci. Publ. Math.}, (24):231, 1965.

\bibitem[GS13]{GS13}
Mark Gross and Bernd Siebert.
\newblock Logarithmic {G}romov-{W}itten invariants.
\newblock {\em J. Amer. Math. Soc.}, 26(2):451--510, 2013.

\bibitem[Har81]{Harris81}
Joe Harris.
\newblock A bound on the geometric genus of projective varieties.
\newblock {\em Ann. Scuola Norm. Sup. Pisa Cl. Sci. (4)}, 8(1):35--68, 1981.

\bibitem[HL20]{HL20}
Jingjun Han and Zhan Li.
\newblock On {F}ujita's conjecture for pseudo-effective thresholds.
\newblock {\em Math. Res. Lett.}, 27(2):377--396, 2020.

\bibitem[HL25]{HaseLiu25}
Matthew Hase-Liu.
\newblock A converse to geometric {M}anin's conjecture for general low degree hypersurfaces, 2025.
\newblock https://arxiv.org/abs/2501.12506.

\bibitem[HM98]{HM98}
Joe Harris and Ian Morrison.
\newblock {\em Moduli of curves}, volume 187 of {\em Graduate Texts in Mathematics}.
\newblock Springer-Verlag, New York, 1998.

\bibitem[H{\"o}r10]{Horing10}
Andreas H{\"o}ring.
\newblock The sectional genus of quasi-polarised varieties.
\newblock {\em Arch. Math. (Basel)}, 95(2):125--133, 2010.

\bibitem[HRS04]{HRS04}
Joe Harris, Mike Roth, and Jason Starr.
\newblock Rational curves on hypersurfaces of low degree.
\newblock {\em J. Reine Angew. Math.}, 571:73--106, 2004.

\bibitem[HTT15]{HTT15}
Brendan Hassett, Sho Tanimoto, and Yuri Tschinkel.
\newblock Balanced line bundles and equivariant compactifications of homogeneous spaces.
\newblock {\em Int. Math. Res. Not. IMRN}, (15):6375--6410, 2015.

\bibitem[JLR25]{JLR25}
Eric Jovinelly, Brian Lehmann, and Eric Riedl.
\newblock Free curves and fundamental groups, 2025.
\newblock https://arxiv.org/abs/2510.27031.

\bibitem[JLR26]{JLR25b}
Eric Jovinelly, Brian Lehmann, and Eric Riedl.
\newblock Optimal bounds in {B}end-and-{B}reak.
\newblock {\em Forum Math. Pi}, 14:Paper No. e16, 2026.

\bibitem[KM98]{KM98}
J\'anos Koll\'ar and Shigefumi Mori.
\newblock {\em Birational geometry of algebraic varieties}, volume 134 of {\em Cambridge Tracts in Mathematics}.
\newblock Cambridge University Press, Cambridge, 1998.
\newblock With the collaboration of C. H. Clemens and A. Corti, Translated from the 1998 Japanese original.

\bibitem[KM99]{KM99}
Se\'an Keel and James McKernan.
\newblock Rational curves on quasi-projective surfaces.
\newblock {\em Mem. Amer. Math. Soc.}, 140(669):viii+153, 1999.

\bibitem[KMM92]{KMM92b}
J\'anos Koll\'ar, Yoichi Miyaoka, and Shigefumi Mori.
\newblock Rational curves on {F}ano varieties.
\newblock In {\em Classification of irregular varieties ({T}rento, 1990)}, volume 1515 of {\em Lecture Notes in Math.}, pages 100--105. Springer, Berlin, 1992.

\bibitem[LRT25]{LRT23}
Brian Lehmann, Eric Riedl, and Sho Tanimoto.
\newblock Non-free curves on {F}ano varieties.
\newblock {\em Osaka J. Math.}, 62(1):19--50, 2025.

\bibitem[LRT26]{LRT24}
Brian Lehmann, Eric Riedl, and Sho Tanimoto.
\newblock Non-free sections of {F}ano fibrations.
\newblock {\em Mem. Amer. Math. Soc.}, 319(1626):v+109, 2026.

\bibitem[LST22]{LST22}
Brian Lehmann, Akash~Kumar Sengupta, and Sho Tanimoto.
\newblock Geometric consistency of {M}anin's conjecture.
\newblock {\em Compos. Math.}, 158(6):1375--1427, 2022.

\bibitem[LT19a]{LT19}
Brian Lehmann and Sho Tanimoto.
\newblock Geometric {M}anin's conjecture and rational curves.
\newblock {\em Compos. Math.}, 155(5):833--862, 2019.

\bibitem[LT19b]{LT19b}
Brian Lehmann and Sho Tanimoto.
\newblock On exceptional sets in {M}anin's conjecture.
\newblock {\em Res. Math. Sci.}, 6(1):Paper No. 12, 41, 2019.

\bibitem[LTT18]{LTT18}
Brian Lehmann, Sho Tanimoto, and Yuri Tschinkel.
\newblock Balanced line bundles on {F}ano varieties.
\newblock {\em J. Reine Angew. Math.}, 743:91--131, 2018.

\bibitem[Mio25]{Mioranci}
Lucas Mioranci.
\newblock Normal bundles of rational normal curves on hypersurfaces.
\newblock {\em Michigan Math. J.}, 75(3):639--671, 2025.

\bibitem[PRT24]{PRT24}
Anand Patel, Eric Riedl, and Dennis Tseng.
\newblock Moduli of linear slices of high degree smooth hypersurfaces.
\newblock {\em Algebra Number Theory}, 18(12):2133--2156, 2024.

\bibitem[RY19]{RY19}
Eric Riedl and David Yang.
\newblock Kontsevich spaces of rational curves on {F}ano hypersurfaces.
\newblock {\em J. Reine Angew. Math.}, 748:207--225, 2019.

\bibitem[Sen21]{Sengupta21}
Akash~Kumar Sengupta.
\newblock Manin's conjecture and the {F}ujita invariant of finite covers.
\newblock {\em Algebra Number Theory}, 15(8):2071--2087, 2021.

\bibitem[{Sta}25]{stacks-project}
The {Stacks project authors}.
\newblock The stacks project.
\newblock \url{https://stacks.math.columbia.edu}, 2025.

\bibitem[Tia12]{Tian12}
Zhiyu Tian.
\newblock Symplectic geometry of rationally connected threefolds.
\newblock {\em Duke Math. J.}, 161(5):803--843, 2012.

\end{thebibliography}
\end{document}